\documentclass
[a4paper,
11pt,
twoside
]{article}

\usepackage[T1]{fontenc}
\usepackage{ae}

\usepackage[left=1in, right=1in, top=1.5in, bottom=1in, headheight=15pt]{geometry}
\usepackage{authblk} 

\usepackage{amsmath}
\usepackage{amssymb}
\usepackage{amsthm}

 \usepackage{charter,courier} 
 \usepackage{chemarrow}
 \usepackage{caption} 
\usepackage[figuresright]{rotating} 

\usepackage{fancyhdr}
\usepackage{empheq}

\usepackage{latexsym}
\usepackage{lscape}

\usepackage{mleftright} 
\usepackage{makeidx}
\makeindex
\usepackage[nottoc]{tocbibind}

\usepackage[intoc]{nomencl}
\makenomenclature
\usepackage{stmaryrd}

\usepackage{textcomp}
\usepackage{tikz}

\usepackage{wasysym}

\usepackage{xcolor}
\usepackage{color}
\usepackage{colortbl}

\usepackage{yhmath}
\usepackage{graphicx}

\usepackage[all]{xy}

\definecolor{colComments}{rgb}{1,0,0}

\theoremstyle{plain}  

\swapnumbers 
\numberwithin{equation}{section} 
\newtheorem{Definition}[equation]{Definition} 
\newtheorem{Definition/Lemma}[equation]{Definition/Lemma}

\newtheorem{Lemma}[equation]{Lemma}
\newtheorem{Theorem}[equation]{Theorem}
\newtheorem{Proposition}[equation]{Proposition}
\newtheorem{Corollary}[equation]{Corollary}
\newtheorem{Remark}[equation]{Remark}

\newtheorem{Notation}[equation]{Notation}
\newtheorem{Notation/Lemma}[equation]{Notation/Lemma}

\newtheorem{Comparison}[equation]{Comparison}

\font\triangles=beta
\newcommand{\Squar}{\mbox{\triangles 3}}

\newcommand{\UR}{\mbox{\triangles 1}}

\newcommand{\UP}{\mbox{\triangles 2}}

\newcommand{\C}{\mbox{\triangles 9}}

\font\trianglesb=betab
\newcommand{\Squarb}{\mbox{\trianglesb 3}}
\newcommand{\URb}{\mbox{\trianglesb 1}}

\newcommand{\UPb}{\mbox{\trianglesb 2}}

\title{Supercharacter theories for Sylow $p$-subgroups
       \\of Lie type $G_2$}
\author{Yujiao Sun}
\affil{\footnotesize School of Mathematics and Statistics,
         Beijing Institute of Technology,\\
      Beijing 100081, P. R. China}
\affil{\footnotesize E-mail: yujiaosun@bit.edu.cn}

\date{}

\begin{document}
\maketitle

\begin{abstract}
We construct a supercharacter theory,
and establish the supercharacter table
for Sylow $p$-subgroups $G_2^{syl}(q)$
of the Chevalley groups $G_2(q)$ of Lie type $G_2$ when $p>2$.
Then we calculate the conjugacy classes,
determine the complex irreducible characters by Clifford theory,
and obtain the character tables for $G_2^{syl}(q)$ when $p>3$.
\end{abstract}


{\small \textbf{Keywords: }supercharacter theory; character table;
                           Sylow $p$-subgroup}

{\small \textbf{2010 Mathematics Subject Classification: }Primary 20C15, 20D15. Secondary 20C33, 20D20}

\section{Introduction}
Let $p$ be a fixed odd prime,
$\mathbb{N}^*$ the set of positive integers,
$q:=p^k$ for a fixed $k\in \mathbb{N}^*$,
$\mathbb{F}_q$ the finite field with $q$ elements
and $A_n(q)$ ($n\in \mathbb{N}^*$)
the group of upper unitriangular $n\times n$-matrices
with entries in $\mathbb{F}_q$.
Then $A_n(q)$ is a Sylow $p$-subgroup of the general linear group $GL_n(q)$
and also a Sylow $p$-subgroup \cite{Carter1}
of the Chevalley group of Lie type $A_{n-1}$ ($n \geq 2$) over $\mathbb{F}_q$.
It is well known that classifying the conjugacy classes of $A_n(q)$ and hence the complex irreducible characters
is a ``wild'' problem,
see e.g. \cite{Evs2011, ps, vla}.
Higman's conjecture \cite{hig} states that
for a fixed $n$, the number of conjugacy classes of $A_n(q)$
is determined by a polynomial in $q$ with integral coefficients depending on $n$.
Isaacs \cite{isa1} proved that
the degrees of complex irreducible characters
of $\mathbb{F}_q$-algebra groups are powers of $q$.
Lehrer \cite{leh}
and later Isaacs \cite{isa2}
refined Higman's conjecture.
Pak and Soffer \cite{ps} verified Higman's conjecture for $n\leq 16$.

Diaconis and Isaacs \cite{di}
introduced the notion of {\it supercharacter theory} (see \ref{supercharacter theory}) for an arbitrary finite group,
which is a coarser approximation of the character theory.
Roughly, a supercharacter theory replaces irreducible characters by supercharacters,
conjugacy classes by superclasses,
and irreducible modules by supermodules.
In such a way, a {\it supercharacter table} is constructed as a replacement for a character table.
Andr\'{e} in \cite{and1} using the Kirillov orbit method,
and Yan in \cite{yan2} using a more elementary method
determined the {\it Andr\'{e}-Yan supercharacter theory} for $A_n(q)$.
This theory is extended to the so-called algebra groups \cite{di}.
The supercharacter theory for $A_n(q)$ is based on
the observation that $u\mapsto u-1$ defines a bijection from $A_n(q)$
to an $\mathbb{F}_q$-vector space of nilpotent upper triangular matrices.
However, this does not work in general for Sylow $p$-subgroups
of the other Lie types.

 Andr\'{e} and Neto \cite{an1,an2, an3} studied the
{\it Andr\'{e}-Neto supercharacter theories}
for the classical finite unipotent groups of untwisted types $B_n$, $C_n$ and $D_n$
(i.e. the classical finite groups: the odd orthogonal groups, the symplectic groups
and the even positive orthogonal groups, respectively).
The construction of \cite{an1, an3} is extended to involutive algebra groups \cite{AFM2015}.
Andrews \cite{Andrews2015, Andrews2016} constructed supercharacter theories of finite unipotent
groups in the orthogonal, symplectic and unitary types
(i.e. the Sylow $p$-groups of untwisted Chevalley groups of types $B_n$ and $D_n$, of type $C_n$,
 and of the twisted Chevalley groups of type $^2A_n$, respectively).
Supercharacters of those classical groups arise as restrictions of supercharacters of overlying
full upper unitriangular groups $A_N(q)$ to the Sylow $p$-subgroups,
and superclasses arise as intersections of superclasses of $A_N(q)$ with these groups.

Jedlitschky  generalised Andr\'{e}-Yan's construction by a procedure called
{\it monomial linearisation}  (see \cite[\S 2.1]{Markus1}) for a finite group,
and decomposed {Andr\'{e}-Neto supercharacters} for Sylow $p$-subgroups
(i.e. the unipotent even positive orthogonal groups)
of Lie type $D$ into much smaller characters \cite{Markus1}.
The smaller characters are pairwise orthogonal,
and each irreducible character is a constituent of exactly one of the smaller characters.
Thus, these characters look like finer supercharacters for the Sylow $p$-subgroups of type $D$.
But, so far there are no corresponding finer superclasses.
A monomial linearisation for Sylow $p$-subgroups of Lie types $B_n$, $D_n$ and $C_n$
is exhibited,
and the stabilizers and orbit modules are studied
in
\cite{Guo2018On, Guo2017Orbit}.
One may ask, if there exists a construction of a supercharacter theory for
Sylow $p$-subgroups of all Lie types
based on the monomial linearisation approach for type $D$.

We try the exceptional types firstly,
apply Jedlitschky's monomial linearisation  to obtain supercharacters,
and then supplement it
to construct superclasses as well in order to obtain a full supercharacter theory.
This has been done for the Sylow $p$-subgroup ${{^3D}_4^{syl}}(q^3)$
of the twisted Lie type ${^3D}_4$ by the author in \cite{sun3D4super}.
It will be determined in this paper in the special case of Lie type $G_2$:
the Sylow $p$-subgroup $G_2^{syl}(q)$
of the Chevalley group $G_2(q)$.
The method seems to work for more exceptional Lie types,
indeed in the forthcoming paper we shall obtain similar results
for the case of twisted type $^2G_2$.
Thus we have some evidence that there is indeed a general supercharacter theory
for all Lie types behind this.

For the matrix Sylow $p$-subgroup $G_2^{syl}(q)$ (see Section \ref{sec:G2-1})
of the Chevalley group of type $G_2$,  the construction are followed.
\begin{itemize}
\setlength\itemsep{0em}
 \item [1.] {\it Determine a monomial module by constructing a monomial linearisation:}
            Determine a Sylow $p$-subgroup ${G}_2^{syl}(q) \leqslant {{^3{D}}_4^{syl}}(q^3)$,
            and construct an intermediate algebra group $G_8(q)$ such that $G_2^{syl}(q) \leqslant G_8(q) \leqslant A_8(q)$
            (see Section \ref{sec:G2-1}).
            Then construct a monomial linearisation for $G_8(q)$
            and obtain a monomial $G_8(q)$-module $\mathbb{C}\left(G_2^{syl}(q)\right)$
            (see Section \ref{sec: monomial G2-module}).
 \item [2.] {\it Establish supercharacters of $G_2^{syl}(q)$ by decomposing monomial $G_2^{syl}(q)$-modules:}
            Every supercharacter is afforded by a direct sum of
            some $G_2^{syl}(q)$-orbit modules
            which is also a direct sum of restrictions of certain $G_8(q)$-orbit modules
            to $G_2^{syl}(q)$
            (see Sections \ref{sec:U-orbit modules-G2},
            \ref{sec:homom. between orbit modules-G2}
             and \ref{sec: supercharacter theories-G2}).
 \item [3.] {\it Calculate the superclasses using the intermediate group $G_8(q)$:}
            Every superclasses is a union of some intersections of biorbits of $G_8(q)$ and $G_2^{syl}(q)$,
            i.e. $\{I_8+g(u-I_8)h\mid g, h \in G_8(q)\}\cap G_2^{syl}(q)$ for all $u\in G_2^{syl}(q)$,
            where $I_8$ is the identity element of $G_2^{syl}(q)$
            (see Sections \ref{partition of U-G2}
            and \ref{sec: supercharacter theories-G2}).
\end{itemize}

We mention that supercharacter theories have proven to raise other questions in particular concerning algebraic combinatorics.
For example, Hendrickson
obtained the connection between supercharacter theories and Schur rings \cite{Hend}.

The set of complex irreducible characters and the set of conjugacy classes form a trivial supercharacter theory for a finite group.
It is also natural to consider Higman's conjecture,
Lehrer's conjecture and Isaacs' conjecture for
the Sylow $p$-subgroups of other Lie
types.
Let $G(q)$ be a finite group of Lie type,
$U(q)$ a Sylow $p$-subgroup of $G(q)$,
and $\#\mathrm{Irr}(U(q))$ the number of all complex irreducible characters (i.e. the number of conjugacy classes).
The $\#\mathrm{Irr}(U(q))$ for $U(q)$ of rank at most 8, except $E_8$, are calculated using an algorithm \cite{GMR2014, GMR2016, GR2009}.
For the Sylow $p$-subgroup $U(q)$ of type $D_4$,
the complex irreducible characters  in \cite{HLMD42011},
the $\#\mathrm{Irr}(U(q))$ in \cite{HLMD42011, Markus1},
and the {\it generic character table} in \cite{GLM2017} are determined.
The irreducible characters of
the Sylow $p$-subgroup $U(q)$ of type $F_4$ in \cite{GLMP2016}
and of type $E_6$ in \cite{LMParxiv2017}
are parameterized.

For the Sylow $p$-subgroup ${{^3D}_4^{syl}}(q^3)$ of the Steinberg triality group $^3D_4(q^3)$,
irreducible characters have been classified by Le \cite{Le}
and the character tables have been given by the author explicitly in \cite{sun1}.
For the Sylow $p$-subgroup $G_2^{syl}(q)$ $(p>3)$ of the Chevalley group $G_2(q)$ of type $G_2$,
the number of conjugacy classes of $G_2^{syl}(q)$
is obtained with an algorithm in \cite{GMR2016, GR2009},
and most irreducible characters (except $q^2-2q+2$ linear characters) of $G_2^{syl}(q)$
are determined by parameterizing {\it midafis} in \cite{HLM2016}.

In this paper, we further
calculate the conjugacy classes of $G_2^{syl}(q)$ $(p>3)$,
and get the relations between the superclasses and conjugacy classes
(see Section \ref{sec:conjugacy classes-G2}).
Then we construct all of the complex irreducible characters of $G_2^{syl}(q)$,
and obtain the relations between the supercharacters and irreducible characters
(see Section \ref{sec:irreducible characters-G2}).
After that, we establish the character table for $G_2^{syl}(q)$
(see Section \ref{sec:irreducible characters-G2}).
Higman's conjecture,
Lehrer's conjecture and Isaacs' conjecture are true for $G_2^{syl}(q)$.

At the end of each section,
we compare the properties of $G_2^{syl}(q)$ and ${^3}D_4^{syl}(q^3)$.
Some related properties of $A_n(q)$, $D_n^{syl}(q)$ and ${^3}D_4^{syl}(q^3)$
are given in  \cite{sun3D4super}.

Here we fix some notation:
Let
$\mathbb{N}$ be the set $\{0,1,2, \dots \}$ of all non-negative integers,
$K$ a field,
$K^*$ the multiplicative group $K\backslash\{0\}$ of $K$,
$K^+$ the additive group of $K$,
$\mathbb{F}_{q^3}$ the finite field with $q^3$ elements,
$\mathbb{C}$ the complex field.
Let $\mathrm{Mat}_{8 \times 8}(K)$
be the set of all $8\times 8$ matrices with entries in the field $K$,
the {general linear group}
${GL}_8(K)$ be the subset of $\mathrm{Mat}_{8\times 8}(K)$
consisting of all invertible matrices.
If $m\in \mathrm{Mat}_{8\times 8}(K)$, then set $m:=(m_{i,j})$,
where $m_{i,j}\in K$ denotes the $(i,j)$-entry of $m$.
For simplicity, we write $m_{ij}:=m_{i,j}$ if there is no ambiguity.
Denote by $e_{i,j}\in \mathrm{Mat}_{8\times 8}(K)$ the matrix unit
with $1$ in the $(i,j)$-position and $0$ elsewhere,
and denote by $A^{\top}$ the {transpose} of $A\in \mathrm{Mat}_{8\times 8}(K)$.
Let $O_8$ be the zero $8 \times 8$-matrix $O_{8\times 8}$,
and $1$ denote the identity element of a finite group.


\section{Sylow $p$-subgroup $G_2^{syl}(q)$ of Lie type $G_2$}
\label{sec:G2-1}
In this section,
we construct a Lie algebra of type $G_2$
and its corresponding Chevalley basis (see \ref{Lie algebra, G2}),
and then determine the Sylow $p$-subgroup ${G}_2^{syl}(q)$ of the Chevalley group
of type $\mathcal{L}_{G_2}$ over the field $\mathbb{F}_q$
(see \ref{sylow p-subg, G2}).
The main references are
\cite{Carter1, Carter2005}.

We recall the construction of Lie algebra of type $D_4$
and the Sylow $p$-subgroup ${}^3{D}_4^{syl}(q^3)$ (see \cite[\S2]{sun3D4super}).
If
${J_{8}^+}:=
\sum_{i=1}^8{e_{i,9-i}}\in {GL}_{8}(\mathbb{C})$,
then $\{A \in \mathrm{Mat}_{8\times 8}(\mathbb{C}) \mid  A^\top J_{8}^+ + J_{8}^+ A=0 \}$
forms a complex simple Lie algebra $\mathcal{L}_{D_4}$ of type $D_4$.
For $1\leq i\leq 4$, let
$h_i:= e_{i,i}-e_{9-i,9-i} \in \mathrm{Mat}_{8\times 8}(\mathbb{C})$.
Then a Cartan subalgebra of $\mathcal{L}_{D_4}$ is
$\mathcal{H}_{D_4}= \{\sum_{i=1}^4{\lambda_ih_i}  \mid \lambda_i\in \mathbb{C}\}$.
Let $\mathcal{H}_{D_4}^*$ be the dual space of $ \mathcal{H}_{D_4}$,
and $h:=\sum_{i=1}^4{\lambda_ih_i}$.
For $1 \leq i \leq 4$,
let $\varepsilon_i \in \mathcal{H}_{D_4}^*$
be defined by $\varepsilon_i(h)=\lambda_i$ for all $i=1,2,3,4$.
If $\mathcal{V}_4:=\mathcal{V}_{D_4}$
is a $\mathbb{R}$-vector subspace of $\mathcal{H}_{D_4}^*$
spanned by $\{ h_i \mid i=1,2,3,4\}$,
then $\mathcal{V}_4$ becomes a Euclidean space (see \cite[\S 5.1]{Carter2005}).
The set
$\Phi_{D_4}=\{\pm\varepsilon_i\pm\varepsilon_j  \mid  1\leq i< j\leq 4\}$
is a root system of type ${D}_4$.
The fundamental system of roots of the root system $\Phi_{D_4}$ is
$\Delta_{D_4}
=\{\varepsilon_1-\varepsilon_2,\ \varepsilon_2-\varepsilon_3,
\ \varepsilon_3-\varepsilon_4,\ \varepsilon_3+\varepsilon_4\}$.
The {positive system of roots} of $\Phi_{D_4}$ is
$\Phi_{D_4}^+:=\{\varepsilon_i\pm\varepsilon_j  \mid  1 \leq i< j\leq4\}$.

Let
$ r_1:=\varepsilon_1-\varepsilon_2$,
$r_2:=\varepsilon_2-\varepsilon_3$,
$r_3:=\varepsilon_3-\varepsilon_4$,
$ r_4:=\varepsilon_3+\varepsilon_4$,
$ r_5:= r_1+r_2$,
$ r_6:= r_2+r_3$,
$ r_7:= r_2+r_4$,
$ r_8:= r_1+r_2+r_3$,
$ r_{9}:= r_1+r_2+r_4$,
$ r_{10}:= r_2+r_3+r_4$
$ r_{11}:= r_1+r_2+r_3+r_4$,
and
$ r_{12}:= r_1+2r_2+r_3+r_4$.
Then $\{h_r \mid r\in \Delta_{D_4}\}\cup \{ e_{\pm r} \mid r\in \Phi_{D_4}^+\}$
is a Chevalley basis of the Lie algebra $\mathcal{L}_{D_4}$,
where
$e_{r_1}:=e_{12}-e_{78}$,
$e_{r_2}:=e_{23}-e_{67}$,
$ e_{r_3}:=e_{34}-e_{56}$,
$ e_{r_4}:=e_{35}-e_{46}$,
$e_{r_5}:={-}(e_{13}-e_{68})$,
$ e_{r_6}:=e_{24}-e_{57}$,
$ e_{r_7}:=e_{25}-e_{47}$,
$ e_{r_8}:=e_{14}-e_{58}$,
$ e_{r_{9}}:=e_{15}-e_{48}$,
$ e_{r_{10}}:=e_{26}-e_{37}$,
$ e_{r_{11}}:=e_{16}-e_{38}$,
$ e_{r_{12}}:=e_{17}-e_{28}$,
and $e_{-r}:=e_r^\top$ for all $r\in \Phi_{D_4}^+$,
$h_r:=[e_r,e_{-r}]=e_re_{-r}-e_{-r}e_r$ for all  $r\in \Delta_{D_4}$.


The group $D_4^{syl}(q):=
  \left<{\,}\exp(te_r) {\,}\middle|{\,} r\in \Phi_{D_4}^+,{\ } t\in \mathbb{F}_q {\,}\right>$
is a Sylow $p$-subgroup of the Chevalley group
${D}_4(q):=\left<{\,} \exp(t{\,}\mathrm{ad}{\,e_r}) {\,}\middle|{\,}  r\in \Phi_{D_4},{\,} t\in \mathbb{F}_q {\,}\right>$.
We set
${x_r(t)}:= \exp(te_r)=I_{8}+t\cdot e_r$
for all $r\in \Phi_{D_4}$ and $t\in \mathbb{F}_q$,
and the root subgroups
${X_r}:= \{ {x_r(t)} {\ }| {\ } t\in {\mathbb{F}_q}\}$
for all $r\in \Phi_{D_4}$.
We have $
 D_4^{syl}(q)=\{\prod_{r\in \Phi^+_{D_4}}{x_r(t_r)}
 {\,}|{\,}
 t_r \in \mathbb{F}_q \}$,
where the product can be taken in an arbitrary, but fixed, order.

Let $\rho$ be a linear transformation of $\mathcal{V}_4$ into itself
arising from a non-trivial symmetry of the Dynkin diagram of $\mathcal{L}_{D_4}$
sending $r_1$ to $r_3$, $r_3$ to $r_4$, $r_4$ to $r_1$, and fixing $r_2$.
Then $\rho^3=\mathrm{id}_{\mathcal{V}_4}$.
Let an automorphism of the Lie algebra $\mathcal{L}_{D_4}$ be determined by
$h_r\mapsto h_{\rho(r)},\ e_r\mapsto e_{\rho(r)},\ e_{-r}\mapsto e_{-\rho(r)}$
 ($r\in \Delta_{D_4}$),
and satisfy that for every $r\in \Phi_{D_4}$ $e_r\mapsto \gamma_re_{\rho(r)}$.
We have
$\gamma_r=1$ for all $r\in \Phi_{D_4}$.
The Chevalley group $D_4(q^3)$ has a field automorphism
$
F_q \colon \mathbb{F}_{q^3}  \to  \mathbb{F}_{q^3}:t\mapsto t^q
$
sending $x_{r}(t)$ to $x_{r}(t^q)$,
and a graph automorphism $\rho$ sending $x_{r}(t)$ to $x_{\rho(r)}(t)$  $(r\in \Phi_{D_4})$
(see \cite[12.2.3]{Carter1}).
Let $F:=\rho {F_q}={F_q} \rho$.
For a subgroup $X$ of $D_4(q^3)$,
we set $X^F:=\{x\in X| F(x)=x\}$.
Then $D_4(q^3)^F={^3}D_4(q^3)$.

For $r\in \Phi_{D_4}^+$ and $t\in \mathbb{F}_{q^3}$,
let
$
x_{r^1}(t):=
 {\left\{
 \begin{array}{ll}
 x_r(t)& \text{if } \rho(r)=r,{\ }t^q=t\\
 x_r(t)\cdot x_{\rho(r)}(t^q)\cdot x_{\rho^2(r)}(t^{q^2}) & \text{if } \rho(r)\neq r,{\ }t^{q^3}=t\\
 \end{array}
 \right.}$.
Then a Sylow $p$-subgroup of $^3D_4(q^3)$ is
\begin{align*}
{}^3{D}_4^{syl}(q^3):
&=\left\{x_{r_2^1}(t_2)x_{r^1_1}(t_1)x_{r^1_5}(t_5)x_{r^1_8}(t_8)x_{r^1_{11}}(t_{11})x_{r^1_{12}}(t_{12})
{\,}\middle|{\,}
{\left\{\begin{array}{l}
t_1, t_5, t_8 \in {\mathbb{F}_{q^3}}\\
t_2, t_{11}, t_{12} \in {\mathbb{F}_{q}}
\end{array}
\right.}
\right\}.
\end{align*}
In particular, $|{}^3{D}_4^{syl}(q^3)|=q^{12}$.

For $t\in \mathbb{F}_{q^3}$, we set
$
x_1(t):=x_{r_1^1}(t)
   =x_{r_{1}}(t)
  \cdot x_{r_{3}}(t^q)
  \cdot x_{r_{4}}(t^{q^2})$,
$x_3(t):=x_{r_5^1}(t)
=x_{r_{5}}(t)
  \cdot x_{r_{6}}(t^q)
  \cdot x_{r_{7}}(t^{q^2})$,
$x_4(t):=x_{r_8^1}(t)
=x_{r_{8}}(t)
 \cdot x_{r_{10}}(t^q)
 \cdot x_{r_{9}}(t^{q^2})$.
For $t\in \mathbb{F}_q$, let
$x_2(t):=x_{r_2^1}(t)
=x_{r_{2}}(t)$,
$x_5(t):=x_{r_{11}^1}(t)
=x_{r_{11}}(t)$,
$x_6(t):=x_{r_{12}^1}(t)
=x_{r_{12}}(t)$.
Then the root subgroups of $^3D_4^{syl}(q^3)$ are
$X_i=\{x_i(t) \mid t\in \mathbb{F}_{q^3}\}$ $(i=1,3,4)$
and $X_i=\{x_i(t) \mid t^q=t,{\ }t\in \mathbb{F}_{q^3}\}$ $(i=2,5,6)$.

Let
$ x(t_1, t_2, t_3, t_4, t_5, t_6)
 := x_2(t_2)x_1(t_1)x_3(t_3)x_4(t_4)x_5(t_5)x_6(t_6)\in {^3{D}_4^{syl}(q^3)}$.
Then
\begin{align*}
 {}^3{D}_4^{syl}(q^3)
 =&\left\{ x(t_1, t_2, t_3, t_4, t_5, t_6){\,}\middle|{\,}
 t_1, t_3, t_4 \in {\mathbb{F}_{q^3}},{\ } t_2, t_5, t_6 \in {\mathbb{F}_{q}}\right\}.
\end{align*}

Motivated by \cite[\S3.4 and \S3.6]{HRT},
we construct a Lie algebra of type $G_2$
which is a subalgebra of $\mathcal{L}_{D_4}$.

Let
$e_{1}:= e_{r_1}+e_{r_3}+e_{r_4},
{\ }
e_{2}: =  e_{r_2},
{\ }
e_{3}:= e_{r_5}+e_{r_6}+e_{r_7},
{\ }
e_{4}:= e_{r_8}+e_{r_{10}}+e_{r_{9}},
{\ }
e_{5}:= e_{r_{11}},
{\ }
e_{6}:= e_{r_{12}}$,
$f_i:=e_i^{\top}$
and $\tilde{h}_i:=[e_i,f_i]$ for all $i=1,2,\dots, 6$.
Then
$e_i, f_i, \tilde{h}_i \in \mathcal{L}_{D_4}$ $(i=1,2,\dots,6)$.
Define two vector subspaces of $\mathcal{L}_{D_4}$ as follows:
\begin{align*}
\tilde{\mathcal{H}}
 :=& \mathbb{C}
 \{\tilde{h}_1,\tilde{h}_2\}
             = \big\{\sum_{i=1}^3\lambda_ih_i {\, }\big|{\, } \lambda_1-\lambda_2-\lambda_3=0\big\}
             = \big\{\sum_{i=1}^4\lambda_ih_i {\, }\big|{\, } \lambda_1-\lambda_2-\lambda_3=0,{\, } \lambda_4=0\big\}
              \subseteq \mathcal{H}_{D_4},\\
 \tilde{\mathcal{L}}
 :=& \mathbb{C}
 \{\tilde{h}_1,\tilde{h}_2, e_i,f_j  \mid i,j=1,2,3,4,5,6\}
             = \tilde{\mathcal{H}}\oplus \sum_{i=1}^6{\mathbb{C}e_i}\oplus \sum_{j=1}^6{\mathbb{C}f_j}.
\end{align*}

\begin{Proposition}[Lie algebra of type $G_2$]\label{Lie algebra, G2}
\begin{itemize}
\setlength\itemsep{0em}
 \item [(1)] $ \tilde{\mathcal{L}}$ is a 14-dimensional subalgebra of the Lie algebra $\mathcal{L}_{D_4}$.
 \item [(2)] $ \tilde{\mathcal{L}}$ is a Lie algebra of type $G_2$.
 \item [(3)] $\{\tilde{h}_k \mid k=1,2\}\cup \{e_i,f_i  \mid  i=1,2,\dots,6\}$
 is a Chevalley basis of $\tilde{\mathcal{L}}$.
\end{itemize}
\end{Proposition}

\begin{proof}
\begin{itemize}
\setlength\itemsep{0em}
\item [(1)]
 $\tilde{\mathcal{L}}$ is closed under the Lie bracket $[{\ },{\ }]$
 with straightforward calculation.
\item [(2)]  Let $\tilde{h}:=\sum_{i=1}^3\lambda_ih_i
                            = \lambda_1 \tilde{h}_1 +(\lambda_1+\lambda_2)\tilde{h}_2
                            \in  \tilde{\mathcal{H}}$.
Then
\begin{alignat*}{3}
[\tilde{h},e_1]=&\frac{\lambda_1-\lambda_2+2\lambda_3}{3}e_1,
& \quad
[\tilde{h},e_2]=&(\lambda_2-\lambda_3)e_{2},
& \quad
[\tilde{h},e_3]=& \frac{\lambda_1+2\lambda_2-\lambda_3}{3}e_3,\\
[\tilde{h},e_4]=& \frac{2\lambda_1+\lambda_2+\lambda_3}{3}e_4,
& \quad
[\tilde{h},e_5]=& (\lambda_1+\lambda_3)e_5,
& \quad
[\tilde{h},e_6]=& (\lambda_1+\lambda_2)e_6,\\
[\tilde{h},f_1]=&-\frac{\lambda_1-\lambda_2+2\lambda_3}{3}f_1,
& \quad [\tilde{h},f_2]=&-(\lambda_2-\lambda_3)f_{2},
& \quad [\tilde{h},f_3]=& -\frac{\lambda_1+2\lambda_2-\lambda_3}{3}f_3,\\
[\tilde{h},f_4]=& -\frac{2\lambda_1+\lambda_2+\lambda_3}{3}f_4,
& \quad [\tilde{h},f_5]=& -(\lambda_1+\lambda_3)f_5,
& \quad [\tilde{h},f_6]=& -(\lambda_1+\lambda_2)f_6.
\end{alignat*}
The functions ${\alpha},\beta \colon \tilde{\mathcal{H}} \to  \mathbb{C}$ are
$\alpha(\tilde{h})= \frac{\lambda_1-\lambda_2+2\lambda_3}{3}$
and
 $\beta(\tilde{h})= \lambda_2-\lambda_3$,
then  set
$\tilde\Phi:=\pm\{ \alpha,{\ } \beta,{\ }\alpha+\beta,
                  {\ }2\alpha+\beta, {\ }3\alpha+\beta, {\ }3\alpha+2\beta\}$.
Thus we write $[\tilde{h},e_{\tilde{\alpha}}]={\tilde{\alpha}(\tilde{h})} e_{\tilde{\alpha}}$
and $\tilde{\mathcal{L}}_{\tilde{\alpha}}=\mathbb{C}e_{\tilde{\alpha}}$
for all $\tilde{\alpha}\in \tilde{\Phi}$
and $e_{\tilde{\alpha}}\in \{e_i,{\ }f_i \mid i=1,2,\dots,6\}$.

{\it We claim that $\tilde{\mathcal{H}}$ is a Cartan subalgebra}.
We know $\tilde{\mathcal{H}}$ is abelian.
Now it is sufficient to show that
$
\tilde{\mathcal{H}}=
N_{\tilde{\mathcal{L}}}(\tilde{\mathcal{H}})
:=\{x\in {\tilde{\mathcal{L}}} \mid
   [x,h]\in \tilde{\mathcal{H}},{\ }\forall \,h\in \tilde{\mathcal{H}}\}$.
If $x\in N_{\tilde{\mathcal{L}}}(\tilde{\mathcal{H}})$,
then $x=h'+\sum_{i=1}^6{(a_ie_i+b_if_i)}$
    for $h'\in \tilde{\mathcal{H}}$ and $a_i,b_i\in \mathbb{C}$.
If $h=4h_1+3h_2+h_3\in \tilde{\mathcal{H}}$, then
$
  [h,x]=(a_1e_1-b_1f_1)+2(a_2e_2-b_2f_2)+3(a_3e_3-b_3f_3)
   +4(a_4e_4-b_4f_4)+5(a_5e_5-b_5f_5)+7(a_6e_6-b_6f_6)
                    \in \tilde{\mathcal{H}}$,
so $a_i=b_i=0$ { for all } $i=1,2,3,4,5,6$.
Thus $x\in \tilde{\mathcal{H}}$,
and  $\tilde{\mathcal{H}}$ is a Cartan subalgebra of $\tilde{\mathcal{L}}$.
Therefore,
\begin{align*}
  \tilde{\mathcal{L}}
  =&\tilde{\mathcal{H}}\oplus \sum_{i=1}^6{\mathbb{C}e_i}\oplus \sum_{j=1}^6{\mathbb{C}f_j}
  =\tilde{\mathcal{H}}\oplus \sum_{\tilde{\alpha}\in \tilde\Phi}{\mathbb{C}e_{\tilde{\alpha}}}.
\end{align*}
is a Cartan decomposition with respect to $\tilde{\mathcal{H}}$.

{\it We claim that $\tilde{\mathcal{L}}$ is semisimple}.
Suppose there exists a non-zero ideal $\tilde{I}$ of $\tilde{\mathcal{L}}$.
Then $[\tilde{\mathcal{H}},\tilde{I}]\subseteq \tilde{I}$.
We may regard $\tilde{I}$ as a $\tilde{\mathcal{H}}$-module and decompose it into weight spaces as follows:
$
 \tilde{I}=
 (\tilde{\mathcal{H}}\cap \tilde{I})
 \oplus \sum_{\tilde{\alpha}\in \tilde\Phi}{(\mathbb{C}e_{\tilde{\alpha}}\cap \tilde{I})}$.
If $x\in \tilde{I}$, then $x=x_0+\sum_{\tilde{\alpha}\in \tilde\Phi}{x_{\tilde{\alpha}}}$
where $x_0\in \tilde{\mathcal{H}}$ and $x_{\tilde{\alpha}}\in \tilde{\mathcal{L}}_{\tilde{\alpha}}$.
We verify that $x_0\in \tilde{I}$ and $x_{\tilde{\alpha}}\in \tilde{I}$.
If $\tilde{\alpha}_0\in \tilde\Phi$
and $h=4\tilde{h}_1+7\tilde{h}_2=4h_1+3h_2+h_3\in \tilde{\mathcal{H}}$,
then ${\tilde{\alpha}}_0(h)\neq 0$
and $\tilde{\beta}(h)\neq {\tilde{\alpha}}_0(h)$
for all $\tilde{\beta}\in \tilde\Phi$ with $\tilde{\beta}\neq \tilde{\alpha}$.
Thus
\begin{align*}
\Big(
\mathrm{ad}{\, h} \prod_{\substack{\tilde{\beta}\in \tilde\Phi \\ \tilde{\beta}\neq {\tilde{\alpha}_0}}}
       {(\mathrm{ad}{\, h}-\tilde{\beta}(h)\mathrm{id}_{\tilde{\mathcal{L}}})}
\Big)
(x)
={\tilde{\alpha}_0}(h)\prod_{\substack{\tilde{\beta}\in \tilde\Phi \\ \tilde{\beta}\neq {\tilde{\alpha}_0}}}
       {({\tilde{\alpha}_0}(h)-\tilde{\beta}(h)) x_{{\tilde{\alpha}_0}}}
\in \tilde{I},
\end{align*}
so $x_{{\tilde{\alpha}_0}} \in \tilde{I}$
and $x_0\in \tilde{I}$. Hence
$
\tilde{I}=
 (\tilde{\mathcal{H}}\cap \tilde{I})
 \oplus \sum_{\tilde{\alpha}\in \tilde\Phi}{(\mathbb{C}e_{\tilde{\alpha}}\cap \tilde{I})}$.
We claim that $\mathbb{C}e_{\tilde{\alpha}}\cap \tilde{I}=\{0\}$.
Suppose that $\mathbb{C}e_{\tilde{\alpha}}\cap \tilde{I}\neq\{0\}$
for some $\tilde{\alpha}\in \tilde\Phi$.
Then $e_{\tilde{\alpha}}\in \tilde{I}$.
So
$\tilde{h}_{\tilde{\alpha}}=[e_{\tilde{\alpha}},e_{-\tilde{\alpha}}]
\in \tilde{I}$ and
$[\tilde{h}_{\tilde{\alpha}},e_{\tilde{\alpha}}]=2e_{\tilde{\alpha}}$.
This is a contradiction to that $\tilde{I}$ is abelian.
Thus $\mathbb{C}e_{\tilde{\alpha}}\cap \tilde{I}=\{0\}$
and $\tilde{I}\subseteq \tilde{\mathcal{H}}$.
If $x\in \tilde{I}$,
then $[x,e_{\tilde{\alpha}}]=\tilde{\alpha}(x)e_{\tilde{\alpha}}
\in \tilde{I}$
for all $\tilde{\alpha}\in \tilde\Phi$.
Thus $\tilde{\alpha}(x)=0$ and $x=0$.
Hence $\tilde{I}=\{0\}$, which is a contradiction.
Therefore, $\tilde{\mathcal{L}}$ is semisimple.

The functions $\tilde{\alpha}\in \tilde\Phi$ are the
roots of $\tilde{\mathcal{L}}$ with respect to $\tilde{\mathcal{H}}$.
A system of fundamental roots is $\{ \alpha,\ \beta \}$,
since all the other roots are integral combinations of these with coefficients
all non-negative or non-positive.
Thus the set of the roots is $\tilde\Phi$.
We determine the Cartan matrix of $\tilde{\mathcal{L}}$.
The $\alpha$-chain of roots through $\beta$ is
$\{\beta, {\ }\beta+\alpha, {\ }\beta+2\alpha, {\ }\beta+3\alpha\}$.
Then $\beta$-chain of roots through $\alpha$ is
$\{\alpha, {\ }\alpha+\beta\}$.
Then $A_{\alpha,\beta}=0-3=-3$ and $A_{\beta,\alpha}=-1$.
Thus the Cartan matrix of $\tilde{\mathcal{L}}$ is
$
\begin{pmatrix}
 2 & -3   \\
 -1 & 2   \\
\end{pmatrix}
$
for the ordering $(\alpha, \beta)$.
It is a Cartan matrix of type $G_2$.
Then the Lie algebra $\tilde{\mathcal{L}}$ is simple since the Cartan matrix is indecomposable.

Therefore, $\tilde{\mathcal{L}}$ is a simple Lie algebra of type $G_2$.
\item [(3)] The co-roots of $\tilde{\mathcal{L}}$ are $\tilde{h}_i=[e_i,f_i]$ $(i=1,2,\dots,6)$.
For in this case $[\tilde{h}_i,e_i]=2e_i$ with $i=1,2,\dots,6$.
We know that $\tilde{\theta}(x)=-x^\top $ is an automorphism of $\tilde{\mathcal{L}}$ with
$\tilde{\theta}(e_i)=-f_i$.
Hence for all $\tilde{\alpha}, \tilde{\beta}\in \tilde{\Phi}$,
$[e_{\tilde{\alpha}},e_{\tilde{\beta}}]
                    =\pm(n_{\tilde{\alpha},\tilde{\beta}}+1)e_{\tilde{\alpha}+\tilde{\beta}}$,
where $n_{\tilde{\alpha},\tilde{\beta}}$ is the biggest integer for which
$\tilde{\beta}-n_{\tilde{\alpha},\tilde{\beta}}\tilde{\alpha} \in \tilde{\Phi}$.
Thus the fundamental co-roots $\tilde{h}_k$ $(k=1,2)$
together with $e_i$, $f_i$ ($i=1,2,\dots, 6$)
form a Chevalley basis of the Lie algebra $\tilde{\mathcal{L}}$.
\end{itemize}
\end{proof}

Let
$\mathcal{L}_{G_2}:=  \tilde{\mathcal{L}}$
and
  $\mathcal{H}_{G_2}:=  \tilde{\mathcal{H}}$.
Then
$
 \mathcal{L}_{G_2}=\mathcal{H}_{G_2}\oplus \sum_{i=1}^6{\mathbb{C}e_i}\oplus \sum_{j=1}^6{\mathbb{C}f_j}$.
If $\mathcal{V}_{G_2}:= \langle \mathcal{H}_{G_2}^*\rangle_{\mathbb{R}}$,
then $\Delta_{G_2}=\{\alpha,\  \beta\}$ is a basis of $\mathcal{V}_{G_2}$.
The set of the root is
$
\Phi_{G_2}=\pm\{ \alpha,{\ } \beta,{\ }\alpha+\beta,
                  {\ }2\alpha+\beta, {\ }3\alpha+\beta, {\ }3\alpha+2\beta\}$.
The set of positive roots is denoted by
$
\Phi_{G_2}^+=\{ \alpha,{\ } \beta,{\ }\alpha+\beta,
                  {\ }2\alpha+\beta, {\ }3\alpha+\beta, {\ }3\alpha+2\beta\}$.
Let
 \begin{align*}
 \begin{array}{rrrrrr}
  h_\alpha:=\tilde{h}_1, &
  h_\beta:=\tilde{h}_2, &
  h_{\alpha+\beta}:=\tilde{h}_3, &
  h_{2\alpha+\beta}:=\tilde{h}_4, &
  h_{3\alpha+\beta}:=\tilde{h}_5, &
  h_{3\alpha+2\beta}:=\tilde{h}_6, \\
  e_\alpha:= e_1, &
  e_\beta:= e_2, &
  e_{\alpha+\beta}:= e_3, &
  e_{2\alpha+\beta}:= e_4, &
  e_{3\alpha+\beta}:= e_5, &
  e_{3\alpha+2\beta}:= e_6, \\
 e_{-\alpha}:= f_1, &
 e_{-\beta}:=f_2, &
 e_{-(\alpha+\beta)}:=f_3, &
 e_{-(2\alpha+\beta)}:=f_4, &
 e_{-(3\alpha+\beta)}:=f_5, &
 e_{-(3\alpha+2\beta)}:=f_6.
 \end{array}
 \end{align*}
Then
 $\{h_\alpha, h_\beta\} \cup \{e_{\pm r} \mid r\in \Phi_{G_2}^+\}$
 is a Chevalley basis of $\mathcal{L}_{G_2}$.

Let $r:=x_1\alpha+x_2\beta \in \mathcal{V}_{G_2}$, $s:=y_1\alpha+y_2\beta \in \mathcal{V}_{G_2}$.
Then we write
$ r\prec s$,
if $\sum_{i=1}^2{x_i}<\sum_{i=1}^2{y_i}$,
or if $\sum_{i=1}^2{x_i}= \sum_{i=1}^2{y_i}$ and
the first non-zero coefficient $x_i-y_i$ is positive.
The total order on $\Phi_{G_2}^+$ is determined:
$
 0\prec
 \alpha \prec \beta \prec
 \alpha+\beta \prec
 2\alpha+\beta \prec
 3\alpha+\beta \prec
 3\alpha+2\beta $.
The Lie algebra $\mathcal{L}_{G_2}$ has the following structure constants:
$N_{\alpha, \beta }=-1$,
$N_{\alpha, \alpha+\beta}=2$,
$N_{\alpha, 2\alpha+\beta}=3$
and
$N_{\beta, 3\alpha+\beta}=1$.

We have
$e_R^3=0$ for all $R\in \Phi_{G_2}^+$ and
 $e_\alpha^2= -2e_{3,6}$,      $e_{\alpha+\beta}^2= -2e_{2,7}$,  $e_{2\alpha+\beta}^2= -2e_{1,8}$
 and
$ e_\beta^2= e_{3\alpha+\beta}^2= e_{3\alpha+2\beta}^2= 0$.
The coefficients of $e_{i,j}$ in $\exp (te_r)=I_8+te_r+\frac{1}{2}t^2e_r^2$
for all $r\in \Phi_{G_2}$
are of the form $\pm 1$, $\pm t$ or $\pm t^2$,
because the coefficient of $e_r^2$ with $r\in \Phi_{G_2}$
are all divisible by $2$.
This fact enables us to transfer to an arbitrary field.
For each matrix $e_r$ in the above representation
and each element $t$ in an arbitrary field $K$,
$\exp (te_r)$ is a well-defined non-singular matrix over $K$.
We are interested in the Chevalley group of type $\mathcal{L}_{G_2}$
over the finite field $\mathbb{F}_q$ with $\mathrm{Char}{\,\mathbb{F}_q} \neq 2$.

The Chevalley group of type $\mathcal{L}_{G_2}$ is
$
G_2(q):= \left<{\,} \exp (t{\,}\mathrm{ad}{\,e_r})
            {\, }\middle|{\, } r\in \Phi_{G_2},{\ } t\in \mathbb{F}_q {\,}\right>$,
and its Sylow $p$-subgroup is
$
U_{G_2}:= \left<{\,} \exp (t{\,}\mathrm{ad}{\,e_r}) {\, }\middle|{\, }
                    r\in \Phi_{G_2}^+,{\ } t\in \mathbb{F}_q {\,}\right>$.
Set
 $y_r(t):=\exp(te_r)=I_8+te_r+\frac{1}{2}t^2e_r^2$
  for all  $r\in \Phi_{G_2},{\ } t\in \mathbb{F}_q$.
Write
$\bar{U}_{G_2}:= \left<{\,} y_r(t) {\, }\middle|{\, } r\in \Phi_{G_2}^+,{\ } t\in \mathbb{F}_q {\,}\right>$.
\begin{Lemma} 
The root subgroups of $\bar{U}_{G_2}(q)$ are
$Y_i:=\left\{y_i(t) {\, }\middle|{\, } t\in \mathbb{F}_q \right\}$
 for all $i=1,2,3,4,5,6$,
where
\begin{alignat*}{2}
y_1(t):=&y_{\alpha}(t)
=x_{r_1}(t)x_{r_3}(t)x_{r_4}(t)
=x_1(t),
& \qquad
  y_2(t):=&y_{\beta}(t)
=x_{r_2}(t)
=x_2(t),\\
y_3(t):=& y_{\alpha+\beta}(t)
=x_{r_5}(t)x_{r_6}(t)x_{r_7}(t)
=x_3(t),
& \qquad
y_5(t):=& y_{3\alpha+\beta}(t)
= x_{r_{11}}(t)
= x_5(t),\\
y_4(t):=&y_{2\alpha+\beta}(t)
= x_{r_{8}}(t)x_{r_{10}}(t)x_{r_{9}}(t)
= x_4(t),
& \qquad
  y_6(t):=&y_{3\alpha+2\beta}(t)
= x_{r_{12}}(t)
= x_6(t).
\end{alignat*}
\end{Lemma}

We note that
$Y_i \leqslant X_i$ $(i=1,3,4)$,
$Y_i=X_i$ $(i=2,5,6)$,
$Y_i\leqslant {^3}D_4^{syl}(q^3)$,
$\bar{U}_{G_2} \leqslant {^3}D_4^{syl}(q^3)$
and $\bar{U}_{G_2} \leqslant D_4^{syl}(q)$.
We have
$
 \bar{U}_{G_2}=
\big\{\prod_{r\in \Phi_{G_2}^+}{y_r(t_r)}
{\, } \big|{\, }
t_r \in {\mathbb{F}_{q}}, r\in \Phi_{G_2}^+ \big \}
 =\big\{\prod_{i\in \{1,2,\dots,6\}}{y_i(t_i)}
 {\, } \big|{\, }
 t_i \in {\mathbb{F}_{q}}\big\}$,
where the product can be taken in an arbitrary, but fixed, order.
In particular, $|\bar{U}_{G_2}|=q^6$.

\begin{Proposition}\label{syl. p-subg. iso.-G2}
Let
$
\sigma_{\bar{U}_{G_2}} \colon \bar{U}_{G_2} \to  {U}_{G_2}:
  \exp(te_r)\mapsto \exp(t{\,}\mathrm{ad}{\,e_r})$,
where $r\in \Phi_{G_2}^+$ and $t\in K$.
Then $\sigma_{\bar{U}_{n,K}}$ is a group isomorphism.
\end{Proposition}
\begin{proof}
We know $\sigma_{\bar{U}_{G_2}}$
is a group epimorphism.
Since $|\bar{U}_{G_2}|=|{U}_{G_2}|=q^6$,
$\sigma_{\bar{U}_{G_2}}$ is an isomorphism.
\end{proof}
 Set ${G}_2^{syl}(q):=\bar{U}_{G_2}$.
\begin{Definition}
A subgroup $P\leqslant {G}_2^{syl}(q)$ is a \textbf{pattern subgroup},
if it is generated by some root subgroups, i.e.
 $P:=\left\langle {\,} Y_i {\,}\middle|{\,} i\in I \subseteq \{1,2,\dots,6\}{\,}\right\rangle \leqslant {G}_2^{syl}(q)$.
\end{Definition}
We get the commutators of $G_2^{syl}(q)$ by calculation.
\begin{Lemma}\label{commutator-G2}
Let $t_1, t_2, t_3,t_4, t_5,t_6 \in \mathbb{F}_{q}$
and define the commutators
\begin{align*}
 [y_i(t_i), y_j(t_j)]:=y_i(t_i)^{-1}y_j(t_j)^{-1}y_i(t_i)y_j(t_j).
\end{align*}
Then the non-trivial commutators of $G_2^{syl}(q)$ are determined as follows:
\begin{align*}
[y_1(t_1), y_2(t_2)]=& y_3(-t_2t_1)\cdot y_4(t_2t_1^{2})\cdot y_5(-t_2t_1^{3})\cdot y_6(2t_2^2t_1^{3}),\\
[y_1(t_1), y_3(t_3)]=& y_4(2t_1t_3)
                       \cdot y_5(-3t_1^{2}t_3)
                       \cdot y_6(-3t_1t_3^2),\\
[y_1(t_1), y_4(t_4)]=& y_5(3t_1t_4),
\quad
[y_3(t_3), y_4(t_4)]= y_6(3t_3t_4),
\quad
[y_2(t_2), y_5(t_5)]= y_6(t_2t_5).
\end{align*}
In particular,
if $\mathrm{Char}\mathbb{F}_{q}=3$,
then the commutators are given as follows:
\begin{align*}
[y_1(t_1), y_2(t_2)]=& y_3(-t_2t_1)\cdot y_4(t_2t_1^{2})\cdot y_5(-t_2t_1^{3})\cdot y_6(2t_2^2t_1^{3}),\\
[y_1(t_1), y_3(t_3)]=& y_4(2t_1t_3),
\quad
[y_2(t_2), y_5(t_5)]= y_6(t_2t_5).
\end{align*}
\end{Lemma}

For $t_i\in \mathbb{F}_q$ with $i\in \{1,2,\dots,6\}$,
 we write
 \begin{align*}
   y(t_1, t_2, t_3, t_4, t_5, t_6):=
  & y_2(t_2)y_1(t_1)y_3(t_3)y_4(t_4)y_5(t_5)y_6(t_6)
 = x(t_1, t_2, t_3, t_4, t_5, t_6)\in G_2^{syl}(q).
 \end{align*}
\begin{Proposition}[Sylow $p$-subgroup ${G}_2^{syl}(q)$] \label{sylow p-subg, G2}
A Sylow $p$-subgroup ${G}_2^{syl}(q)$ of the Chevalley group $G_2(q)$
is written as follows:
\begin{align*}
{G}_2^{syl}(q)=\bar{U}_{G_2} =& \left\{y(t_1, t_2, t_3, t_4, t_5, t_6)
  {\, }\middle|{\, }
   t_1, t_2, t_3, t_4, t_5, t_6 \in {\mathbb{F}_{q}}\right \}\\
=& \left\{x(t_1, t_2, t_3, t_4, t_5, t_6)
  {\, }\middle|{\,  }
  t_1, t_2, t_3, t_4, t_5, t_6 \in {\mathbb{F}_{q}}\right\},
 \end{align*}
{where}
\begin{align*}
& y(t_1, t_2, t_3, t_4, t_5, t_6)
=x(t_1, t_2, t_3, t_4, t_5, t_6)\\
=&
 \left(
\newcommand{\mc}[3]{\multicolumn{#1}{#2}{#3}}
\begin{array}{cccccccc}\cline{2-7}
\mc{1}{c|}{1}
& \mc{1}{c|}{t_1 }
& \mc{1}{c|}{-t_3}
& \mc{1}{c|}{t_1t_3+t_4}
& \mc{1}{c|}{t_1t_3+t_4}
& \mc{1}{c|}{
\begin{array}{l}
 t_1t_4\\
+t_5
\end{array}}
& \mc{1}{c|}{
\rule{0pt}{18pt}
\begin{array}{l}
 -t_1t_3^2
+t_3t_4\\
+t_6
 \end{array}}
& \begin{array}{l}
-2t_1{t_3}{t_4}
-t_1t_6\\
+t_3t_5
-{t_4^2}\\
\end{array}\\\cline{2-7}
×
& \mc{1}{c|}{1}
& \mc{1}{c|}{t_2}
& {t_1}t_2+{t_3}
& t_1t_2+t_3
& \begin{array}{l}
-{t_1^2}t_2\\
+{t_4}
\end{array}
& \begin{array}{l}
-2{t_1}t_2{t_3}\\
-{t_2}{t_4}
-{t_3^2}\\
  \end{array}
& \rule{0pt}{22pt}
\begin{array}{l}
-{t_1^2}t_2t_3
-2{t_1}{t_2}{t_4}\\
-t_2t_5
-2{t_3}{t_4}
-t_6
\end{array}\\\cline{3-3}
× & × & {1} & t_1
& t_1
& -t_1^2
& \begin{array}{l}
-2{t_1}{t_3}
-{t_4}
  \end{array}
& \rule{0pt}{15pt}
\begin{array}{l}
  -{t_1^2}t_3
-2{t_1}{t_4}
-t_5\\
\end{array}\\
× & × & × & 1 & 0 & -{t_1} & -{t_3}
&  \begin{array}{l}
   -{t_1}t_3
   -{t_4}
   \end{array} \\
× & × & × & × & 1 & -{t_1} &-t_3
& \begin{array}{l}
    -{t_1}t_3
    -t_4
   \end{array}\\
× & × & × & × & × & 1
& -t_2
& \begin{array}{l}
    t_1t_2
    +t_3\\
   \end{array}
\\
× & × & × & × & × & × & 1
& -t_1\\
× & × & × & × & × & × & × & 1\\
\end{array}
\right).
\end{align*}
\end{Proposition}

\nomenclature{$G_2^{syl}(q)$}
{a Sylow $p$-subgroup of the Chevalley group $G_2(q)$ \nomrefpage}%

\begin{proof}
By \ref{syl. p-subg. iso.-G2},
$G_2^{syl}(q)$ is a Sylow $p$-subgroup of $G_2(q)$.
By calculation,
we get the matrix form as claimed.
\end{proof}

\begin{Corollary}
$
       {G}_2^{syl}(q)         \leqslant
       {{^3{D}}_4^{syl}}(q^3) \leqslant
        D_4^{syl}(q^3)         \leqslant A_8(q^3)$
{and }
$       {G}_2^{syl}(q)        \leqslant
       {D}_4^{syl}(q)        \leqslant
        D_4^{syl}(q^3)       \leqslant A_8(q^3)$.
\end{Corollary}

 Define the following sets of matrix entry coordinates:
  $\Squar:= \{(i,j) \mid  1\leq i,j \leq 8 \}$,
  $\UR:= \{(i,j) \mid 1\leq i< j \leq 8 \}$
  and
  $\UP:= \{(i,j)\in \Squar  \mid  i < j < 9-i \}$.
For $t\in \mathbb{F}_q$ and $(i,j)\in \UR$,
set $\tilde{x}_{i,j}(t)= I_8+t e_{i,j}\in A_8(q)$.
For $t\in \mathbb{F}_{q^3}$ and $(i,j)\in \UP$,
set $ x_{i,j}(t): = I_8+t e_{i,j}-t e_{9-j,9-i}
=\tilde{x}_{i,j}(t)\tilde{x}_{9-j,9-i}(-t) \in {D}_4^{syl}(q^3)$.
We construct a group $G_8(q)$
such that $G_2^{syl}(q)\leqslant G_8(q) \leqslant A_8(q)$.
Then we determine a monomial $G_8(q)$-module
to imitate the $^3D_4$ case in Section \ref{sec: monomial G2-module},
and use the group $G_8(q)$
to calculate the superclasses of $G_2^{syl}(q)$
in Section \ref{partition of U-G2}.
\begin{Definition/Lemma}[An intermediate group $G_8(q)$]
We set
 \begin{align*}
G_8(q):=&
 \left\{u=(u_{i,j}) \in A_8(q)
 {\,}\middle|{\,}
{\left\{
\begin{array}{ll}
      u_{i,j}=0 &  \text{if }(i,j)=(4,5)\\
      u_{i,j+1}=u_{i,j}    & \text{if }(i,j)\in \{(2,4), (3,4)\}\\
      u_{i-1,j}=u_{i,j}    & \text{if }(i,j)\in \{(5,6), (5,7)\}\\
\end{array}
\right.}
\right\}.
\end{align*}
Then $G_8(q)$ is a subgroup of $A_8(q)$ and $|G_8(q)|=q^{23}$.
\end{Definition/Lemma}

We write
$ \ddot{J}:=\UR\backslash\{(2,5),(3,5),(4,5), (4,6),(4,7)\}$.
For $(i,j)\in \ddot{J}$ and $t\in \mathbb{F}_q$, we set
\begin{align*}
\dot{x}_{i,j}(t):=
{\left\{
\begin{array}{ll}
\tilde{x}_{i,j}(t) \tilde{x}_{i,(j+1)}(t), & (i,j)\in \{(2,4),(3,4)\} \\
\tilde{x}_{i,j}(t) \tilde{x}_{(i-1),j}(t), & (i,j)\in \{(5,6),(5,7)\} \\
\tilde{x}_{i,j}(t), & \text{otherwise}
\end{array}
\right..}
\end{align*}
For $(i,j)\in \ddot{J}$, the subgroups of $G_8(q)$ are
$\dot{Y}_{i,j}:=
\{\dot{x}_{i,j}(t) \mid t\in \mathbb{F}_{q}\}$.

\begin{Proposition}
$
 G_8(q)=\big\{\prod_{(i,j)\in \ddot{J}}{\dot{x}_{i,j}(t_{i,j})}
 {\,}\big|{\,}
 t_{i,j}\in \mathbb{F}_q \big\}$,
where the product can be taken in an arbitrary, but fixed, order.
\end{Proposition}
\begin{proof}
c.f. the proof of Proposition 3.3 of \cite{sun3D4super}.
\end{proof}
Note that
$|G_2^{syl}(q)|=q^{6}$, $|G_8(q)|=q^{23}$, $|A_8(q)|=q^{28}$ and
$G_2^{syl}(q) \leqslant G_8(q) \leqslant A_8(q)$.
Set
 $J:=\{(1,2), (1,3), (1,4), (1,5), (1,6), (1,7), (2,3)\} \subseteq \UP$.
\begin{Comparison}[Sylow $p$-subgroups]
\label{com:sylow-G2}
\begin{itemize}
\setlength\itemsep{0em}
 \item [(1)]
 Similar to ${^3}D_4^{syl}(q^3)$,
for every element of $G_2^{syl}(q)$  in \ref{sylow p-subg, G2},
 we have matrix entries $t_1$, $t_2$ and
 up to sign also $t_3$ with postilions in $J$,
 but $t_4$, $t_5$ and $t_6$
 appear in $J$ only in
 polynomials
 involving the other
 parameters.
 \item[(2)]
We can also obtain a Sylow $p$-subgroup $G_2^{syl}(q)$ of $7\times 7$ matrices
(e.g. see \cite[\S3.6]{HRT} and \cite[\S19.3]{Hum}).
In this paper,
we determine the Sylow $p$-subgroup $G_2^{syl}(q)$ of $8\times 8$ matrices
which is a subgroup of ${^3}D_4^{syl}(q^3)$,
so that the following constructions of the supercharacter theory and the character table
are easier.
\end{itemize}
\end{Comparison}

For the rest of this paper,
the omitted proofs of the properties for $G_2^{syl}(q)$
are the adaption of the corresponding statements of
${^3}D_4^{syl}(q^3)$ (see \cite{sun3D4super} and \cite{sun1}).


\section{Monomial $G_2^{syl}(q)$-module}
\label{sec: monomial G2-module}
Let $G:=G_8(q)$ and $U:=G_2^{syl}(q)$.
In this section,
we construct an $\mathbb{F}_q$-subspace $V$ of $V_0$ (\ref{V, G2}),
establish a monomial linearisation $(f,\kappa|_{V\times V})$ for $G$
(\ref{mono. lin.-G2}),
determine a monomial linearisation $(f|_U,\kappa|_{V\times V})$ for $U$ (\ref{mono. lin. for U-G2}),
and obtain a monomial $G_8(q)$-module $\mathbb{C}U$ (\ref{fund thm U-G2}).

Let $V_0:=\mathrm{Mat}_{8\times 8}(q)$.
For any subset $I\subseteq \Squar$,
let $V_I:=\bigoplus_{(i,j)\in I}{\mathbb{F}_{q}}e_{ij} \subseteq V_0$.
In particular, $V_{\Squarb}=V_0$.
Then $V_I$ is an $\mathbb{F}_{q}$-vector subspace.
We have $\dim_{\mathbb{F}_q}{V_J}=7$,
since
 $J=\{(1,2), (1,3), (1,4), (1,5), (1,6), (1,7), (2,3)\}$.
The \textbf{trace} of $A=(A_{i,j})\in V_0$ is denoted by
$\mathrm{tr}(A):=\sum_{i=1}^{8}{A_{i,i}}$.
The map
$
\kappa\colon V_0\times V_0  \to  {\mathbb{F}_{q}}: (A,B)\mapsto
 \mathrm{tr} (A^\top B)
$
is a symmetric $\mathbb{F}_{q}$-bilinear form on $V_0$
which is called the \textbf{trace form}.
In particular, $\kappa(A,B)=\sum_{(i,j)\in \Squarb}{A_{i,j}B_{i,j}}$
and $\kappa$ is non-degenerate.
 Let $V_J^{\bot}$ denote the orthogonal complement of $V_J$ in $V_0$
 with respect to the trace form $\kappa$, i.e.
 $V_J^{\bot}:=\{B\in V_0 \mid \kappa(A,B)=0,{\ }\forall {\ } A \in V_J\}$.
 Then
  $V_J^{\bot}=V_{\Squarb \backslash J}$
  {and}
  $V_0=V_J \oplus V_J^{\bot}$.
 $\kappa|_{V_J\times V_J}\colon V_J\times V_J \to  \mathbb{F}_q$ is a non-degenerate bilinear form.
The map
$\pi_J\colon V_0=V_J\oplus V_J^\bot  \to  V_J:
  A\mapsto
\sum_{(i,j)\in J}{A_{i,j}e_{i,j}}$
is a projection of $V_0$ to the first component $V_J$.
The \textbf{support}
of $A\in \mathrm{Mat}_{8\times 8}(K)$ is defined by
$ \mathrm{supp}(A):= \{(i,j)\in \Squar \mid  A_{i,j}\neq 0\}$.
If $V\subseteq V_0$ is a subspace of $V_0$, then set
$\mathrm{supp}(V):= \bigcup_{A\in V} \mathrm{supp}(A)$.
Suppose $A,B\in V_0$, such that $\mathrm{supp}(A)\cap\mathrm{supp}(B)\subseteq J$.
Then
$\kappa(A,B)=\kappa(\pi_J(A),B)=\kappa(A,\pi_J(B))
= \kappa(\pi_J(A),\pi_J(B))=\kappa|_{V_J\times V_J}(\pi_J(A),\pi_J(B))$.

\begin{Notation/Lemma}\label{V, G2}
Let
$ V:=\left\{  A=(A_{ij})\in V_0 {\,} \middle|{\, }  \mathrm{supp}(A)\in J,
{\ } A_{14}=A_{15}  \right\}$.
Then $V$ is a 6-dimensional subspace of $V_J$ over $\mathbb{F}_{q}$
and $\mathrm{supp}(V)=J$.
\end{Notation/Lemma}

\begin{Notation/Lemma}\label{pi,G2}
Let
\begin{align*}
& \pi \colon V_0 \to  V:
  A\mapsto
{
\begin{array}{l}
A_{12}e_{12}+A_{13}e_{13}
+\frac{A_{14}+A_{15}}{2}e_{14}
+\frac{A_{14}+A_{15}}{2}e_{15}\\
+A_{16}e_{16}+A_{17}e_{17}+A_{23}e_{23}
\end{array}
},
\end{align*}
i.e.
\begin{align*}
 \pi(A)=&
\left(
{%
\newcommand{\mc}[3]{\multicolumn{#1}{#2}{#3}}
\begin{array}{cccccccc}\cline{2-7}
\mc{1}{c|}{}
& \mc{1}{c|}{A_{12}}
& \mc{1}{c|}{A_{13}}
& \mc{1}{c|}{\rule{0pt}{13pt}\frac{A_{14}+A_{15}}{2}}
& \mc{1}{c|}{\frac{A_{14}+A_{15}}{2}}
& \mc{1}{c|}{A_{16}}
& \mc{1}{c|}{A_{17}}
& \\\cline{2-7}
× & \mc{1}{c|}{}
& \mc{1}{c|}{A_{23}}
& × & × & ×&  & ×\\\cline{3-3}
\end{array}
}%
\right)_{8\times 8}
\end{align*}
omitting all zero entries in the matrices,
in particular at positions $(1,1)$ and $(1,8)$.
Then $\pi$ is ${\mathbb{F}_q}$-epimorphism.
Particularly,
$\pi|_V=\mathrm{id}_{V}$, $\pi^2=\pi$ and $\pi(I_8)=O_8$.
\end{Notation/Lemma}

\begin{Lemma}
Let $V^{\bot}$ denote the orthogonal complement of $V$ in $V_0$
 with respect to the trace form $\kappa$, i.e.
$V^\bot:= \{B\in V_0  \mid  \kappa(A,B)=0 {\ } \text{ for all }  A \in V\}$,
 and
 \begin{align*}
  W:=& \bigoplus_{(i,j)\notin J}{\mathbb{F}_{q}e_{ij}}
+\{x(e_{15}-e_{14}) \mid  x\in \mathbb{F}_{q}\}\\
=& \{ A=(A_{ij})\in V_0  \mid  A_{12}=A_{13}=A_{16}=A_{17}=A_{23}=0,{\ }A_{14}=-A_{15}  \}.
\end{align*}
Then $W =V^\bot$.
\end{Lemma}

\begin{Lemma} \label{nond, G2}
$\kappa|_{V\times V}$ is a non-degenerate $\mathbb{F}_q$-bilinear form.
\end{Lemma}

\begin{Corollary}\label{pi proj,G2}
 $V_0=V\oplus V^\bot$,
 and $\pi \colon V_0 \to  V$ is the projective map to the first component.
\end{Corollary}

\begin{Corollary}\label{pi=piJ-kappa, G2}
If $A,B\in V_0$ and $\pi_J(A)\in V$, then $\pi(A)=\pi_J(A)$.
If $\mathrm{supp}(A)\cap\mathrm{supp}(B)\subseteq J$,
then
$ \kappa(A,B)=\kappa(\pi(A),B)=\kappa(A,\pi(B))
= \kappa(\pi(A),\pi(B))=\kappa|_{V\times V}(\pi(A),\pi(B))$.
\end{Corollary}

\begin{Lemma}\label{pi_J,AgT-intersection in J,G2}
If  $A\in V$ and $g,h\in G$, then
$ \pi_J(Ag^\top)\in V$
and
$\mathrm{supp}(Bh^\top)\cap \mathrm{supp}(Ag)\subseteq  J$.
In particular, $\pi_J(Ag^\top)=\pi(Ag^\top)$.
\end{Lemma}

\begin{Proposition}[Group action of $G$ on $V$]\label{circ action-G2}
The map
\begin{align*}
-\circ- \colon  V \times G \to  V: (A,g)\mapsto A\circ g:=\pi(Ag)
\end{align*}
is a group action,
and the elements of the group $G$ act as
$\mathbb{F}_q$-automorphisms.
\end{Proposition}

\begin{Corollary}\label{circ g-circ gT-G2}
If $A,B\in V$ and $g\in G$,
then
$\kappa(A, B\circ g)=\kappa(A, Bg)=\kappa(Ag^\top, B)=\kappa(\pi(A g^\top), B)$.
\end{Corollary}

Let $A.g$ ($A\in V$, $g\in G$) denote $\pi(A g^{-\top})$.
Then this is a group action of $G$ by \ref{circ action-G2}.
By \cite[\S 2.1]{Markus1}, we obtain a new action:
\begin{Corollary}\label{action A dot g-G2}
There exists an unique linear action $-.-$ of $G$ on $V$:
\begin{align*}
-.- \colon V\times G  \to  V: (A,g)\mapsto A.g:=\pi(A g^{-\top})
\end{align*}
such that
$\kappa|_{V\times V}(A.g,B)=\kappa|_{V\times V}(A,B\circ g^{-1})$
 for all  $B\in V$.
\end{Corollary}
\begin{Notation}  \label{f-G2}
Set $f:=\pi|_G \colon G \to  V$.
\end{Notation}

\begin{Lemma}\label{f(x)g,G2}
Let $x, g \in G$
and $1:=I_8$.
Then $f(x)g\equiv(x-1)g \mod V^\bot$.
In particular, $f(x)\equiv x-1 \mod V^\bot$.
\end{Lemma}

\begin{Proposition}\label{f(xg)-G2}
If $x, g \in G$, then
$ f(xg)=f(x)\circ g+f(g)$.
\end{Proposition}

\begin{proof}
For all $x, g \in G$,
$ f(xg)\stackrel{\ref{f(x)g,G2}}{\equiv} xg-1
                                    = (x-1)g+(g-1)
      \stackrel{\ref{f(x)g,G2}}{\equiv} f(x)g+f(g) \mod V^\bot$,
so $\pi(f(xg))= \pi(f(x)g+f(g))$ by \ref{pi proj,G2}.
Thus $f(xg)= \pi(f(x)g)+\pi(f(g))=f(x)\circ g + f(g)$.
\end{proof}

\begin{Proposition}[Bijective 1-cocycle of ${G_2^{syl}(q)}$]\label{f_U bij-G2}
\index{1-cocycle!-of ${G_2^{syl}(q)}$}
If $U={G_2^{syl}(q)}$,
 then $f|_U:=\pi|_U \colon  U \to  V $ is a bijection.
In particular, $f|_U$ is a bijective 1-cocycle of $U$.
\end{Proposition}

\begin{Corollary}[Monomial linearisation for $G_8(q)$]
\label{mono. lin.-G2}
 The map $f=\pi|_G :G \to  V$ is a surjective 1-cocycle of $G$ in $V$,
 and  $(f,{\kappa}|_{V\times V})$ is a monomial linearisation for $G=G_8(q)$.
\end{Corollary}

\begin{Corollary}
\label{mono. lin. for U-G2}
\index{monomial linearisation!-for ${G_2^{syl}(q)}$}
 $(f|_{G_2^{syl}(q)},\kappa|_{V\times V})$
is a monomial linearisation
for $G_2^{syl}(q)$.
\end{Corollary}

Now we determine the monomial $G$-module $\mathbb{C}{\left(G_2^{syl}(q)\right)}$,
which is essential for the construction of the supercharacter theory for ${G_2^{syl}(q)}$.
\begin{Theorem}[Fundamental theorem for $G_2^{syl}(q)$]\label{fund thm U-G2}
Let $G=G_8(q)$, $U=G_2^{syl}(q)$ and
 \begin{equation*}
  [A]=\frac{1}{|U|}\sum_{u\in U}{\overline{\chi_A(u)}u}
  \qquad \text{for all $A\in V$}.
 \end{equation*}
 where $\chi_{A}(u)=\vartheta\kappa(A,f(u))$.
Then the set $\{[A] \mid A\in V\}$
forms a $\mathbb{C}$-basis for the complex group algebra $\mathbb{C}U$.
For all $g\in G,{\ }A\in V$,
let
$[A]*g:=\chi_{A.g}(g)[A.g]=\vartheta\kappa(A.g,f(g))[A.g]$.
Then $\mathbb{C}U$ is a monomial $\mathbb{C}G$-module.
The restriction of the $*$-operation to $U$ is given by the usual right multiplication
  of $U$ on $\mathbb{C}U$, i.e.
  \begin{align*}
 [A]*u=[A]u=\frac{1}{|U|}\sum_{y\in U}{\overline{\chi_A(y)}yu}
 \qquad \text{for all $u\in U$, $A\in V$}.
\end{align*}

\end{Theorem}
\begin{proof}
By \ref{mono. lin.-G2}, $(f,\kappa|_{V\times V})$ is a monomial linearisation for $G$,
 satisfying that $f|_U$ is a bijective map.
By \ref{action A dot g-G2}, $A.u:=\pi(Au^{-\top})$.
Hence the theorem is proved by \cite[2.1.35]{Markus1}.
\end{proof}

\begin{Comparison}[Monomial linearisations]
\label{com:monomial modules-G2}
Let $U$ be
$A_n(q)$, $D_n^{syl}(q)$, ${^3}D_4^{syl}(q^3)$
or $G_2^{syl}(q)$,
$G$ an intermediate group of $U$,
$V_0:=V_{\Squarb}$,
$V$ a subspace of $V_0$,
$J:=\mathrm{supp}(V)$,
$f\colon G\to V$ a surjective 1-cocycle of $G$
such that $f|_U$ is injective,
$\kappa\colon V \times V\to \mathbb{F}_q \text{ (or $\mathbb{F}_{q^3}$)}$
a trace form
such that
$(f,\kappa|_{V\times V})$ is a monomial linearisation for $G$
(i.e. $(f|_U,\kappa|_{V\times V})$ is a monomial linearisation for $U$).
Then
the corresponding notations for
$A_n(q)$ (see \cite[2.2]{Markus1}),
$D_n^{syl}(q)$ (see \cite[3.1]{Markus1}),
${^3}D_4^{syl}(q^3)$ (see \cite[\S 4]{sun3D4super}),
and
$G_2^{syl}(q)$ (see \S \ref{sec: monomial G2-module})
are listed as follows:
\begin{align*}
\begin{array}{|l|l|l|l|l|l|l|}
\hline
\multicolumn{1}{|c|}{U}
& \multicolumn{1}{c|}{G}
& \multicolumn{1}{c|}{V_0}
& \multicolumn{1}{c|}{J}
& \multicolumn{1}{c|}{V}
& \multicolumn{1}{c|}{f\colon G\to V}
& \multicolumn{1}{c|}{\kappa|_{V\times V}} \\\hline
A_n(q)
& A_{n}(q)
& \mathrm{Mat}_{n\times n}(q)
& \UR
& V=V_{\URb}
& f(g)=\pi_{\URb}(g)=g-I_n
& \kappa|_{V\times V}
\\\hline
D_n^{syl}(q)
& A_{2n}(q)
& \mathrm{Mat}_{2n\times 2n}(q)
& {\UP}
& V=V_{\UPb}
& f(g)=\pi_{\UPb}(g)
& \kappa|_{V\times V}
\\\hline
{^3}D_4^{syl}(q^3)
& G_8(q^3)
& \mathrm{Mat}_{8\times 8}(q^3)
& J
& V\neq V_J
& f(g)=\pi(g)\neq \pi_J(g)
& \kappa_q|_{V\times V}
\\\hline
G_2^{syl}(q)
& G_8(q)
& \mathrm{Mat}_{8\times 8}(q)
& J
& V\neq V_J
& f(g)=\pi(g)\neq \pi_J(g)
& \kappa|_{V\times V}
\\\hline
\end{array}
\end{align*}
\end{Comparison}

From now on, we mainly consider the regular right module $(\mathbb{C}U,*)_{\mathbb{C}U}=\mathbb{C}U_{\mathbb{C}U}$.


\section{$G_2^{syl}(q)$-orbit modules}
\label{sec:U-orbit modules-G2}

Let $U:=G_2^{syl}(q)$,
$A_{ij}\in \mathbb{F}_q$,
$A_{ij}^*\in \mathbb{F}_q^*$ $(1\leq i, j \leq 8)$,
and
$t_i\in \mathbb{F}_q$,
$t_i^*\in \mathbb{F}_q^*$ $(i=1,2,\dots,6)$.
In this section,
we classify the $U$-orbit modules (\ref{prop:class orbit-G2}),
and obtain the stabilizers $\mathrm{Stab}_U(A)$ for all $A\in V$ (\ref{prop: G2-stab}).

 Let $A\in V$, the $U$\textbf{-orbit module}
 \index{orbit module}
 associated to $A$ is
$\mathbb{C}\mathcal{O}_U([A])
:=\mathbb{C}\{[A]u \mid u\in U\}
=\mathbb{C}\{[A.u] \mid u\in U\}$.
Then $\mathbb{C}\mathcal{O}_U([A])$ has a $\mathbb{C}$-basis
$
\{[A.u] \mid  u\in U\}
          = \left\{[C] \mid C\in \mathcal{O}_U(A)\right\}$,
where
$\mathcal{O}_U(A):=\left\{A.g \mid  g\in U\right\}$
is the \textbf{orbit} of $A$ under the operation $-.-$ defined in \ref{action A dot g-G2}.
The \textbf{stabilizer} \index{stabilizer!-of $A$}
 $\mathrm{Stab}_U(A)$ of $A$ in $U$ is defined to be
$\mathrm{Stab}_U(A)=\{u\in U  \mid  A.u=A\}$,
then
 $\mathrm{dim}_{\mathbb{C}}\mathbb{C}\mathcal{O}_U([A])
 =|\mathcal{O}_U(A)|
 =\frac{|U|}{|\mathrm{Stab}_U(A)|}$.
If $A,B\in V$, then
$\mathbb{C}\mathcal{O}_U([A])$ and $\mathbb{C}\mathcal{O}_U([B])$
are identical (if $A.u=B$ for some $u\in U$)
or their intersection is $\{0\}$.
 Two $\mathbb{C}U$-modules having no nontrivial
 $\mathbb{C}U$-homomorphism between them are called
 \textbf{orthogonal}.

\begin{Lemma}\label{G2-A.xi, figures}
Let $A\in V$, $y_i(t_i)\in U$ and $t_i \in \mathbb{F}_{q}$ with $i\in\{1,2,\dots,6\}$.
Then $A.y_i(t_i)$  and
the corresponding figures of moves
are obtained as follows:
\begin{alignat*}{3}
A.y_1(t_1)=& A.(x_{34}(t_1)x_{35}(t_1)),
& \qquad
A.y_3(t_3)=& A.(x_{24}(t_3)x_{25}(t_3)),
& \qquad
A.y_4(t_4)=& A.x_{26}(t_4),\\
A.y_2(t_2)=& A.x_{23}(t_2),
& \qquad
A.y_5(t_5)=& A,
& \qquad
A.y_6(t_6)=& A.
\end{alignat*}
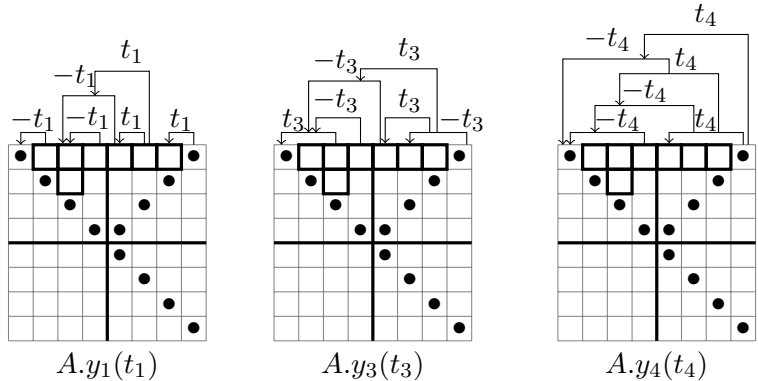
\begin{align*}
\begin{tikzpicture}[scale=0.65]
\draw[step=0.5cm, gray, very thin](-2, -2)grid(2, 2);
\draw[very thick] (-2,0)--(2,0);
\draw[very thick] (0,-2)--(0,2);
\draw (-1.75,1.75) node{$\bullet$};
\draw (-1.25,1.25) node{$\bullet$};
\draw (-0.75,0.75) node{$\bullet$};
\draw (-0.25,0.25) node{$\bullet$};
\draw (0.25,-0.25) node{$\bullet$};
\draw (0.75,-0.75) node{$\bullet$};
\draw (1.25,-1.25) node{$\bullet$};
\draw (1.75,-1.75) node{$\bullet$};
\draw (1.75,1.75) node{$\bullet$};
\draw (1.25,1.25) node{$\bullet$};
\draw (0.75,0.75) node{$\bullet$};
\draw (0.25,0.25) node{$\bullet$};
\foreach \x in {-1.5, -1,...,1.0}
{
\draw [very thick] (\x,1.5)   rectangle +(0.5,0.5);
}
\draw [very thick] (-1.0,1.0) rectangle +(0.5,0.5);
\draw[->] (-1.25,2)--(-1.25, 2.25)--(-1.75,2.25)--(-1.75,2);
\draw (-1.5,2.5) node{$-t_1$};
\draw[->] (1.75,2)--(1.75, 2.25)--(1.25,2.25)--(1.25,2.03);
\draw(1.5,2.5) node{$t_1$};
\draw[->] (-0.15,2)--(-0.15, 2.25)--(-0.75,2.25)--(-0.75,2.03);
\draw(-0.4,2.55) node{$-t_1$};
\draw[->] (0.75,2)--(0.75, 2.25)--(0.25,2.25)--(0.25,2.03);
\draw(0.5,2.55) node{$t_1$};
\draw[->] (0.15,2.0)--(0.15, 3)--(-0.9,3)--(-0.9,2.03);
\draw(-0.65,3.4) node{$-t_1$};
\draw[->] (0.85,2.0)--(0.85, 3.5)--(-0.25,3.5)--(-0.25,3);
\draw(0.5,3.9) node{$t_1$};
\draw(0,-2.5) node{$A.y_1(t_1)$};
\end{tikzpicture}
\qquad
\begin{tikzpicture}[scale=0.65]
\draw[step=0.5cm, gray, very thin](-2, -2)grid(2, 2);
\draw[very thick] (-2,0)--(2,0);
\draw[very thick] (0,-2)--(0,2);
\draw (-1.75,1.75) node{$\bullet$};
\draw (-1.25,1.25) node{$\bullet$};
\draw (-0.75,0.75) node{$\bullet$};
\draw (-0.25,0.25) node{$\bullet$};
\draw (0.25,-0.25) node{$\bullet$};
\draw (0.75,-0.75) node{$\bullet$};
\draw (1.25,-1.25) node{$\bullet$};
\draw (1.75,-1.75) node{$\bullet$};
\draw (1.75,1.75) node{$\bullet$};
\draw (1.25,1.25) node{$\bullet$};
\draw (0.75,0.75) node{$\bullet$};
\draw (0.25,0.25) node{$\bullet$};
\foreach \x in {-1.5, -1,...,1.0}
{
\draw [very thick] (\x,1.5)   rectangle +(0.5,0.5);
}
\draw [very thick] (-1.0,1.0) rectangle +(0.5,0.5);
\draw[->] (-0.75,2)--(-0.75, 2.25)--(-1.85,2.25)--(-1.85,2);
\draw (-1.6,2.5) node{$t_3$};
\draw[->] (1.9,2)--(1.9, 2.25)--(0.75,2.25)--(0.75,2.03);
\draw(1.8,2.5) node{$-t_3$};
\draw[->] (-0.25,2)--(-0.25, 2.55)--(-1.15,2.55)--(-1.15,2.25);
\draw(-0.75,2.9) node{$-t_3$};
\draw[->] (1.15,2.25)--(1.15, 2.55)--(0.25,2.55)--(0.25,2.03);
\draw(0.75,2.9) node{$t_3$};
\draw[->] (0.15,2.0)--(0.15, 3.3)--(-1.3,3.3)--(-1.3,2.25);
\draw(-0.75,3.7) node{$-t_3$};
\draw[->] (1.3,2.25)--(1.3, 3.55)--(-0.25,3.55)--(-0.25,3.3);
\draw(0.75,3.95) node{$t_3$};
\draw(0,-2.5) node{$A.y_3(t_3)$};
\end{tikzpicture}
\qquad
\begin{tikzpicture}[scale=0.65]
\draw[step=0.5cm, gray, very thin](-2, -2)grid(2, 2);
\draw[very thick] (-2,0)--(2,0);
\draw[very thick] (0,-2)--(0,2);
\draw (-1.75,1.75) node{$\bullet$};
\draw (-1.25,1.25) node{$\bullet$};
\draw (-0.75,0.75) node{$\bullet$};
\draw (-0.25,0.25) node{$\bullet$};
\draw (0.25,-0.25) node{$\bullet$};
\draw (0.75,-0.75) node{$\bullet$};
\draw (1.25,-1.25) node{$\bullet$};
\draw (1.75,-1.75) node{$\bullet$};
\draw (1.75,1.75) node{$\bullet$};
\draw (1.25,1.25) node{$\bullet$};
\draw (0.75,0.75) node{$\bullet$};
\draw (0.25,0.25) node{$\bullet$};
\foreach \x in {-1.5, -1,...,1.0}
{
\draw [very thick] (\x,1.5)   rectangle +(0.5,0.5);
}
\draw [very thick] (-1.0,1.0) rectangle +(0.5,0.5);
\draw[->] (-0.25,2)--(-0.25, 2.25)--(-1.75,2.25)--(-1.75,2.0);
\draw (-0.8,2.5) node{$-t_4$};
\draw[->] (1.75,2)--(1.75, 2.25)--(0.25,2.25)--(0.25,2.03);
\draw(1.0,2.5) node{$t_4$};
\draw[->] (0.75,2.25)--(0.75, 2.8)--(-1.25,2.8)--(-1.25,2.25);
\draw(-0.25,3.1) node{$-t_4$};
\draw[->] (1.25,2.25)--(1.25, 3.45)--(-0.75,3.45)--(-0.75,2.8);
\draw(0.6,3.75) node{$t_4$};
\draw[->] (0.25,3.45)--(0.25, 3.75)--(-1.9,3.75)--(-1.9,2);
\draw(-1,4.15) node{$-t_4$};
\draw[->] (1.85,2)--(1.85, 4.25)--(-0.25,4.25)--(-0.25,3.75);
\draw(1,4.65) node{$t_4$};
\draw(0,-2.5) node{$A.y_4(t_4)$};
\end{tikzpicture}
\end{align*}
\begin{align*}
\begin{tikzpicture}[scale=0.65]
\draw[step=0.5cm, gray, very thin](-2, -2)grid(2, 2);
\draw[very thick] (-2,0)--(2,0);
\draw[very thick] (0,-2)--(0,2);
\draw (-1.75,1.75) node{$\bullet$};
\draw (-1.25,1.25) node{$\bullet$};
\draw (-0.75,0.75) node{$\bullet$};
\draw (-0.25,0.25) node{$\bullet$};
\draw (0.25,-0.25) node{$\bullet$};
\draw (0.75,-0.75) node{$\bullet$};
\draw (1.25,-1.25) node{$\bullet$};
\draw (1.75,-1.75) node{$\bullet$};
\draw (1.75,1.75) node{$\bullet$};
\draw (1.25,1.25) node{$\bullet$};
\draw (0.75,0.75) node{$\bullet$};
\draw (0.25,0.25) node{$\bullet$};
\foreach \x in {-1.5, -1,...,1.0}
{
\draw [very thick] (\x,1.5)   rectangle +(0.5,0.5);
}
\draw [very thick] (-1.0,1.0) rectangle +(0.5,0.5);
\draw[->] (-0.75,2)--(-0.75, 2.25)--(-1.25,2.25)--(-1.25,2.03);
\draw (-1.0,2.55) node{$-t_2$};
\draw[->] (1.25,2)--(1.25, 2.25)--(0.75,2.25)--(0.75,2.03);
\draw(1.0,2.55) node{$t_2$};
\draw(0,-2.5) node{$A.y_2(t_2)$};
\end{tikzpicture}
\qquad
\begin{tikzpicture}[scale=0.65]
\draw[step=0.5cm, gray, very thin](-2, -2)grid(2, 2);
\draw[very thick] (-2,0)--(2,0);
\draw[very thick] (0,-2)--(0,2);
\draw (-1.75,1.75) node{$\bullet$};
\draw (-1.25,1.25) node{$\bullet$};
\draw (-0.75,0.75) node{$\bullet$};
\draw (-0.25,0.25) node{$\bullet$};
\draw (0.25,-0.25) node{$\bullet$};
\draw (0.75,-0.75) node{$\bullet$};
\draw (1.25,-1.25) node{$\bullet$};
\draw (1.75,-1.75) node{$\bullet$};
\draw (1.75,1.75) node{$\bullet$};
\draw (1.25,1.25) node{$\bullet$};
\draw (0.75,0.75) node{$\bullet$};
\draw (0.25,0.25) node{$\bullet$};
\foreach \x in {-1.5, -1,...,1.0}
{
\draw [very thick] (\x,1.5)   rectangle +(0.5,0.5);
}
\draw [very thick] (-1.0,1.0) rectangle +(0.5,0.5);
\draw[->] (0.75,2)--(0.75, 2.25)--(-1.75,2.25)--(-1.75,2);
\draw (-0.25,2.5) node{$-t_5$};
\draw[->] (1.75,2)--(1.75, 2.85)--(-0.75,2.85)--(-0.75,2.25);
\draw(0.5,3.1) node{$t_5$};
\draw(0,-2.5) node{$A.y_5(t_5)$};
\end{tikzpicture}
\qquad
\begin{tikzpicture}[scale=0.65]
\draw[step=0.5cm, gray, very thin](-2, -2)grid(2, 2);
\draw[very thick] (-2,0)--(2,0);
\draw[very thick] (0,-2)--(0,2);
\draw (-1.75,1.75) node{$\bullet$};
\draw (-1.25,1.25) node{$\bullet$};
\draw (-0.75,0.75) node{$\bullet$};
\draw (-0.25,0.25) node{$\bullet$};
\draw (0.25,-0.25) node{$\bullet$};
\draw (0.75,-0.75) node{$\bullet$};
\draw (1.25,-1.25) node{$\bullet$};
\draw (1.75,-1.75) node{$\bullet$};
\draw (1.75,1.75) node{$\bullet$};
\draw (1.25,1.25) node{$\bullet$};
\draw (0.75,0.75) node{$\bullet$};
\draw (0.25,0.25) node{$\bullet$};
\foreach \x in {-1.5, -1,...,1.0}
{
\draw [very thick] (\x,1.5)   rectangle +(0.5,0.5);
}
\draw [very thick] (-1.0,1.0) rectangle +(0.5,0.5);
\draw[->] (1.25,2)--(1.25, 2.25)--(-1.75,2.25)--(-1.75,2);
\draw (-0.25,2.5) node{$-t_6$};
\draw[->] (1.75,2)--(1.75, 2.85)--(-1.25,2.85)--(-1.25,2.25);
\draw(0.25,3.1) node{$t_6$};
\draw(0,-2.5) node{$A.y_6(t_6)$};
\end{tikzpicture}
\end{align*}
\end{Lemma}
These figures describe the way of classifying the orbits.

\begin{Lemma}[$G_2^{syl}(q)$-orbit modules]\label{prop: G2-orbit}
For $A \in V$,
the $U$-orbit module $\mathbb{C}\mathcal{O}_U([A])$ $(A\in V)$
is obtained as follows:
\begin{align*}
&\mathbb{C}\mathcal{O}_U([A])\\
=& \mathbb{C} \Big\{
{%
\left[
\newcommand{\mc}[3]{\multicolumn{#1}{#2}{#3}}
\begin{array}{cccccccc}
\cline{2-7}
\mc{1}{c|}{}
& \mc{1}{c|}{
\begin{array}{l}
 A_{12}\\
-A_{13}t_2\\
-2A_{15}t_3\\
-2A_{16}t_1t_3\\
-2A_{17}t_2t_1t_3\\
-A_{17}t_3^2\\
-A_{16}t_4\\
-A_{17}t_2t_4\\
\end{array}
}
& \mc{1}{c|}{
\begin{array}{l}
A_{13}\\
-2A_{15}t_1\\
-A_{16}t_1^2\\
-A_{17}t_2t_1^2\\
+A_{17}t_4\\
\end{array}
}
& \mc{1}{c|}{
\begin{array}{l}
A_{15}\\
+A_{16}t_1\\
+A_{17}t_2t_1\\
+A_{17}t_3\\
\end{array}
}
& \mc{1}{c|}{
\begin{array}{l}
A_{15}\\
+A_{16}t_1\\
+A_{17}t_2t_1\\
+A_{17}t_3\\
\end{array}}
& \mc{1}{c|}{
\begin{array}{l}
A_{16}\\
+A_{17}t_2\\
\end{array}
}
& \mc{1}{c|}{A_{17}}
& \\\cline{2-7}
× & \mc{1}{c|}{} & \mc{1}{c|}{A_{23}} & × & × & × &  & ×\\\cline{3-3}
  \end{array}
\right]} \\
& \phantom{\mathbb{C} \Big\{}
{\ }\Big|{\ }
t_1, t_2, t_3, t_4\in \mathbb{F}_{q}\Big\}.
\end{align*}
\end{Lemma}
\begin{proof}
By \ref{G2-A.xi, figures}, we calculate the orbit modules directly.
\end{proof}

 The elements of $V$  are called \textbf{patterns}.
 The monomial action of $G$ on $\mathbb{C}U$: $([A],g)\mapsto [A]*g$
 (e.g.  \ref{fund thm U-G2})
 and also the corresponding permutation operation on $V$: $(A,g)\mapsto A.g$
 (e.g. \ref{action A dot g-G2})
 are called
 \textbf{truncated column operation}.
Let $A\in V$.
Then
 $(i,j)\in J$ is a \textbf{main condition} \index{main condition} of $A$
if and only if $A_{ij}$ is the rightmost non-zero entry in the $i$-th row.
We set
$\mathrm{main}(A):=
 \{ (i,j)\in J  \mid  (i,j) \text{ is a main condition of }A\}$.
 The coordinate $(i,j)$ is called \textbf{the $i$-th main condition},
if $(i,j)\in \mathrm{main}(A)$. Set
$
\mathrm{main}_i(A):= \{ (i,j)\in J \mid (i,j) \text{ is the $i$-th main condition of }A\}$.
Let $A\in V$ be a pattern.
Then $A$ is a \textbf{staircase pattern},
if the elements in $\mathrm{main}(A)$ lie in different columns.
Analogously, a $U$-orbit module $\mathbb{C}\mathcal{O}_U([A])$ is called a \textbf{staircase $U$-module},
if the elements in $\mathrm{main}(A)$ lie in different columns.
The \textbf{verge} of $A\in V$ is
$
  \mathrm{verge}(A):=
  \sum_{(i,j)\in \mathrm{main}(A)}{A_{i,j}e_{i,j}}$.
The \textbf{$i$-th verge} of $A$ is
$\mathrm{verge}_i(A):=
  \sum_{(i,k)\in \mathrm{main}_i (A)}{A_{i,k}e_{i,k}}$.
The (staircase) pattern $A\in V$ is called the \textbf{(staircase) verge pattern},
if $A=\mathrm{verge}(A)$.
A \textbf{minor condition}
 \index{condition! minor-}
 of $A\in V$ is $(i,j)\in J$ $(j\leq 4)$,
if $(i,9-j)$ is a main condition.
Set
$\mathrm{minor}(A):=
 \{(i,j)\in J  \mid (i,j) \text{ is a minor condition of }A\}\subseteq J$.
The \textbf{core}
of $A\in V$ is denoted by
$\mathrm{core}(A):=\mathrm{main}(A) \dot{\cup} \mathrm{minor}(A)$.
A (staircase) pattern $A\in V$ is a \textbf{(staircase) core pattern}
if $\mathrm{supp}(A) \subseteq \mathrm{core}(A)$.

\begin{Notation}
Define the families of $U$-orbit modules as follows:
$\mathfrak{F}_6:= \{\mathbb{C}\mathcal{O}_U(A) \mid  A\in V,{\,}A_{17}\neq 0\}$,
$\mathfrak{F}_5:= \{\mathbb{C}\mathcal{O}_U(A) \mid A\in V,{\,} A_{16}\neq 0,{\,} A_{17}= 0\}$,
$\mathfrak{F}_4:= \{\mathbb{C}\mathcal{O}_U(A) \mid A\in V,{\,} A_{15}\neq 0,{\,} A_{16}=A_{17}= 0\}$,
$\mathfrak{F}_3:= \{\mathbb{C}\mathcal{O}_U(A) \mid A\in V,{\,} A_{13}\neq 0,{\,} A_{15}=A_{16}=A_{17}= 0\}$,
and
$ \mathfrak{F}_{1,2}:= \{\mathbb{C}\mathcal{O}_U(A) \mid A\in V,{\,} A_{13}= A_{15}=A_{16}=A_{17}= 0\}$.
For $A\in V$, we also say $A\in \mathfrak{F}_i$,
if $\mathbb{C}\mathcal{O}_U([A])\in \mathfrak{F}_i$.
\end{Notation}
\begin{Proposition}[Classification of $G_2^{syl}(q)$-orbit modules]\label{prop:class orbit-G2}
Every $U$-orbit module is contained in one of the families
$\{\mathfrak{F}_{1,2},\mathfrak{F}_{3},\mathfrak{F}_{4}, \mathfrak{F}_{5},\mathfrak{F}_{6}\}$,
and
\begin{align*}
\mathfrak{F}_6=& \{\mathbb{C}\mathcal{O}_U([A_{12}e_{12}+A_{23}e_{23}+A_{17}^*e_{17}])
         \mid A_{12},A_{23}\in \mathbb{F}_{q}, {\ } A_{17}^*\in \mathbb{F}_{q}^*\},\\
\mathfrak{F}_5=& \{\mathbb{C}\mathcal{O}_U([A_{13}e_{13}+A_{23}e_{23}+A_{16}^*e_{16}])
         \mid A_{13},A_{23}\in \mathbb{F}_{q}, {\ } A_{6}^*\in \mathbb{F}_{q}^*\},\\
\mathfrak{F}_4=& \{\mathbb{C}\mathcal{O}_U([A_{23}e_{23}+{A_{15}^*}(e_{14}+e_{15})])
         \mid A_{23}\in \mathbb{F}_{q}, {\ } A_{15}^*\in \mathbb{F}_{q}^*\},\\
\mathfrak{F}_3=& \{\mathbb{C}\mathcal{O}_U([A_{23}e_{23}+A_{13}^*e_{13}])
         \mid A_{23}\in \mathbb{F}_{q}, {\ } A_{13}^*\in \mathbb{F}_{q}^*\},\\
\mathfrak{F}_{1,2}=& \{\mathbb{C}\mathcal{O}_U([A_{12}e_{12}+A_{23}e_{23}])
         \mid A_{12},A_{23}\in \mathbb{F}_{q}\}.
\end{align*}
The dimensions of $U$-orbit modules are determined.
In particular, every $U$-orbit module of
families $\mathfrak{F}_{1,2}$, $\mathfrak{F}_{4}$, $\mathfrak{F}_{5}$ and $\mathfrak{F}_{6}$
contains one and only one staircase core pattern,
and
every $U$-orbit module of family $\mathfrak{F}_{3}$
contains precisely one core pattern.
\end{Proposition}
\begin{proof}
Let $A=(A_{ij})\in V$ with $A_{17}=A_{17}^*\in \mathbb{F}_q^*$.
Then
\begin{align*}
&\mathbb{C}\mathcal{O}_U([A])
=\mathbb{C} \left\{
{%
\left[
\newcommand{\mc}[3]{\multicolumn{#1}{#2}{#3}}
\begin{array}{cccccccc}
\cline{2-7}
\mc{1}{c|}{}
& \mc{1}{c|}{
\begin{array}{l}
 A_{12}\\
+\frac{A_{13}A_{16}+A_{15}^2}{A_{17}^*}\\
-\frac{B_{13}B_{16}+B_{15}^2}{A_{17}^*}
\end{array}
}
& \mc{1}{c|}{B_{13}}
& \mc{1}{c|}{B_{15}}
& \mc{1}{c|}{B_{15}}
& \mc{1}{c|}{B_{16}}
& \mc{1}{c|}{A^*_{17}}
& \\\cline{2-7}
× & \mc{1}{c|}{} & \mc{1}{c|}{A_{23}} & × & × & × &  & ×\\\cline{3-3}
  \end{array}
\right]}
{\, }\middle|{\, }
B_{13},B_{15}, B_{16} \in \mathbb{F}_{q}\right\}.
\end{align*}
Thus $\mathrm{dim}_{\mathbb{C}}\mathbb{C}\mathcal{O}_U([A])=q^3$.
Let $u:=y(t_1,
-\frac{A_{16}}{A_{17}^*},
-\frac{A_{15}}{A_{17}^*},
-\frac{A_{13}-2A_{15}t_1}{A_{17}^*},
t_5,t_6)\in U $.
Then there is a staircase core pattern
\begin{align*}
C:=A.u=
 {\newcommand{\mc}[3]{\multicolumn{#1}{#2}{#3}}
\begin{array}{cccccccc}\cline{2-7}
\mc{1}{c|}{}
& \mc{1}{c|}{
\rule{0pt}{15pt} A_{12}+\frac{A_{13}A_{16}+A_{15}^2}{A_{17}^*}}
& \mc{1}{c|}{} & \mc{1}{c|}{}
                 & \mc{1}{c|}{} & \mc{1}{c|}{} & \mc{1}{c|}{A^*_{17}} & \\\cline{2-7}
× & \mc{1}{c|}{} & \mc{1}{c|}{A_{23}} & × & × & × &  & ×\\\cline{3-3}
  \end{array}}
  \in \mathcal{O}_U(A).
\end{align*}
Thus $\mathcal{O}_U(C)=\mathcal{O}_U(A)$
and $\mathbb{C}\mathcal{O}_U([A])=\mathbb{C}\mathcal{O}_U([C])$.
Since $C$ only depends on $A$,
the staircase core pattern is determined uniquely.
Thus
\begin{align*}
\mathfrak{F}_6=\{\mathbb{C}\mathcal{O}_U([D_{12}e_{12}+D_{23}e_{23}+D_{17}^*e_{17}])
         \mid D_{12},D_{23}\in \mathbb{F}_{q}, {\,} D_{17}^*\in \mathbb{F}_{q}^*\}.
\end{align*}
Similarly, all of the statements are proved.
\end{proof}

\begin{Remark}\label{G2-staircase, F3}
 Let $A\in V$.
In \ref{prop:class orbit-G2},
the orbit modules $\mathbb{C}\mathcal{O}_U([A])$
are staircase modules
except that $\mathbb{C}\mathcal{O}_U([A]) \subseteq \mathfrak{F}_3$
when $A_{2,3}\neq 0$.
\end{Remark}

\begin{Proposition}[$G_2^{syl}(q)$-stabilizer]\label{prop: G2-stab}
If $A\in V$,
then $\mathrm{Stab}_U(A)$ is established in Table \ref{table:G2-stab}.
\begin{table}[!htp]
\caption{$G_2^{syl}(q)$-stabilizers}
\label{table:G2-stab}
\begin{align*}
 \begin{array}{|c|l|l|}
\hline
& \multicolumn{1}{c|}{A\in V}
& \multicolumn{1}{c|}{\mathrm{Stab}_U(A)} \\\hline
\mathfrak{F}_6 &
{%
\newcommand{\mc}[3]{\multicolumn{#1}{#2}{#3}}
\begin{array}{cccccccc}\cline{2-7}
\mc{1}{c|}{} & \mc{1}{c|}{A_{12}} & \mc{1}{c|}{A_{13}} & \mc{1}{c|}{A_{15}}
                        & \mc{1}{c|}{A_{15}} & \mc{1}{c|}{A_{16}} & \mc{1}{c|}{A^*_{17}} & \\\cline{2-7}
× & \mc{1}{c|}{} & \mc{1}{c|}{A_{23}} & × & × & × &  & ×\\\cline{3-3}
\end{array}
}%
&
\rule{0pt}{21pt}
\begin{array}{c}
y(t_1,0,-\frac{A_{16}t_1}{A^*_{17}},\frac{2A_{15}t_1+A_{16}t_1^{2}}{A^*_{17}},t_5,t_6)\\
\forall{\ }t_1,t_5,t_6\in \mathbb{F}_{q}
\end{array}
\\\hline
\mathfrak{F}_5 &
{%
\newcommand{\mc}[3]{\multicolumn{#1}{#2}{#3}}
\begin{array}{cccccccc}\cline{2-7}
\mc{1}{c|}{} & \mc{1}{c|}{A_{12}} & \mc{1}{c|}{A_{13}} & \mc{1}{c|}{A_{15}}
                        & \mc{1}{c|}{A_{15}} & \mc{1}{c|}{A_{16}^*} & \mc{1}{c|}{×} & \\\cline{2-7}
× & \mc{1}{c|}{} & \mc{1}{c|}{A_{23}} & × & × & × &  & ×\\\cline{3-3}
  \end{array}
}%
&
\rule{0pt}{21pt}
\begin{array}{c}
y(0,t_2,t_3,\frac{-A_{13}t_2-2A_{15}t_3}{A^*_{16}},t_5,t_6)\\
\forall{\ }t_2,t_3,t_5,t_6\in \mathbb{F}_{q}
\end{array}
\\\hline
\mathfrak{F}_4 &
{%
\newcommand{\mc}[3]{\multicolumn{#1}{#2}{#3}}
\begin{array}{cccccccc}\cline{2-7}
\mc{1}{c|}{} & \mc{1}{c|}{A_{12}} & \mc{1}{c|}{A_{13}} & \mc{1}{c|}{A^*_{15}}
                        & \mc{1}{c|}{A_{15}^*} & \mc{1}{c|}{×} & \mc{1}{c|}{×} & \\\cline{2-7}
× & \mc{1}{c|}{} & \mc{1}{c|}{A_{23}} & × & × & × &  & ×\\\cline{3-3}
  \end{array}
}%
& \rule{0pt}{20pt}
\begin{array}{c}
y(0,{t}_2,\frac{-A_{13}t_2}{2A_{15}^*},t_4,t_5,t_6)\\
 \forall{\ }{t}_2,t_4, t_5,t_6\in \mathbb{F}_{q}
\end{array}
\\\hline
\mathfrak{F}_3 &
\rule{0pt}{18pt}
{%
\newcommand{\mc}[3]{\multicolumn{#1}{#2}{#3}}
\begin{array}{cccccccc}\cline{2-7}
\mc{1}{c|}{} & \mc{1}{c|}{A_{12}} & \mc{1}{c|}{A_{13}^*} & \mc{1}{c|}{×}
                        & \mc{1}{c|}{×} & \mc{1}{c|}{×} & \mc{1}{c|}{×} & \\\cline{2-7}
× & \mc{1}{c|}{} & \mc{1}{c|}{A_{23}} & × & × & × &  & ×\\\cline{3-3}
  \end{array}
}
& Y_1Y_3Y_4Y_5Y_6
\\\hline
\mathfrak{F}_{1,2} &
\rule{0pt}{18pt}
{%
\newcommand{\mc}[3]{\multicolumn{#1}{#2}{#3}}
\begin{array}{cccccccc}\cline{2-7}
\mc{1}{c|}{} & \mc{1}{c|}{A_{12}} & \mc{1}{c|}{×} & \mc{1}{c|}{×} & \mc{1}{c|}{×} & \mc{1}{c|}{×} & \mc{1}{c|}{×} & \\\cline{2-7}
× & \mc{1}{c|}{} & \mc{1}{c|}{A_{23}} & × & × & × &  & ×\\\cline{3-3}
  \end{array}
}
& U
\\\hline
 \end{array}
\end{align*}
where $A_{13}^*, A_{15}^*, A_{16}^*, A_{17}^*\in \mathbb{F}_{q}^*$.
\end{table}
\end{Proposition}
\begin{proof}
The stabilizers are obtained by straightforward calculation.
\end{proof}

\begin{Comparison}
\label{com:classification and stabilizer-G2}
\begin{itemize}
\setlength\itemsep{0em}
\item [(1)] (Classification of orbit modules).
Every (staircase) $G_2^{syl}(q)$-orbit module
has one and only one (staircase) core pattern (see \ref{prop:class orbit-G2}),
which does not hold for (staircase) ${^3}D_4^{syl}(q^3)$-orbit modules (e.g. the family $\mathfrak{F}_3$ of \cite[5.12]{sun3D4super}).
 \item [(2)] (Stabilizer).
Every (staircase) $A_n(q)$-orbit module has a basis element whose stabilizer is a pattern subgroup
(see \cite[\S 3.3]{yan2}).
This does not hold
for ${^3}D_4^{syl}(q^3)$-orbit modules
(see \cite[5.12]{sun3D4super})
or
for $G_2^{syl}(q)$-orbit modules (e.g. the family $\mathfrak{F}_5$ of \ref{prop: G2-stab}).
\end{itemize}
\end{Comparison}



\section{Homomorphisms between orbit modules}
\label{sec:homom. between orbit modules-G2}
Let $U:=G_2^{syl}(q)$,
$t_i\in \mathbb{F}_q$,
$t_i^*\in \mathbb{F}_q^*$ $(i=1,2,\dots,6)$,
and $A_{ij}\in \mathbb{F}_q$,
$A_{ij}^*\in \mathbb{F}_q^*$ $(1\leq i, j \leq 8)$.
In this section,
we show that
every $U$-orbit module is isomorphic to a staircase orbit module (\ref{G2-orbit to staircase}).
Then some irreducible modules are determined,
and any two orbit modules are orthogonal
when the $1$st verges are different (\ref{G2-orth.}).

This property is well known:
every $\varphi \in {\mathrm{End}_{\mathbb{C}U}(\mathbb{C}U)}$ is of the form
$
 \lambda_a \colon  {\mathbb{C}U} \to  {\mathbb{C}U}: y \mapsto ay$,
for a unique $a\in {\mathbb{C}U}$.
If $g\in U$ and $A\in V$, then
$
 \lambda_g|_{\mathbb{C}\mathcal{O}_U([A])} \colon
 {\mathbb{C}\mathcal{O}_U([A])} \to  \mathrm{Im ( \lambda_g|_{\mathbb{C}\mathcal{O}_U([A])}) }
 = g{\mathbb{C}\mathcal{O}_U([A])}
$
is a ${\mathbb{C}U}$-isomorphism,
and $\lambda_g([ A ])
={\frac{1}{|U|}} \sum_{y\in U}{\overline{\vartheta\kappa(g^{-\top}A, y)}}y$.
Let $A=(A_{ij})\in V$ and $y:=y(t_1,t_2,t_3,t_4,t_5,t_6)\in U$.
Then
$ \pi(y^{-\top}A))=A-t_1A_{13}e_{23}$.

\begin{Definition/Lemma}
\label{def: row operation-G2}
The map
$
 U\times V  \to  V: (u,A)\mapsto u.A:=\pi(u^{-\top}A)
$
is a (left) group action called the {\textbf{truncated row operation}},
and the elements of $U$ act as $\mathbb{F}_q$-automorphisms on $V$.
\end{Definition/Lemma}

\begin{Corollary}\label{G2-prop x.A}
Let $A\in V$ and $y:=y(t_1,t_2,t_3,t_4,t_5,t_6)\in U$.
Then
$ y.A=A-t_1A_{13}e_{23} $.
In particular,
$y_1(t_1).A=A-t_1A_{13}e_{23}$
and
$y_i(t_i).A=A$ for all $i\in \{2,3,4,5,6\}$.
\end{Corollary}

\begin{Remark}
Let $A\in V$ and $g,u\in U$.
In general $g.(A.u)\neq (g.A).u$.
For example:
if $t_1,t_4\in \mathbb{F}_{q}^*$ and
$A=A_{17}^*e_{17}$ ($A_{17}^*\in \mathbb{F}_{q}^*$),
then
$\left(y_1(t_1).A\right).y_4(t_4)=A+t_4A_{17}^*e_{13}$
and
$y_1(t_1).\left(A.y_4(t_4)\right)=A+t_4A_{17}^*e_{13}-t_1t_4A_{17}^*e_{23}$.
So $\left(y_1(t_1).A\right).y_4(t_4)\neq y_1(t_1).\left(A.y_4(t_4)\right)$.
\end{Remark}

\begin{Lemma}\label{G2-g[B] obit, Lemma}
Let $B:=B_{12}e_{12}+B_{13}e_{13}+B_{23}e_{23}\in V$,
$g:=y(t_1,t_2,t_3,t_4,t_5,t_6)\in U$ and $y\in U$.
Then
$\vartheta\kappa(g^{-\top}B, y-1)=\chi_{g.B}(y)$.
In particular,
$\vartheta\kappa(y_1(t_1)^{-\top}B, y-1)=\chi_{y_1(t_1).B}(y)$
for all $t_1\in \mathbb{F}_{q}$.
\end{Lemma}

\begin{Proposition}\label{G2-g[B] orbit,Prop}
If $g\in U$ and
$A:=A_{12}e_{12}+A_{13}e_{13}+A_{23}e_{23}\in V$,
then
\begin{displaymath}
 \lambda_g([B])=\chi_{g.B}(g)[g.B] \qquad \text{ for all } B\in \mathcal{O}_U(A).
\end{displaymath}
\end{Proposition}

\begin{Corollary}\label{right action A23 image-G2}
If $g\in U$ and $A:=A_{12}e_{12}+A_{13}e_{13}+A_{23}e_{23}\in V$,
then
$\mathrm{im}(\lambda_g|_{\mathbb{C}\mathcal{O}_{U}([A])})
=\mathbb{C}\mathcal{O}_{U}([g.A])$,
and
$g.(B.u)= (g.B).u$ for all $B\in \mathcal{O}_U(A)$ and $u\in U$.
\end{Corollary}

\begin{Corollary}\label{G2-delete A23}
Let  $A:=A_{12}e_{12}+A_{23}e_{23}+A_{13}^*e_{13}\in V$,
and $y_1(t_1)\in U$
such that $t_1A_{13}^*=A_{23}$.
Then $\mathbb{C}\mathcal{O}_{U}([A])   \cong  \mathbb{C}\mathcal{O}_{U}([y_1(t_1).A])
= \mathbb{C}\mathcal{O}_{U}([A-A_{23}e_{23}])$,
i.e.
\begin{align*}
\mathbb{C}\mathcal{O}_{U}({
\left[%
\newcommand{\mc}[3]{\multicolumn{#1}{#2}{#3}}
\begin{array}{cccccccc}\cline{2-7}
\mc{1}{c|}{} & \mc{1}{c|}{A_{12}} & \mc{1}{c|}{A_{13}^*} & \mc{1}{c|}{×} & \mc{1}{c|}{×} & \mc{1}{c|}{×} & \mc{1}{c|}{×} & \\\cline{2-7}
× & \mc{1}{c|}{} & \mc{1}{c|}{A_{23}} & × & × & × &  & ×\\\cline{3-3}
     \end{array}
\right]}%
)
& \cong  \mathbb{C}\mathcal{O}_{U}({
\left[%
\newcommand{\mc}[3]{\multicolumn{#1}{#2}{#3}}
\begin{array}{cccccccc}\cline{2-7}
\mc{1}{c|}{} & \mc{1}{c|}{A_{12}} & \mc{1}{c|}{A_{13}^*} & \mc{1}{c|}{×} & \mc{1}{c|}{×} & \mc{1}{c|}{×} & \mc{1}{c|}{×} & \\\cline{2-7}
× & \mc{1}{c|}{} & \mc{1}{c|}{0} & × & × & × &  & ×\\\cline{3-3}
     \end{array}
\right]}%
).
\end{align*}
\end{Corollary}

\begin{Corollary}\label{G2-orbit to staircase}
Every $U$-orbit module is isomorphic to a (not necessarily unique) staircase module,
and the isomorphism is given by the left multiplication by a group element.
\end{Corollary}

\begin{Lemma}\label{G2-x5 A17}
Let $A\in V$ with $A_{17}=A_{17}^* \in \mathbb{F}_q^*$,
$y_5(s_5)\in U$ and $s_5\in \mathbb{F}_q$.
Then
\begin{align*}
\lambda_{y_5(s_5)}([A])=\vartheta(s_5A_{16})[A+s_5{A_{17}^*}e_{23}].
\end{align*}
\end{Lemma}
\begin{proof}
Let $A\in V$ with $A_{17}=A_{17}^* \in \mathbb{F}_q^*$
and $y_5(s_5)\in U$, then
\begin{align*}
&\lambda_{y_5(s_5)}([A])
=
{\frac{1}{|U|}} \sum_{y\in U}{\overline{\vartheta\kappa(y_5(s_5)^{-\top}A, y)}}y
{=}
{\frac{1}{|U|}} \sum_{y\in U}{\overline{\vartheta\kappa(A-s_5A_{16}e_{66}-s_5{A_{17}^*}e_{67}, y)}}y\\
{=}&
\vartheta(s_5A_{16})\cdot
{\frac{1}{|U|}} \sum_{y\in U}{\big(\overline{\vartheta\kappa(A, y)}
\cdot \overline{\vartheta(-s_5{A_{17}^*}y_{67})}\big)}y\\
{=}&
\vartheta(s_5A_{16})\cdot
{\frac{1}{|U|}} \sum_{y\in U}{\big(\overline{\vartheta\kappa(A, y)}
\cdot \overline{\vartheta(s_5{A_{17}^*}y_{23})}\big)}y
{=}
\vartheta(s_5A_{16})\cdot
{\frac{1}{|U|}} \sum_{y\in U}{\overline{\vartheta\kappa(A+s_5{A_{17}^*}e_{23}, y)}}y\\
{=}&
\vartheta(s_5A_{16})\cdot
{\frac{1}{|U|}} \sum_{y\in U}{\overline{\vartheta\kappa(A+s_5{A_{17}^*}e_{23}, f(y))}}y
{=}
\vartheta(s_5A_{16}) [A+s_5{A_{17}^*}e_{23}].
\end{align*}
\end{proof}

\begin{Proposition}\label{G2-delete A23, A17}
Let $A,B\in V$, $A_{17}=A_{17}^* \in \mathbb{F}_q^*$, and
\begin{align*}
A:= {%
\newcommand{\mc}[3]{\multicolumn{#1}{#2}{#3}}
\begin{array}{cccccccc}\cline{2-7}
\mc{1}{c|}{} & \mc{1}{c|}{A_{12}} & \mc{1}{c|}{A_{13}} & \mc{1}{c|}{A_{15}}
& \mc{1}{c|}{A_{15}} & \mc{1}{c|}{A_{16}} & \mc{1}{c|}{A_{17}^*} & \\\cline{2-7}
× & \mc{1}{c|}{} & \mc{1}{c|}{A_{23}} & × & × & × &  & ×\\\cline{3-3}
     \end{array}
},
\quad
B:=
{%
\newcommand{\mc}[3]{\multicolumn{#1}{#2}{#3}}
\begin{array}{cccccccc}\cline{2-7}
\mc{1}{c|}{} & \mc{1}{c|}{A_{12}} & \mc{1}{c|}{A_{13}} & \mc{1}{c|}{A_{15}}
& \mc{1}{c|}{A_{15}} & \mc{1}{c|}{A_{16}} & \mc{1}{c|}{A_{17}^*} & \\\cline{2-7}
× & \mc{1}{c|}{} & \mc{1}{c|}{0} & × & × & × &  & ×\\\cline{3-3}
     \end{array}
}.
\end{align*}
Then
$ \mathbb{C}\mathcal{O}_{U}([A])  \cong  \mathbb{C}\mathcal{O}_{U}([B])$.
\end{Proposition}
\begin{proof}
Let $C\in \mathcal{O}_U(A)$ and $s_5:=-\frac{A_{23}}{A_{17}^*}\in \mathbb{F}_q$.
By \ref{G2-x5 A17}, we get
$\lambda_{y_5(s_5)}([C])=\vartheta(s_5C_{16})[C+s_5{A_{17}^*}e_{23}]$,
where $C+s_5{A_{17}^*}e_{23}\in \mathcal{O}_U(B)$.
Thus $\mathbb{C}\mathcal{O}_{U}([A])
\cong  \mathbb{C}\mathcal{O}_{U}([B])$.
\end{proof}

Let $1\leq i \leq 8$.
Then the $i$\textbf{-th hook} of $J$ is
$H_i:=\{(a,b) \in J  \mid  b=i \text{ or } a=9-i\}$.
In particular, $H_7=\{(1,7), (2,3)\}$.
Let $A\in V$.
Then $A$ is called
\textbf{hook-separated},
if on every hook $H_i$ of $J$ lies at most one main condition of $A$.
The hook-separated patterns are always the staircase patterns.
If $A\in V$ be hook-separated,
then $\mathbb{C}\mathcal{O}_U([A])$ is called
a \textbf{hook-separated staircase module}.

\begin{Proposition}\label{G2-classification, hook-sep.}
 Every $U$-orbit module is isomorphic to a
 hook-separated staircase module.
\end{Proposition}
\begin{proof}
 By \ref{G2-orbit to staircase}, every $U$-orbit module is isomorphic to a
 staircase module.
 By \ref{G2-delete A23, A17}, we get the desired conclusion.
\end{proof}

\begin{Corollary}
Let $A\in V$ be a staircase pattern and $S$ an irreducible constituent of $\mathbb{C}\mathcal{O}_U([A])$.
Then there exists a hook-separated
core
pattern $C$, such that $S$ is a constituent of
$\mathbb{C}\mathcal{O}_U([C])$.
\end{Corollary}

\begin{Corollary}
\label{irr. module-G2}
Every irreducible $\mathbb{C}U$-module is
 a constituent of
some
hook-separated staircase module.
\end{Corollary}

Let $A, B \in V$,
$\mathrm{Stab}_U(A, B):=\mathrm{Stab}_U(A)\cap \mathrm{Stab}_U(B)$,
$\psi_A$ be the character of $\mathbb{C}\mathcal{O}_U([A])$
and $\psi_B$ denote the character of $\mathbb{C}\mathcal{O}_U([B])$.
Then
$\mathrm{Hom}_{\mathbb{C}U}(\mathbb{C}\mathcal{O}_U([A]),\mathbb{C}\mathcal{O}_U([B]))=\{0\}$
if and only if for all $C\in \mathcal{O}_U(A)$ and $D\in \mathcal{O}_U(B)$ holds
$\mathrm{Hom}_{\mathrm{Stab}_U(C, D)}(\mathbb{C}[C],\mathbb{C}[D])=\{0\}$.
In particular,
\begin{align*}
&\mathrm{dim}_{\mathbb{C}} \mathrm{Hom}_{\mathbb{C}U}(\mathbb{C}\mathcal{O}_U([A]),\mathbb{C}\mathcal{O}_U([B]))
    = \langle \psi_A, \psi_B\rangle_{U}\\
   = &\sum_{\substack{C\in \mathcal{O}_U(A)\\ D\in \mathcal{O}_U(B)}}
   \frac{|\mathrm{Stab}_U(C, D)|}{|U|} \Big(\mathrm{dim}_{\mathbb{C}}
                          \mathrm{Hom}_{\mathrm{Stab}_U(C, D)}(\mathbb{C}[C],\mathbb{C}[D])\Big).
\end{align*}
Thus
$\mathrm{Hom}_{\mathbb{C}U}(\mathbb{C}\mathcal{O}_U([A]),\mathbb{C}\mathcal{O}_U([B]))=\{0\}$
if and only if
$\mathrm{Hom}_{\mathrm{Stab}_U(A, D)}(\mathbb{C}[ A ],\mathbb{C}[D])=\{0\}$
for all $D\in \mathcal{O}_U(B)$ (\cite[\S 3.3]{Markus1}).
We have
$\langle \psi_A, \psi_B\rangle_{U}
=\sum_{D\in \mathcal{O}_U(B)}
    \frac{|\mathrm{Stab}_U(A, D)|}
    {|\mathrm{Stab}_U(A)|}{\langle \chi_A, \chi_D\rangle_{\mathrm{Stab}_U(A, D)}}$,
where
$\chi_A$ and $\chi_D$ are the characters of the $\mathbb{C}\mathrm{Stab}_U(A, D)$-modules
$\mathbb{C}[A]$ and $\mathbb{C}[D]$ respectively.

\begin{Proposition}\label{G2-orth.}
Every $U$-orbit module is isomorphic to a hook-separated staircase module
in Table \ref{Table: hook sep.-G2},
and they satisfy the following properties.
\begin{itemize}
\setlength\itemsep{0em}
\item [(1)]
Let $A,B\in V$.
If $\mathrm{verge}_1(A)\neq \mathrm{verge}_1(B)$,
$
 \mathrm{Hom}_{\mathbb{C}U}(\mathbb{C}\mathcal{O}_U([A]),\mathbb{C}\mathcal{O}_U([B]))=\{0\}$.
In particular, if $\mathbb{C}\mathcal{O}_U([A])\in \mathfrak{F}_i$,
$\mathbb{C}\mathcal{O}_U([B])\in \mathfrak{F}_j$ and $i\neq j$,
then
$
 \mathrm{Hom}_{\mathbb{C}U}(\mathbb{C}\mathcal{O}_U([A]),\mathbb{C}\mathcal{O}_U([B]))=\{0\}$.
\item [(2)]
In the family $\mathfrak{F}_{1,2}$, the $q^2$ hook-separated staircase modules
are irreducible
and pairwise orthogonal.
\item [(3)]
In the family $\mathfrak{F}_3$, the $(q-1)$ hook-separated staircase modules
are irreducible
and pairwise orthogonal.
\item [(4)]
In the family $\mathfrak{F}_{4}$, $\mathfrak{F}_{5}$ and $\mathfrak{F}_{6}$,
the hook-separated staircase modules
are reducible.
\end{itemize}
\begin{table}[!htp]
\caption{Hook-separated staircase $G_2^{syl}(q)$-modules}
\label{Table: hook sep.-G2}
\begin{align*}
\begin{array}{|c|l|c|c|}\hline
\text{Family}
& \multicolumn{1}{c|}{\mathbb{C}\mathcal{O}_U([A]){\ } (A\in V)}
& \mathrm{dim}_{\mathbb{C}}\mathbb{C}\mathcal{O}_U([A])
& \text{Irreducible}\\\hline
\hline
\mathfrak{F}_6
& \rule{0pt}{20pt}
\mathbb{C}\mathcal{O}_U\Big(
{\left[%
\newcommand{\mc}[3]{\multicolumn{#1}{#2}{#3}}
\begin{array}{cccccccc}\cline{2-7}
\mc{1}{c|}{} & \mc{1}{c|}{A_{12}} & \mc{1}{c|}{×} & \mc{1}{c|}{×} & \mc{1}{c|}{×} & \mc{1}{c|}{×} & \mc{1}{c|}{A_{17}^*} & \\\cline{2-7}
× & \mc{1}{c|}{} & \mc{1}{c|}{0} & × & × & × & & ×\\\cline{3-3}
\end{array}
\right]}\Big)
  & q^{3}
  & \text{NO}\\\hline
\mathfrak{F}_5
& \rule{0pt}{20pt}
\mathbb{C}\mathcal{O}_U\Big(
{\left[%
\newcommand{\mc}[3]{\multicolumn{#1}{#2}{#3}}
\begin{array}{cccccccc}\cline{2-7}
\mc{1}{c|}{} & \mc{1}{c|}{×} & \mc{1}{c|}{A_{13}} & \mc{1}{c|}{×} & \mc{1}{c|}{×} & \mc{1}{c|}{A_{16}^*} & \mc{1}{c|}{×} & \\\cline{2-7}
× & \mc{1}{c|}{} & \mc{1}{c|}{A_{23}} & × & × & × &  & ×\\\cline{3-3}
\end{array}
\right]}\Big)
& q^{2}
& \text{NO}\\\hline
\mathfrak{F}_4
& \rule{0pt}{20pt}
\mathbb{C}\mathcal{O}_U\Big(
{\left[%
\newcommand{\mc}[3]{\multicolumn{#1}{#2}{#3}}
\begin{array}{cccccccc}\cline{2-7}
\mc{1}{c|}{} & \mc{1}{c|}{×} & \mc{1}{c|}{×}
& \mc{1}{c|}{{A_{15}^*}} & \mc{1}{c|}{A_{15}^*} & \mc{1}{c|}{×} & \mc{1}{c|}{×} & \\\cline{2-7}
× & \mc{1}{c|}{} & \mc{1}{c|}{A_{23}} & × & × & × &  & ×\\\cline{3-3}
\end{array}
\right]}\Big)
 & q^2
 & \text{NO} \\\hline
\mathfrak{F}_3
& \rule{0pt}{20pt}
\mathbb{C}\mathcal{O}_U\Big(
{\left[%
\newcommand{\mc}[3]{\multicolumn{#1}{#2}{#3}}
\begin{array}{cccccccc}\cline{2-7}
\mc{1}{c|}{}
 & \mc{1}{c|}{×} & \mc{1}{c|}{A_{13}^*}
& \mc{1}{c|}{×} & \mc{1}{c|}{×} & \mc{1}{c|}{×} & \mc{1}{c|}{×} & \\\cline{2-7}
× & \mc{1}{c|}{} & \mc{1}{c|}{0} & × & × & × &  & ×\\\cline{3-3}
\end{array}
\right]}\Big)
& q
& \text{YES}  \\\hline
\mathfrak{F}_{1,2}
& \rule{0pt}{20pt}
\mathbb{C}\mathcal{O}_U\Big(
{\left[%
\newcommand{\mc}[3]{\multicolumn{#1}{#2}{#3}}
\begin{array}{cccccccc}\cline{2-7}
\mc{1}{c|}{} & \mc{1}{c|}{A_{12}} & \mc{1}{c|}{×} & \mc{1}{c|}{×} & \mc{1}{c|}{×} & \mc{1}{c|}{×} & \mc{1}{c|}{×} & \\\cline{2-7}
× & \mc{1}{c|}{} & \mc{1}{c|}{A_{23}} & × & × & × &  & ×\\\cline{3-3}
\end{array}
\right]}\Big)
 & 1
 & \text{YES} \\\hline
 \end{array}
\end{align*}
where $A_{13}^*, A_{15}^*,
A_{16}^*,A_{17}^* \in \mathbb{F}_{q}^*$.
\end{table}
\end{Proposition}

\begin{proof}
By \ref{G2-orbit to staircase}
and \ref{G2-classification, hook-sep.},
every $U$-orbit module is isomorphic to a hook-separated staircase module in Table \ref{Table: hook sep.-G2}.
\begin{itemize}
\setlength\itemsep{0em}
\item [(a)]
Let $A=A_{15}^*(e_{15}+e_{14})+A_{23}e_{23}\in \mathfrak{F}_4$,
$B=B_{15}^*(e_{15}+e_{14})+B_{23}e_{23}\in \mathfrak{F}_4$
(i.e. $A_{15}^*, B_{15}^*\in \mathbb{F}_q^*$)
and $C\in \mathcal{O}_U(B)$.
By \ref{prop: G2-stab},
$
\mathrm{Stab}_U(A)
{=}Y_2Y_4Y_5Y_6$,
so
$
\mathrm{Stab}_U(A, C)
=\left\{
\begin{array}{ll}
 Y_2Y_4Y_5Y_6, & C_{13}=0\\
 Y_4Y_5Y_6, & C_{13}\neq 0\\
\end{array}
\right.$.
We calculate the inner product:
\begin{align*}
& \langle \chi_A, \chi_C\rangle_{\mathrm{Stab}_U(A, C)}
   =\frac{1}{|{\mathrm{Stab}_U(A, C)}|}\sum_{y\in {\mathrm{Stab}_U(A, C)}} \vartheta \kappa(A-C, f(y))\\
=&\frac{1}{|{\mathrm{Stab}_U(A, C)}|}
  \sum_{y\in {\mathrm{Stab}_U(A, C)}}   \vartheta \kappa
  \left(
{%
\newcommand{\mc}[3]{\multicolumn{#1}{#2}{#3}}
\begin{array}{c|c|cccc}\hline
\mc{1}{|c|}{-C_{12}} & -C_{13} & \mc{1}{c|}{A_{15}^*-B_{15}^*}
                               & \mc{1}{c|}{A_{15}^*-B_{15}^*} & \mc{1}{c|}{×} & \mc{1}{c|}{×}\\\hline
× & A_{23}-B_{23} & × & × & × & ×\\\cline{2-2}
\end{array}
}%
,
f(y)\right).
\end{align*}
{If $C_{13}=0$, then}
\begin{align*}
0\neq & \mathrm{dim}_\mathbb{C} \mathrm{Hom}_{\mathrm{Stab}_U(A, C)}(\mathbb{C}[ A ],\mathbb{C}[ C ])
     = \langle \chi_A, \chi_C\rangle_{\mathrm{Stab}_U(A, C)}\\
=&\frac{1}{|Y_2Y_4Y_5Y_6|}\sum_{ t_2,t_4, t_5,t_6\in \mathbb{F}_{q} }
         \vartheta \Big((A_{23}-B_{23})t_2+2(A_{15}^*-B_{15}^*)t_4\Big)\\
=&\bigg(\frac{1}{q}\sum_{t_2\in \mathbb{F}_{q}}  \vartheta \Big((A_{23}-B_{23})t_2\Big)\bigg)
\bigg(\frac{1}{q}\sum_{ t_4\in \mathbb{F}_{q}} \vartheta \Big(2(A_{15}^*-B_{15}^*)t_4\Big)\bigg)\\
\iff &  \{B_{23}= A_{23}\} \wedge \{B^*_{15}=A^*_{15}\}.
\end{align*}
{If $C_{13}\neq 0$, then}
$ 0\neq
    \langle \chi_A, \chi_C\rangle_{\mathrm{Stab}_U(A, C)}
=\frac{1}{q}\sum_{ t_4\in \mathbb{F}_{q}} \vartheta \Big(2(A_{15}^*-B_{15}^*)t_4\Big)
 \iff    B^*_{15}=A^*_{15}$.
Thus
$
\mathrm{Hom}_{\mathrm{Stab}_U(A, C)}(\mathbb{C}[ A ],\mathbb{C}[ C ])\neq\{0\}
 \iff  \langle \chi_A, \chi_C\rangle_{\mathrm{Stab}_U(A, C)}\neq 0 \quad (i.e. =1)
 \iff  \big\{\{B_{23}= A_{23}\} \wedge \{B^*_{15}=A^*_{15}\}\big\} \wedge \{B^*_{15}=A^*_{15}\}
 \iff  B^*_{15}=A^*_{15}$.
{Thus}
$
 \mathrm{Hom}_{\mathbb{C}U}(\mathbb{C}\mathcal{O}_U([A]),\mathbb{C}\mathcal{O}_U([B]))=\{0\}
\iff  B^*_{15}\neq A^*_{15}$.
\item[(b)]
Let $A\in \mathfrak{F}_i$ and  $B\in \mathfrak{F}_j$,
$\psi_A$ denote the character of $\mathbb{C}\mathcal{O}_U([A])$
and  $\psi_B$ the character of $\mathbb{C}\mathcal{O}_U([B])$.
In the similar way to (a), we calculate $ \langle \psi_A, \psi_B\rangle_{U}$.
Then the statement of (1) is proved.
\item[(c)]
Let $A,B\in V$ be hook-separated staircase core patterns of the family $\mathfrak{F}_4$.
Let $D\in \mathcal{O}_U(A)$ and $\psi_A$ denote the character of $\mathbb{C}\mathcal{O}_U([A])$.
By (a),  we have $\mathbb{C}[A] \cong \mathbb{C}[A]$
as $\mathbb{C}\mathrm{Stab}_U(A,D)$-modules.
Then
\begin{align*}
& \mathrm{dim}_{\mathbb{C}}
\mathrm{Hom}_{\mathbb{C}U}(\mathbb{C}\mathcal{O}_U([A]),\mathbb{C}\mathcal{O}_U([A]))
= \langle \psi_A, \psi_A\rangle_{U}\\
{=}& \sum_{D\in \mathcal{O}_U(A)}
   \frac{|\mathrm{Stab}_U(A, D)|}{|\mathrm{Stab}_U(A)|}
   \mathrm{dim}_{\mathbb{C}}\mathrm{Hom}_{\mathrm{Stab}_U(A, D)}(\mathbb{C}[ A ],\mathbb{C}[ D ])
  =\frac{q^4\cdot q}{q^4}+\frac{q^3\cdot (q-1)q}{q^4}\\
=& 2q-1 > 1.
\end{align*}
Thus,
$\mathbb{C}\mathcal{O}_U([A])$ is not irreducible.
\item[(d)] Let $A\in V$ be a hook-separated staircase core pattern of the family $\mathfrak{F}_5$.
In the similar way to (c), $\mathbb{C}\mathcal{O}_U([A])$ is not irreducible.
\item[(e)]
Let $A,B\in V$ be hook-separated staircase core patterns of the family $\mathfrak{F}_3$ and $A\neq B$.
We have $\langle \psi_A, \psi_A\rangle_{U}=1$ and $\langle \psi_A, \psi_B\rangle_{U}=0$.
Thus the statement of (3) is proved.
\item[(f)]
The $q^2$ hook-separated staircase modules of $\mathfrak{F}_{1,2}$ are of dimension $1$,
so they are irreducible.
They are pairwise orthogonal by calculating $ \langle \psi_A, \psi_B\rangle_{U}$
(c.f. (a)).
\item[(g)]
Let $A\in V$ be a hook-separated staircase core pattern of the family $\mathfrak{F}_6$.
Then the orbit module $\mathbb{C}\mathcal{O}_U([A])$ is reducible.
Suppose it is irreducible.
Then by (1) and (2) we get
$
\big(\dim_{\mathbb{C}} \mathbb{C}\mathcal{O}_U([A])\big)^2=q^{6}< |U|-q^2=q^{6}-q^2$.
This is a contradiction.
Thus the orbit modules of the family $\mathfrak{F}_6$ are reducible.
\end{itemize}
\end{proof}

\begin{Remark}
\label{hook-separated intersect-G2}
\begin{itemize}
\setlength\itemsep{0em}
\item [(1)]
The proof of the reducible properties of families $\mathfrak{F}_4$ and $\mathfrak{F}_5$ of $G_2^{syl}(q)$
(i.e. (c) and (d) of the proof of \ref{G2-orth.})
is different from that of ${^3}D_4^{syl}(q^3)$ (see \cite[6.15]{sun3D4super}).
\item [(2)]
There exist two  hook-separated staircase modules such that
they are neither orthogonal nor isomorphic.
For example:
if $A,B\in V$ be hook-separated staircase core patterns of the family $\mathfrak{F}_4$
with $A_{15}^*=B_{15}^*$ and $A_{23}\neq B_{23}$,
then
$ \langle \psi_A, \psi_A\rangle_{U}
 =\langle \psi_B, \psi_B\rangle_{U}=2q-1 $
but
$ \langle \psi_A, \psi_B\rangle_{U}=q-1\notin\{0,{\,} 2q-1\}$,
so $\mathbb{C}\mathcal{O}_U([A])$ and $\mathbb{C}\mathcal{O}_U([B])$
are neither orthogonal nor isomorphic.
\end{itemize}
\end{Remark}

\begin{Comparison}
\label{com:classification staircase U-modules-G2}
\begin{itemize}
\setlength\itemsep{0em}
\item [(1)] (Classification of staircase $U$-modules).
Every $G_2^{syl}(q)$-orbit module is isomorphic to a staircase $U$-module
(see  \ref{G2-orbit to staircase}).
\item[(2)] (Irreducible $U$-modules).
Every irreducible $G_2^{syl}(q)$-module is a constituent of some
hook-separated staircase module  (see \ref{irr. module-G2}).
\end{itemize}
The two properties also hold for ${^3}D_4^{syl}(q^3)$-modules (see \cite[6.17]{sun3D4super}).
\end{Comparison}


\section{A partition of $G_2^{syl}(q)$}
\label{partition of U-G2}
Let $G:=G_8(q)$, $U:=G_2^{syl}(q)$,
and $t_i\in \mathbb{F}_q$,
$t_i^*\in \mathbb{F}_q^*$ $(i=1,2,\dots,6)$.
In this section,
a partition of $G_2^{syl}(q)$ is determined (see \ref{superclass:ci ti-G2})
which is a set of the superclasses proved in Section \ref{sec: supercharacter theories-G2}.

\begin{Lemma}
The set
$ V_G:=G-1=\{ g-1\mid g\in G \} $
is a nilpotent associative $\mathbb{F}_q$-algebra (G is an algebra group).
\end{Lemma}

\begin{Notation/Lemma}
If $g\in G$ and $u\in U$,
then set
$
G(g-1)G:= \{x(g-1)y \mid x,y\in G\}\subseteq V_G$,
$C_g^G:= \{1+x(g-1)y \mid x,y\in G\}=1+G(g-1)G \subseteq G$,
and
$C_u^U:= \{1+x(u-1)y \mid x,y\in G\}\cap U\subseteq C_u^G$.
\end{Notation/Lemma}

\begin{Lemma}\label{G2,biorbit-G}
If $g,h\in G$,
then the following statements are equivalent:
(1) There exist $x,y\in G$, such that $g-1=x(h-1)y$,
(2) $C_g^G=C_h^G$,
(3) $g\in C_h^G$.
 The set $\{C_g^G \mid g\in G\}$ forms a partition of $G$
 with respect to the equivalence relations.
If $g\in G$, then $C_g^G$ is a union of conjugacy classes of $G$.
\end{Lemma}

\begin{Lemma}\label{G2,biorbit-U}
If $u,v\in U$, then the following statements are equivalent:
(1) There exist $x,y\in G$, such that $u-1=x(v-1)y$,
(2) $C_u^U=C_v^U$,
(3) $u\in C_v^U$.
 The set $\{C_u^U \mid u\in U\}$ forms a partition of $U$
 with respect to the equivalence relations.
If $u\in U$, then $C_u^U$ is a union of conjugacy classes of $U$.
\end{Lemma}
We obtain a partition of $G_2^{syl}(q)$
by straightforward calculation.
\begin{Proposition}[A partition of $G_2^{syl}(q)$]\label{partition-G2}
The $C_u^U$ $(u\in U)$ are given in Table \ref{table:partition-G2}.
\begin{table}[!htp]
\caption{A partition of $G_2^{syl}(q)$}
\label{table:partition-G2}
\begin{align*}
\begin{array}{|l|l|c|}\hline
\multicolumn{1}{|c|}{u \in U}
& \multicolumn{1}{c|}{C_u^U}
& \rule{0pt}{13pt}
\begin{array}{c}
|C_u^U| \\
\end{array}
\\\hline
I_8 & y(0,0,0,0,0,0) & 1\\\hline
y_6(t_6^*), {\ }
t_6^*\in \mathbb{F}_q^*
& y(0,0,0,0,0,t_6^*) & 1\\\hline
y_5(t_5^*), {\ }
t_5^*\in \mathbb{F}_q^*
& y(0,0,0,0,t_5^*,s_6),{\ }
s_6\in \mathbb{F}_q
& q\\\hline
y_4(t_4^*), {\ }
t_4^*\in \mathbb{F}_q^*
& y(0,0,0,t_4^*,s_5,s_6),{\ }
s_5,s_6\in \mathbb{F}_q
& q^2\\\hline
y_3(t_3^*), {\ }
t_3^*\in \mathbb{F}_q^*
& y(0,0,t_3^*,s_4,s_5,s_6),{\ }
s_4, s_5,s_6\in \mathbb{F}_q
& q^3\\\hline
y_2(t_2^*)y_4(t_4^*), {\ }
t_2^*,t_4^*\in \mathbb{F}_q^*
& y(0,t_2^*,s_3,t_4^*-\frac{s_3^{2}}{t_2^*},s_5,s_6),{\ }
 s_3, s_5,s_6\in \mathbb{F}_q
&  q^3 \\\hline
y_2(t_2^*)y_5(t_5), {\ }
t_2^*\in \mathbb{F}_q^*, {\ }
t_5\in \mathbb{F}_q
& y(0,t_2^*,s_3,-\frac{s_3^{2}}{t_2^*},t_5+\frac{s_3^{3}}{{t_2^*}^2},s_6),{\ }
s_3, s_6\in \mathbb{F}_q
& q^2 \\\hline
y_1(t_1^*), {\ }
t_1^*\in \mathbb{F}_q^*
& y(t_1^*,0,s_3,s_4,s_5,s_6),{\ }
s_3, s_4,s_5,s_6\in \mathbb{F}_q
& q^4 \\\hline
y_2(t_2^*)y_1(t_1^*), {\ }
t_1^*,t_2^*\in \mathbb{F}_q^*
& y(t_1^*,t_2^*,s_3,s_4,s_5,s_6),{\ }
 s_3,s_4, s_5,s_6\in \mathbb{F}_q
& q^4\\\hline
\end{array}
\end{align*}
\end{table}
\end{Proposition}

\begin{Notation/Lemma}\label{superclass:ci ti-G2}
Set
 \begin{align*}
  & C_6(t_6^*):= C_{y_6(t_6^*)}^U,\quad
  C_5(t_5^*):= C_{y_5(t_5^*)}^U,\quad
  C_4(t_4^*):= C_{y_4(t_4^*)}^U,\quad
  C_3(t_3^*):= C_{y_3(t_3^*)}^U,\\
  & C_2(t_2^*):= \big(\bigcup_{t_4^*\in \mathbb{F}_q^*}^{.}{ C_{y_2(t_2^*)y_4(t_4^*)}^U}\big)
                \dot{\bigcup}
                \big(\bigcup_{t_5\in \mathbb{F}_q}^{.}{ C_{y_2(t_2^*)y_5(t_5)}^U}\big),\\
  & C_{1}(t_1^*):=C_{y_1(t_1^*)}^U,\quad
  C_{1,2}(t_1^*,t_2^*):= C_{y_2(t_2^*)y_1(t_1^*)}^U,\quad
  C_0:=\{1_U\}=\{ I_8 \}.
 \end{align*}
Note that these sets form a partition of $U$, denoted by $\mathcal{K}$.
\end{Notation/Lemma}

\begin{Comparison}[Superclasses]
\label{com:superclasses-G2}
The superclasses of $G_2^{syl}(q)$
are determined by $C_u^U=\{I_{8}+x(u-I_{8})y\mid x,y \in G_8(q)\}\cap G_2^{syl}(q)$ for all $u\in G_2^{syl}(q)$
(see \ref{partition-G2},
\ref{superclass:ci ti-G2} and \ref{supercharacter theory-G2}).
This construction is analogous
to that of ${^3}D_4^{syl}(q^3)$ (see \cite[\S7]{sun3D4super}).
\end{Comparison}


\section{A supercharacter theory for $G_2^{syl}(q)$}
\label{sec: supercharacter theories-G2}

In this section, we determine a supercharacter theory for $G_2^{syl}(q)$ (\ref{supercharacter theory-G2}),
and establish the supercharacter table for $G_2^{syl}(q)$ in
Table \ref{table:supercharacter table-G2}.
Let $U:=G_2^{syl}(q)$,
$t_i\in \mathbb{F}_q$,
$t_i^*\in \mathbb{F}_q^*$ $(i=1,2,\dots,6)$,
and $A_{ij}\in \mathbb{F}_q$,
$A_{ij}^*\in \mathbb{F}_q^*$ $(1\leq i, j \leq 8)$.
\begin{Definition}[{\cite[\S 2]{di}}/{\cite[3.6.2]{Markus1}}]
\label{supercharacter theory}
Let $G$ be a finite group.
Suppose that $\mathcal{K}$ is a partition of $G$
and that $\mathcal{X}$ is a set of (nonzero) complex characters of $G$,
such that
\begin{itemize}
 \setlength\itemsep{0em}
 \item [(a)] $|\mathcal{X}|=|\mathcal{K}|$,
 \item [(b)] every character $\chi \in \mathcal{X}$ is constant on each member of $\mathcal{K}$,
 \item [(c)] the elements of $\mathcal{X}$ are pairwise orthogonal and
  \item[(d)] the set $\{1\}$ is a member of $\mathcal{K}$.
\end{itemize}
Then $(\mathcal{X},\mathcal{K})$
is called a \textbf{supercharacter theory} for $G$.
We refer to the elements of $\mathcal{X}$ as \textbf{supercharacters},
and to the elements of $\mathcal{K}$ as \textbf{superclasses}
of $G$.
A $\mathbb{C}G$-module is called a $\mathbb{C}G$-\textbf{supermodule},
if it affords a supercharacter of $G$.
\end{Definition}

\begin{Notation/Lemma}\label{notation:supermodules-G2}
For $A=(A_{ij})\in V$, we set
\begin{align*}
  M{(A_{12}e_{12}+A_{23}e_{23})}:
= \mathbb{C}\mathcal{O}_{U}([A_{12}e_{12}+A_{23}e_{23}])
=\mathbb{C}[A_{12}e_{12}+A_{23}e_{23}],
\end{align*}

\begin{align*}
 M{(A_{13}^*e_{13})}:=&
{\mathbb{C}\left\{
\left[{\newcommand{\mc}[3]{\multicolumn{#1}{#2}{#3}}
\begin{array}{cccccccc}\cline{2-7}
\mc{1}{c|}{}
& \mc{1}{c|}{A_{12}}
& \mc{1}{c|}{{A_{13}^*}}
& \mc{1}{c|}{×}
& \mc{1}{c|}{×} & \mc{1}{c|}{×} & \mc{1}{c|}{×} & \\\cline{2-7}
× & \mc{1}{c|}{} & \mc{1}{c|}{} & × & × & × & & ×\\\cline{3-3}
\end{array}
}
\right]
  {\,}\middle|{\,}
  A_{12}\in \mathbb{F}_{q}
\right\}
}
= \mathbb{C}\mathcal{O}_{U}([A_{13}^*e_{13}]),
\end{align*}
\begin{align*}
M{({A_{15}^*}(e_{14}+e_{15}))}:=
&
\mathbb{C}\left\{
\left[{\newcommand{\mc}[3]{\multicolumn{#1}{#2}{#3}}
\begin{array}{cccccccc}\cline{2-7}
\mc{1}{c|}{}
& \mc{1}{c|}{A_{12}}
& \mc{1}{c|}{A_{13}}
& \mc{1}{c|}{{{A_{15}^*}}}
& \mc{1}{c|}{{A_{15}^*}} & \mc{1}{c|}{×} & \mc{1}{c|}{×} & \\\cline{2-7}
× & \mc{1}{c|}{} & \mc{1}{c|}{A_{23}} & × & × & × &  & ×\\\cline{3-3}
\end{array}
}
\right]
 {\,}\middle|{\,}
A_{12},A_{13},A_{23}\in \mathbb{F}_{q}
\right\}
\\
=& \bigoplus_{A_{23}\in \mathbb{F}_{q}}
\mathbb{C}\mathcal{O}_{U}([{A_{15}^*}(e_{14}+e_{15})+A_{23}e_{23}]),
\end{align*}
\begin{align*}
 M{(A_{16}^*e_{16})}:=
&
{\mathbb{C}\left\{
\left[{\newcommand{\mc}[3]{\multicolumn{#1}{#2}{#3}}
\begin{array}{cccccccc}\cline{2-7}
\mc{1}{c|}{}
& \mc{1}{c|}{A_{12}}
& \mc{1}{c|}{A_{13}}
& \mc{1}{c|}{A_{15}}
& \mc{1}{c|}{A_{15}} & \mc{1}{c|}{{A_{16}^*}} & \mc{1}{c|}{×} & \\\cline{2-7}
× & \mc{1}{c|}{} & \mc{1}{c|}{A_{23}} & × & × & × &  & ×\\\cline{3-3}
\end{array}
}
\right]
 {\,}\middle|{\,}
A_{12},A_{13},A_{15},A_{23}\in \mathbb{F}_{q}
\right\}
}\\
=& \bigoplus_{A_{13},A_{23}\in \mathbb{F}_{q}}
\mathbb{C}\mathcal{O}_{U}([A_{16}^*e_{16}+A_{13}e_{13}+A_{23}e_{23}]),
\end{align*}
\begin{align*}
 M{(A_{17}^*e_{17})}:=
&
{\mathbb{C}\left\{
\left[{\newcommand{\mc}[3]{\multicolumn{#1}{#2}{#3}}
\begin{array}{cccccccc}\cline{2-7}
\mc{1}{c|}{}
& \mc{1}{c|}{A_{12}}
& \mc{1}{c|}{A_{13}}
& \mc{1}{c|}{A_{15}}
& \mc{1}{c|}{A_{15}} & \mc{1}{c|}{A_{16}} & \mc{1}{c|}{{A_{17}^*}} & \\\cline{2-7}
× & \mc{1}{c|}{} & \mc{1}{c|}{} & × & × & × &  & ×\\\cline{3-3}
\end{array}
}\right]
 {\,}\middle|{\,}
A_{12},A_{13},A_{15},A_{16}\in \mathbb{F}_{q}
\right\}
}\\
=& \bigoplus_{A_{12}\in \mathbb{F}_{q}}
\mathbb{C}\mathcal{O}_{U}([A_{17}^*e_{17}+A_{12}e_{12}]).
\end{align*}
Denote by $\mathcal{M}$
the set of all of the above $\mathbb{C}U$-modules.
\end{Notation/Lemma}

\begin{Lemma}\label{notation:supermodules and G-module-G2}
Let $A=(A_{ij})\in V$ and $G:=G_8(q)$.
Then
all $G$-orbit modules are irreducible,
and every $U$-module in $\mathcal{M}$ is
a direct sum of restrictions of some $G_8(q)$-orbit modules to $G_2^{syl}(q)$
as follows:
 \begin{align*}
 & M{(A_{12}e_{12}+A_{23}e_{23})}=
  \mathrm{Res}^G_U \mathbb{C}\mathcal{O}_{G}([A_{12}e_{12}+A_{23}e_{23}]),
 \quad
 M{(A_{13}^*e_{13})}=
 \mathrm{Res}^G_U \mathbb{C}\mathcal{O}_{G}([A_{13}^*e_{13}]),\\
& M{(A_{15}^*(e_{14}+e_{15}))}=
  \bigoplus_{A_{23}\in \mathbb{F}_{q}}
\mathrm{Res}^G_U \mathbb{C}\mathcal{O}_{G}
([A_{15}^*(e_{14}+e_{15})+A_{23}e_{23}]),\\
&  M{(A_{16}^*e_{16})}= \bigoplus_{A_{23}\in \mathbb{F}_{q}}
\mathrm{Res}^G_U \mathbb{C}\mathcal{O}_{G}
([A_{16}^*e_{16}+A_{23}e_{23}]),
 \quad
 M{(A_{17}^*e_{17})}=
 \mathrm{Res}^G_U \mathbb{C}\mathcal{O}_{G}
([A_{17}^*e_{17}]).
\end{align*}
\end{Lemma}

\begin{Notation}\label{set of supercharacters-G2}
For $M\in \mathcal{M}$, the complex character
of the $\mathbb{C}U$-module $M$ is denoted by $\Psi_M$.
We set
$\mathcal{X}:=\left\{\Psi_M  {\,}\middle|{\,}  M\in \mathcal{M} \right\}$.
\end{Notation}

\begin{Corollary}
Let $A=(A_{ij})\in V$, and $\psi_A$ be the character of $\mathbb{C}\mathcal{O}_U([A])$.
Then
 \begin{alignat*}{2}
 & \Psi_{M(A_{12}e_{12}+A_{23}e_{23})}= \psi_{A_{12}e_{12}+A_{23}e_{23}},
 & \quad
 & \Psi_{M(A_{13}^*e_{13})}= {\psi_{A_{13}^*e_{13}}},\\
 & \Psi_{M({A_{15}^*}(e_{14}+e_{15}))}
                         = \sum_{A_{23}\in \mathbb{F}_q}
                           {\psi_{A_{23}e_{23}+{A_{15}^*}(e_{14}+e_{15})}},
&\quad
& \Psi_{M(A_{16}^*e_{16})}= \sum_{A_{13},A_{23}\in \mathbb{F}_q}
                           {\psi_{A_{13}e_{13}+A_{23}e_{23}+A_{16}^*e_{16}}},\\
& \Psi_{M(A_{17}^*e_{17})}= \sum_{A_{12}\in \mathbb{F}_q}
                           {\psi_{A_{12}e_{12}+A_{17}^*e_{17}}}.
 \end{alignat*}
\end{Corollary}

\begin{Proposition}[Supercharacter theory for $G_2^{syl}(q)$]\label{supercharacter theory-G2}
$(\mathcal{X},\mathcal{K})$ is a supercharacter theory for Sylow $p$-subgroup $G_2^{syl}(q)$,
where $\mathcal{K}$ is defined in \ref{superclass:ci ti-G2},
and $\mathcal{X}$ is defined in \ref{set of supercharacters-G2}.
\index{supercharacter theory!-for $G_2^{syl}(q)$}
\nomenclature{$(\mathcal{X},\mathcal{K})$}{a supercharacter theory for $G_2^{syl}(q)$ \nomrefpage}%
\end{Proposition}
\begin{proof}
 By \ref{superclass:ci ti-G2}, $\mathcal{K}$ is a partition of $U$.
We know that $\mathcal{X}$ is a set of nonzero complex characters of $U$.
\begin{itemize}
\setlength\itemsep{0em}
 \item [(a)] {\it Claim that $|\mathcal{X}|=|\mathcal{K}|$}.
By \ref{superclass:ci ti-G2}, \ref{notation:supermodules-G2} and \ref{set of supercharacters-G2},
$|\{\Psi_{M{(A_{17}^*e_{17})}}  \mid   A_{17}^*\in \mathbb{F}_q^* \}|
=|\{{M{(A_{17}^*e_{17})}}  \mid   A_{17}^*\in \mathbb{F}_q^* \}|
=|\{C_6(t_6^*) \mid t_6^* \in \mathbb{F}_q^* \}|$.
Similarly, we obtain $|\mathcal{X}|=|\mathcal{K}|$.
 \item [(b)] {\it Claim that the characters $\chi \in \mathcal{X}$ are
              constant on the members of $\mathcal{K}$}.
Let $A \in \mathfrak{F}_4$ and
 \begin{align*}
 \mathcal{B}_{15}(A_{15}^*):=
 \left\{
 {\newcommand{\mc}[3]{\multicolumn{#1}{#2}{#3}}
\begin{array}{cccccccc}\cline{2-7}
\mc{1}{c|}{}
& \mc{1}{c|}{C_{12}}
& \mc{1}{c|}{C_{13}}
& \mc{1}{c|}{{{A_{15}^*}}}
& \mc{1}{c|}{{A_{15}^*}} & \mc{1}{c|}{×} & \mc{1}{c|}{×} & \\\cline{2-7}
× & \mc{1}{c|}{} & \mc{1}{c|}{C_{23}} & × & × & × &  & ×\\\cline{3-3}
\end{array}
}
{\, }\middle|{\, }
C_{12},C_{13},C_{23}\in \mathbb{F}_{q}
\right\}.
\end{align*}
If $y\in U$, then
 \begin{align*}
 \Psi_{M({A_{15}^*}(e_{14}+e_{15}))}(y)=
     \sum_{\substack{C\in \mathcal{B}_{15}(A_{15}^*)\\C.y=C}}{\chi_C(y)}
 =\sum_{\substack{C\in \mathcal{B}_{15}(A_{15}^*)\\ y\in \mathrm{Stab}_U(C)}}{\chi_C(y)}.
 \end{align*}
 If $y=y(0,0,0,t_4,t_5,t_6)\in C_0\cup C_4(t_4^*)\cup C_5(t_5^*)\cup C_6(t_6^*) \subseteq \mathcal{K}$,
 then
 $y\in \mathrm{Stab}_U(C)$ for all $C\in \mathcal{B}_{15}(A_{15}^*)$ by \ref{prop: G2-stab}.
 Thus
\begin{align*}
\Psi_{M({A_{15}^*}(e_{14}+e_{15}))}(y)=&
               \sum_{C\in \mathcal{B}_{15}(A_{15}^*)}{\chi_C(y)}
            = \sum_{C\in \mathcal{B}_{15}(A_{15}^*)}
               {\vartheta(2{A_{15}^*}t_4)}
            =q^3\cdot{\vartheta(2{A_{15}^*}t_4)}.
\end{align*}
If $y\in C_1(t_1^*)\cup C_{1,2}(t_1^*,t_2^*)\cup C_3(t_3^*) \subseteq \mathcal{K}$,
then $y \notin \mathrm{Stab}_U(C)$ for all $C\in \mathcal{B}_{15}(A_{15}^*)$ by \ref{prop: G2-stab}.
Thus $
\Psi_{M({A_{15}^*}(e_{14}+e_{15}))}(y)=0$.

If $y=y(0,t_2^*,s_3,s_4,s_5,s_6)\in C_2(t_2^*) \subseteq \mathcal{K}$,
then by \ref{prop: G2-stab}
\begin{align*}
&   \Psi_{M({A_{15}^*}(e_{14}+e_{15}))}(y)
=  \sum_{\substack{C\in \mathcal{B}_{15}(A_{15}^*)
              \\C_{13}=-\frac{2A_{15}^*s_3}{t_2^*}}}{\chi_C(y)}\\
= & \sum_{C_{12},C_{23}\in \mathbb{F}_{q}}{
   \vartheta \kappa \left({\,}
{%
\newcommand{\mc}[3]{\multicolumn{#1}{#2}{#3}}
\begin{array}{c|c|cccc}\hline
\mc{1}{|c|}{C_{12}}
& -\frac{2A_{15}^*s_3}{t_2^*}
& \mc{1}{c|}{{A_{15}^*}} & \mc{1}{c|}{{A_{15}^*}} & \mc{1}{c|}{×} & \mc{1}{c|}{×}\\\hline
× & C_{23} & × & × & × & ×\\\cline{2-2}
\end{array}
}%
,{\,}
{%
\newcommand{\mc}[3]{\multicolumn{#1}{#2}{#3}}
\begin{array}{c|c|cccc}\hline
\mc{1}{|c|}{0} & -s_3 & \mc{1}{c|}{s_4}
& \mc{1}{c|}{s_4} & \mc{1}{c|}{*} & \mc{1}{c|}{*}\\\hline
× & t_2^* & × & × & × & ×\\\cline{2-2}
\end{array}
}%
{\,}\right)}\\
= & \sum_{C_{12},C_{23}\in \mathbb{F}_{q}}{
   \vartheta (C_{23}t_2^*+\frac{2A_{15}^*s_3^2}{t_2^*}
          +2A_{15}^*s_4)}
=  q\cdot {\vartheta(\frac{2A_{15}^*s_3^2}{t_2^*}+2{A_{15}^*}s_4)}
     \cdot \sum_{C_{23}\in \mathbb{F}_{q}}{\vartheta(C_{23}t_2^*)}
= 0.
\end{align*}
Similarly,  we calculate the other values of the Table \ref{table:supercharacter table-G2}.
Thus the claim is proved.
 \item [(c)] The elements of $\mathcal{X}$ are pairwise orthogonal by \ref{G2-orth.}.
 \item [(d)] The set $\{I_8\}$ is a member of $\mathcal{K}$.
\end{itemize}
By
\ref{supercharacter theory},
 $(\mathcal{X},\mathcal{K})$ is a supercharacter theory for $G_2^{syl}(q)$.
\end{proof}

\begin{table}[!htp]
\caption{Supercharacter table of $G_2^{syl}(q)$ for $p>2$}
\label{table:supercharacter table-G2}
{\tiny
\begin{align*}
\renewcommand\arraystretch{1.5}
\begin{array}{l|cccccccc}
×
& C_0
& C_1(t_1^*)
& C_2(t_2^*)
& C_{1,2}(t_1^*,t_2^*)
& C_3(t_3^*)
& C_4(t_4^*)
& C_5(t_5^*)
& C_6(t_6^*)\\
\hline
\Psi_{M{(0)}}
& 1 & 1 & 1 & 1 & 1 & 1 & 1 & 1\\
\Psi_{M{(A_{12}^*e_{12})}}
& 1
& \vartheta (A_{12}^*t_1^*)
& 1
& \vartheta (A_{12}^*t_1^*)
& 1 & 1 & 1 & 1 \\
\Psi_{M{(A_{23}^*e_{23})}}
& 1
& 1
& \vartheta (A_{23}^*t_2^*)
& \vartheta (A_{23}^*t_2^*)
& 1 & 1 & 1 & 1 \\
\Psi_{M{(A_{12}^*e_{12}+A_{23}^*e_{23})}}
& 1
& \vartheta (A_{12}^*t_1^*)
& \vartheta (A_{23}^*t_2^*)
& \begin{array}{l}
\vartheta (A_{12}^*t_1^*)\\
\cdot \vartheta (A_{23}^*t_2^*)
\end{array}
& 1 & 1 & 1 & 1
\\
\Psi_{M{(A_{13}^*e_{13})}}
& q
& 0 & 0 & 0
& \begin{array}{l}
\vartheta(-A_{13}^*t_{3}^*)\\
  \cdot q
  \end{array}
& q & q & q
\\
\Psi_{M{({A_{15}^*}(e_{14}+e_{15}))}}
& q^3
& 0 & 0 & 0
& 0
& \begin{array}{l}
\vartheta(2{A_{15}^*}t_4^*)\\\
\cdot  q^3
  \end{array}
& q^3 & q^3
\\
 \Psi_{M{(A_{16}^*e_{16})}}
& q^{4}
& 0 & 0 & 0
& 0
& 0
& \begin{array}{l}
  \vartheta(A_{16}^*t_{5}^*)\\
  \cdot q^{4}
  \end{array}
& q^{4}
\\
\Psi_{M{(A_{17}^*e_{17})}}
& q^{4}
& 0 & 0 & 0
& 0 & 0 & 0
& \begin{array}{l}
  \vartheta(A_{17}^*t_{6}^*)\\
  \cdot q^{4}
  \end{array}
\end{array}
\end{align*}
}

\end{table}

\begin{Corollary}
The number of the supercharacters $G_2^{syl}(q)$ is
$|\mathcal{X}|=|\mathcal{M}|=|\mathcal{K}|
=q^2+4q-4
=(q-1)^2+6(q-1)+1$.
\end{Corollary}

\begin{Definition}
 Let $A$ be a staircase pattern.
 Then  the \textbf{verge module} of $A$
 is the right $\mathbb{C}U$-module
 $\mathbb{C}\mathcal{V}(A)=\mathbb{C}\text{-span}\{[B] \mid B\in V,{\,} \mathrm{verge}(B)=\mathrm{verge}(A)\}$,
 and the \textbf{first verge module} of $A$ is
 the right $\mathbb{C}U$-module
 $\mathbb{C}\mathcal{V}_1(A)=\mathbb{C}\text{-span}\{[B] \mid B\in V,{\,} \mathrm{verge}_1(B)=\mathrm{verge}_1(A)\}
 \supseteq \mathbb{C}\mathcal{V}(A)$.
\end{Definition}

\begin{Comparison}[Supercharacters]
\label{com:supercharacter theory-G2}
Every supercharacter of the families $\mathfrak{F}_{1,2}$, $\mathfrak{F}_3$ and $\mathfrak{F}_6$
for $G_2^{syl}(q)$
is afforded by the verge module of some
staircase pattern,
and every supercharacter of the families $\mathfrak{F}_{4}$ and $\mathfrak{F}_5$
for $G_2^{syl}(q)$
is afforded by the first verge module of some staircase pattern
(see \ref{notation:supermodules-G2} and \ref{supercharacter theory-G2}).
These also hold for the supercharacters of ${^3}D_4^{syl}(q^3)$
except the supercharacters of the family $\mathfrak{F}_{3}$
(see \cite[8.3 and 8.7]{sun3D4super})
\end{Comparison}

\section{Conjugacy classes}
\label{sec:conjugacy classes-G2}
In this section, we determine the conjugacy classes of $G_2^{syl}(q)$
(see \ref{prop:conjugacy classes-G2,p not 3}),
and establish the relations between the superclasses and the conjugacy classes of $G_2^{syl}(q)$
(see \ref{relations super- and conj. classes-G2}).
Let $U:=G_2^{syl}(q)$,
$\mathrm{char}{\,\mathbb{F}_q}=p>3$,
and $t_i\in \mathbb{F}_q$,
$t_i^*\in \mathbb{F}_q^*$ $(i=1,2,\dots,6)$.

If $y,u\in U$, then the conjugate of $x$ by $u$ is
${^u}y:=uyu^{-1}$,
and the conjugacy class of $u$ is
${^U}y:=\left\{vyv^{-1} {\,}\middle|{\,}  v\in U \right\}$.
By the commutator relations, we obtain the following conjugate elements.
\begin{Lemma}\label{prop:conjugacy classes of x_i-G2,p not 3}
Let $\mathrm{char}{\,\mathbb{F}_q}=p>3$, $u:=y(r_1,r_2,r_3,r_4,r_5,r_6)\in U$ and $y_i(t_i)\in U$.
Then
\begin{align*}
{^u}y_6(t_6)=&y_6(t_6),\qquad
{^u}y_5(t_5)=y_5(t_5)\cdot y_6(r_2t_5),\\
{^u}y_4(t_4)=& y_4(t_4)\cdot y_5(3r_1t_4)
                      \cdot y_6(3r_1r_2t_4+3r_3t_4),\\
{^u}y_3(t_3)=&y_3(t_3)\cdot y_4(2r_1t_3)
                      \cdot y_5(3r_1^{2}t_3)
                      \cdot y_6(3r_1^{2}r_2t_3
                               -3r_1t_3^{2}
                               -3t_3r_4),\\
{^u}y_2(t_2)=&y_2(t_2)\cdot y_3(-r_1t_2)
                      \cdot y_4(-t_2r_1^{2})
                      \cdot y_5(-t_2r_1^{3})
                      \cdot y_6(-t_2r_5-t_2^2r_1^{3}-t_2r_1^{3}r_2),\\
{^u}y_1(t_1)=&y_1(t_1)\cdot y_3(r_2t_1)
                      \cdot y_4(-r_2t_1^{2}-2t_1r_3)
                      \cdot y_5(r_2t_1^{3}
                                 -6r_1r_3t_1
                                 +3r_3t_1^{2}
                                 -3t_1r_4
                                 )\\
                      & \cdot y_6(2r_2^2t_1^{3}
                                  -6r_1r_2r_3t_1
                                  +3r_2r_3t_1^{2}
                                  -3r_2r_4t_1
                                  -3t_1r_3^{2} ),
\end{align*}
and
\begin{align*}
{^u}\big(y_3(t_3)y_5(t_5)\big)
                     =&y_3(t_3)\cdot y_4(2r_1t_3)
                      \cdot y_5(t_5+3r_1^{2}t_3)
                      \cdot y_6(r_2t_5+3r_1^{2}r_2t_3
                               -3r_1t_3^{2}
                               -3t_3r_4),\\
{^u}\big(y_2(t_2)y_4(t_4)y_5(t_5)\big)
                      = &y_2(t_2)\cdot y_3(-r_1t_2)
                      \cdot y_4(t_4-t_2r_1^{2})
                      \cdot y_5(t_5-t_2r_1^{3}+3r_1t_4)\\
                      & \cdot y_6(-t_2r_5-t_2^2r_1^{3}-t_2r_1^{3}r_2
                          +3r_1r_2t_4
                          +3r_3t_4
                          +r_2t_5),\\
{^u}\big(y_2(t_2)y_1(t_1)\big)
 =&y_2(t_2)y_1(t_1)\cdot y_3(r_2t_1-r_1t_2)
                      \cdot y_4(-r_2t_1^{2}-2t_1r_3
                                 -t_2r_1^{2}+2t_1t_2r_1)\\
                      & \cdot y_5(r_2t_1^{3}
                                 -6r_1r_3t_1
                                 +3r_3t_1^{2}
                                 -3t_1r_4
                                 -t_2r_1^{3}
                                -3r_1t_1^{2}t_2
                                +3t_1t_2r_1^{2} )\\
                      & \cdot y_6(2r_2^2t_1^{3}
                                  -6r_1r_2r_3t_1
                                  +3r_2r_3t_1^{2}
                                  -3r_2r_4t_1
                                  -3t_1r_3^{2} \\
                      &\phantom{\cdot y_6(}
                                  -t_2r_5-t_2^2r_1^{3}-t_2r_1^{3}r_2
                              -6r_1r_2t_1^{2}t_2
                              +3t_1t_2r_1^{2}r_2
                              +3r_1^{2}t_1t_2^2 ).
\end{align*}

\end{Lemma}

\begin{Proposition}[Conjugacy classes of $G_2^{syl}(q)$]\label{prop:conjugacy classes-G2,p not 3}
If $\mathrm{char}{\,\mathbb{F}_q}=p>3$,
then the conjugacy classes of $G_2^{syl}(q)$ are
listed in Table \ref{table:conjugacy classes-G2,p not 3}.
\begin{table}[!htp]
\caption{Conjugacy classes of $G_2^{syl}(q)$ for $p>3$}
\label{table:conjugacy classes-G2,p not 3}
\begin{align*}
\begin{array}{|l|l|c|}\hline
\multicolumn{1}{|c|}{
\rule{0pt}{13pt}
\begin{array}{c}
\text{Representatives } y \in U \\
\end{array}
}
& \multicolumn{1}{c|}{\text{Conjugacy Classes } {^Uy}}
& |{^Uy}|
\\\hline
I_8  & y(0,0,0,0,0,0) & 1\\\hline
y_6(t_6^*), {\ }
t_6^*\in \mathbb{F}_q^*
& y(0,0,0,0,0,t_6^*) & 1\\\hline
y_5(t_5^*), {\ }
t_5^*\in \mathbb{F}_q^*
&
y(0,0,0,0,t_5^*,s_6), {\ }
 s_6\in \mathbb{F}_q
& q\\
\hline
y_4(t_4^*), {\ }
t_4^*\in \mathbb{F}_q^*
& y(0,0,0,t_4^*,s_5,s_6), {\ }
 s_5,s_6\in \mathbb{F}_q
& q^2\\
\hline
y(0,0,t_3^*,0,t_5,0),{\ }
t_3^*\in \mathbb{F}_{q}^*, {\, }
t_5\in \mathbb{F}_q
&
y(0,0,t_3^*,s_4,\hat{s}_5,s_6), {\ }
s_4, s_6\in \mathbb{F}_q
& q^2\\
\hline
y(0, t_2^*, 0, t_4, t_5, 0), {\ }
t_2^*\in \mathbb{F}_{q}^*, {\, }
t_4,t_5\in \mathbb{F}_q
&
y(0,t_2^*,s_3,\hat{s}_4,\hat{s}_5,s_6), {\ }
s_3, s_6\in \mathbb{F}_q
&  q^2 \\
\hline
y(t_1^*,0,0,0,0,t_6), {\ }
t_1^*\in \mathbb{F}_{q}^*, {\, }
t_6\in \mathbb{F}_q
& y(t_1^*,0,s_3,s_4,s_5,\hat{s}_6),{\ }
s_3,s_4,s_5\in \mathbb{F}_q
& q^3 \\\hline
y(t_1^*, t_2^*,0,0,0,0), {\ }
t_1^*,t_2^*\in \mathbb{F}_{q}^*
&
y(t_1^*,t_2^*,s_3,s_4,s_5,s_6), {\ }
s_3,s_4, s_5,s_6\in \mathbb{F}_q
& q^4\\\hline
\end{array}
\end{align*}
where $\hat{s}_{-}$ is determined by some of $t_{-}^*$, $t_{-}$ and $s_{-}$.
\end{table}
\end{Proposition}

\begin{proof}
Let
$u:=y(r_1,r_2,r_3,r_4,r_5,r_6)\in U$,
$0\neq t_1\in \mathbb{F}_{q}^*$, $t_6\in  \mathbb{F}_{q}$,
and
$y(a_1,a_2,a_3,a_4,a_5,a_6):={^u}\big(y_1(t_1)y_6(t_6)\big)$.
Then by \ref{prop:conjugacy classes of x_i-G2,p not 3},
$a_1= t_1$,
$a_2= 0$,
$a_3= r_2t_1$,
$a_4= -r_2t_1^{2}-2r_3t_1$,
$a_5= r_2t_1^{3}
       -6r_1r_3t_1
       +3r_3t_1^{2}
       -3t_1r_4$,
$a_6= t_6+r_2^2t_1^{3}+r_2a_5-3r_3^2t_1$.
If $a_3$, $a_4$ and $a_5$ are fixed,
then $a_6$ is determined uniquely.
Hence the conjugacy classes of $y_1(t_1)y_6(t_6)$ is
\begin{align*}
{^U}\big(y_1(t_1)y_6(t_6)\big)=
 \left\{y(t_1,0,s_3,s_4,s_5,\hat{s}_6) \mid s_3,s_4,s_5\in\mathbb{F}_{q} \right\}.
\end{align*}
By \ref{prop:conjugacy classes of x_i-G2,p not 3}, the other conjugacy classes are determined analogously.
\end{proof}

\begin{Corollary}[Superclasses and conjugacy classes]\label{relations super- and conj. classes-G2}
Let $t_i\in \mathbb{F}_q$, $t_i^*\in \mathbb{F}_q*$ $(i=1,2,\dots, 6)$.
Then the relations between the superclasses and the conjugacy classes are determined.
 \begin{align*}
  C_6(t_6^*)=& {^U}y_6(t_6^*),\quad
    C_5(t_5^*)= {^U}y_5(t_5^*),\quad
    C_4(t_4^*)= {^U}y_4(t_4^*),\\
  C_3(t_3^*)=& \bigcup_{t_5\in \mathbb{F}_q}^{.}{{^U}\big(y_3(t_3^*)y_5(t_5)\big)},\quad
     C_2(t_2^*)= \bigcup_{t_4,t_5\in \mathbb{F}_q}^{.}{{^U}\big(y_2(t_2^*)y_4(t_4)y_5(t_5)\big)},\\
  C_1(t_1^*)=&  \bigcup_{t_6\in \mathbb{F}_q}^{.}{{^U}\big(y_1(t_1^*)y_6(t_6)\big)},\quad
     C_{1,2}(t_1^*,t_2^*)= {^U}\big(y_2(t_2^*)y_1(t_1^*)\big),\quad
     C_0 = \{1_U\}=\{ 1 \}.
 \end{align*}
Note that the superclasses $ C_1(t_1^*)$, $C_2(t_2^*)$
and $ C_3(t_3^*)$  are not conjugacy classes,
but the other superclasses are conjugacy classes.
\end{Corollary}
\begin{Comparison}[Conjugacy classes]
\label{com:conjugacy classes-G2}
The classification of conjugacy classes of $G_2^{syl}(q)$
is similar to that of ${{^3D}_4^{syl}}(q^3)$ (see \cite[\S3]{sun1}).
\end{Comparison}


\section{Irreducible characters}
\label{sec:irreducible characters-G2}
In this section,
we construct irreducible characters of $G_2^{syl}(q)$
(see \ref{construction of irr. char. of G2})
by Clifford's Theorem (see \cite{CR1}),
and determine the character table of $G_2^{syl}(q)$
in Table \ref{table:character table-G2,p not 3}.

Let $G$ be a finite group,
$N$ a normal subgroup of $G$,
and
$K$ a field.
Let $\mathrm{Irr}(G)$ be the set of all complex irreducible characters of $G$,
and $\mathrm{triv}_{G}$  the trivial character of $G$.
If $H$ is a subgroup of $G$,
$\chi\in \mathrm{Irr}(G)$ and $\lambda \in \mathrm{Irr}(H)$,
then
we denote by $\mathrm{Ind}_H^G{\lambda}$ the character induced from $\lambda$,
and denote by $\mathrm{Res}^G_H{\chi}$ the restriction of $\chi$ to $H$.
The center of $G$ is denoted by $Z(G)$.
The kernel of $\chi$ is
              $\ker{\chi}=\{g\in G \mid \chi(g)=\chi(1)\}$.
The commutator subgroup of $G$ is
             $G'=\left<{\,} [x,y] \mid  x,y\in G {\,}\right>$,
           where $[x,y]=x^{-1}y^{-1}xy$.
If $\lambda \in \mathrm{Irr}(N)$,
then the inertia group in $G$ is
$I_G(\lambda)=\{g\in G \mid  \lambda^g=\lambda \}$
where $\lambda^g(n)=\lambda(gng^{-1})$ for all $n\in N$.
In particular, $N\unlhd I_G(\lambda) \leqslant G$.
Let $\mathrm{char}{\,\mathbb{F}_q}=p>3$,
$U:=G_2^{syl}(q)$,
$t_i\in \mathbb{F}_q$,
$t_i^*\in \mathbb{F}_q^*$
$(i=1,2,\dots,6)$,
and $A_{ij}\in \mathbb{F}_q$,
$A_{ij}^*\in \mathbb{F}_q^*$ $(1\leq i, j \leq 8)$.

 Let $\vartheta \colon \mathbb{F}_q^+ \to  \mathbb{C}^*$ denote a fixed nontrivial
 linear character of the additive group $\mathbb{F}_q^+$
 of $\mathbb{F}_q$
 once and for all.
 In particular, $\sum_{x\in \mathbb{F}_q^+}{\vartheta(x)}=0$.
Let $b\in \mathbb{F}_{q}$ and
$\vartheta_b \colon  \mathbb{F}_q^+ \to  \mathbb{C}^*:y\mapsto \vartheta(by)$.
Then
$\mathrm{Irr}(\mathbb{F}_{q}^+)= \{\vartheta_b \mid  b\in \mathbb{F}_{q}\}$.
 Let $G$ be a finite group, $Z(G)\subseteq N\trianglelefteq G$,
 and $\chi \in \mathrm{Irr}(G)$.
 Let $\lambda\in \mathrm{Irr}(N)$
 such that $\langle \mathrm{Res}^G_N\chi,  \lambda\rangle_N=e >0$.
 Then
$\left(\mathrm{Res}^G_N\chi \right)(g)=e \frac{|G|}{|I_G(\lambda)|}\lambda(g)$
 for all $g\in Z(G)$,
and $g\notin \ker \chi \iff  g \notin \ker \lambda$.
In particular,  if $X \leqslant Z(G)$, then  $X \nsubseteq \ker\chi$
if and only if $X \nsubseteq \ker\lambda$.

\begin{Lemma}
If $Y_i\leqslant U$, then
$
  Z(U)= Y_6,\
  Z({Y_6}\backslash U)= \bar{Y}_5,\
  Z({Y_5Y_6}\backslash U)= \bar{Y}_4,\
  Z({Y_4Y_5Y_6}\backslash U)= \bar{Y}_3$,
and ${Y_4Y_4Y_5Y_6}\backslash U$ is abelian.
\end{Lemma}
\begin{proof}
By the commutator relations,
we get the centers of the groups.
\end{proof}

\begin{Lemma}\label{some inertia groups N-G2}
Let $T:=Y_2Y_3Y_4Y_5Y_6$,
$N:=Y_4Y_5Y_6$,
and $H:=Y_1Y_4Y_5Y_6$.
 \begin{itemize}
 \item[(1)]
The subgroup $N$ is abelian,
$N\trianglelefteq U$, $T\trianglelefteq U$ and $H\leqslant U$.
\item[(2)] Let $\lambda \in \mathrm{Irr}(N)$ and $\mathrm{Res}^N_{Y_6}\lambda \neq \mathrm{triv}_{Y_6}$.
If $\lambda$ satisfies that $\mathrm{Res}^N_{Y_5}\lambda = \mathrm{triv}_{Y_5}$, then
$I_{U}(\lambda)=\{u\in U \mid \lambda^u=\lambda\}=H$.
\item[(3)]
If $\lambda \in \mathrm{Irr}(N)$,
then the inertia group is
$
 I_{T}(\lambda)=
 {\left\{
 \begin{array}{ll}
  T & \text{if }\mathrm{Res}^N_{Y_6}{\lambda} =\mathrm{triv}_{Y_6}\\
  N & \text{if }\mathrm{Res}^N_{Y_6}{\lambda} \neq \mathrm{triv}_{Y_6}
 \end{array}
 \right.}$.
\item[(4)]  If $\lambda \in \mathrm{Irr}(N)$,
then the inertia group is
$
 I_{H}(\lambda)=
 {\left\{
 \begin{array}{ll}
  H & \text{if }\mathrm{Res}^N_{Y_5}{\lambda} =\mathrm{triv}_{Y_5}\\
  N & \text{if }\mathrm{Res}^N_{Y_5}{\lambda} \neq \mathrm{triv}_{Y_5}
 \end{array}
 \right.}$.
\item [(5)]
If $\psi \in \mathrm{Irr}(T)$ and $Y_6=Z(T)\nsubseteq \ker\psi$,
then the inertia group is
$
 I_{U}(\psi)=\{u\in U \mid \psi^u=\psi\}=U$.
\end{itemize}
\end{Lemma}
We determine the irreducible characters of the abelian group $N:=Y_4Y_5Y_6$.
\begin{Lemma}\label{irr. char. X456-3D4}
 Let  $A_{17}, A_{16},A_{15}\in \mathbb{F}_q$and
$       \lambda^{A_{17}, A_{16}, A_{15}}(y_4(t_4)y_5(t_5)y_6(t_6))
          :=\vartheta(A_{17}t_6)
          \cdot \vartheta(A_{16}t_5)
          \cdot \vartheta(2A_{15}t_4)$.
Then
$\mathrm{Irr}(N)
 = \left\{ \lambda^{A_{17}, A_{16}, A_{15}}
      {\,}\middle|{\,} A_{17}, A_{16}, A_{15}\in \mathbb{F}_q\right\}$.
\end{Lemma}

Now we determine the irreducible characters of the subgroup
$H=Y_1Y_4Y_5Y_6$ of $U$.
\begin{Lemma}\label{irr. char. X1456-G2}
Let $H=Y_1Y_4Y_5Y_6$ and $\tilde{\chi}\in \mathrm{Irr}(H)$.
\begin{itemize}
 \item [(1)] If $Y_5\subseteq \ker\tilde{\chi}$,
then set
$\bar{H}_{146}:={{Y_5}\backslash H} \cong \bar{Y}_1\bar{Y}_4\bar{Y}_6$,
$\bar{\chi}^{A_{17}, A_{15}, A_{12}}\in \mathrm{Irr}(\bar{H}_{146})$,
\begin{align*}
\bar{\chi}^{A_{17}, A_{15}, A_{12}}(\bar{y}_1(t_1)\bar{y}_4(t_4)\bar{y}_6(t_6))
          :=\vartheta(A_{17}t_6)
          \cdot \vartheta(2A_{15}t_4)
          \cdot \vartheta(A_{12}t_1),
\end{align*}
 and $\tilde{\chi}^{A_{17}, A_{15}, A_{12}}$ be the lift of
 $\bar{\chi}^{A_{17}, A_{15}, A_{12}}$ to $H$.
 Thus
 \begin{align*}
 \mathrm{Irr}(H)_1:=& \{\tilde{\chi}\in \mathrm{Irr}(H)  \mid  Y_5\subseteq \ker\tilde{\chi}\}
 = \{\tilde{\chi}^{A_{17}, A_{15}, A_{12}}\in \mathrm{Irr}(H)
       \mid  A_{17},A_{15},A_{12}\in \mathbb{F}_{q}\}.
 \end{align*}
 \item [(2)] If $Y_5\nsubseteq \ker\tilde{\chi}$,
then
$
 \mathrm{Irr}(H)_2:= \{\tilde{\chi}\in \mathrm{Irr}(H)  \mid  Y_5\nsubseteq \ker\tilde{\chi}\}
 = \{\mathrm{Ind}_N^H\lambda^{A_{17}, A_{16}^*, 0}
       \mid  A_{17}\in \mathbb{F}_q, A_{16}^*\in \mathbb{F}_q^*\}$.
\end{itemize}
Thus, $\mathrm{Irr}(H)=\mathrm{Irr}(H)_1 \dot{\cup} \mathrm{Irr}(H)_2$,
i.e. $H$ has $q^3$ linear characters
and $(q-1)q$ irreducible characters of degree $q$.
Let $y:=y(t_1,0,0,t_4,t_5,t_6)\in H=Y_1Y_4Y_5Y_6$ be a representative of one conjugacy class of $H$.
Then the character table of $H$ is shown in Table \ref{table:character table-H-G2}.
\begin{table}[!htp]
\caption{Character table of $H=Y_1Y_4Y_5Y_6$}
\label{table:character table-H-G2}
{
\begin{center}
\begin{tabular}{c|ccc}
$|^Hy|$
& $1$
& $q$
& $q$\\
\begin{tabular}{c}
$y$
\end{tabular}
&  $y_5(t_5)y_6(t_6)$
&  $y_4(t_4^*) y_6(t_6)$
& $y_1(t_1^*)y_4(t_4) y_6(t_6)$
\\\hline
$\tilde{\chi}^{A_{17}, A_{15}, A_{12}}$
& $\vartheta(A_{17}t_6)$
& $\vartheta(A_{17}t_6+2A_{15}t_4^*)$
& $\vartheta(A_{17}t_6+2A_{15}t_4+A_{12}t_1^*)$
\\
$\mathrm{Ind}_N^H\lambda^{A_{17}, A_{16}^*, 0}$
&$q\cdot \vartheta(A_{17}t_6+A_{16}^*t_5)$
& $0$
& $0$\\
\end{tabular}
\end{center}
}%
\end{table}
\end{Lemma}
We obtain the irreducible characters of the
normal subgroup $T=Y_2Y_3Y_4Y_5Y_6$ of $U$.
\begin{Lemma}\label{irr. char. X23456-G2}
If $T=Y_2Y_3Y_4Y_5Y_6$ and ${\psi}\in \mathrm{Irr}(T)$,
then $T'=Y_6$.
\begin{itemize}
 \item [(1)] If $Y_6\subseteq \ker{\psi}$,
let  $\bar{H}_{2345}:={{Y_6}\backslash T} \cong \bar{Y}_2\bar{Y}_3\bar{Y}_4\bar{Y}_5$,
     $\bar{\chi}^{A_{16}, A_{15},A_{13}, A_{23}}\in \mathrm{Irr}(\bar{H}_{2345})$,
     \begin{align*}
       & \bar{\chi}^{A_{16}, A_{15},A_{13}, A_{23}}
         (\bar{y}_2(t_2)\bar{y}_3(t_3)\bar{y}_4(t_4)\bar{y}_5(t_5))
           :=\vartheta(A_{16}t_5)
          \cdot \vartheta(2A_{15}t_4)
          \cdot \vartheta(-A_{13}t_3)
          \cdot \vartheta(A_{23}t_2),
      \end{align*}
 and ${\psi}^{A_{16}, A_{15},A_{13}, A_{23}}$ be the lift of
 $\bar{\chi}^{A_{16}, A_{15},A_{13}, A_{23}}$ to $T$.
Thus
\begin{align*}
 \mathrm{Irr}(T)_1:=& \{{\psi}\in \mathrm{Irr}(T)  \mid  Y_6\subseteq \ker{\psi}\}
 = \{{\psi}^{A_{16}, A_{15},A_{13}, A_{23}}
       \mid  A_{16},A_{15},A_{13}, A_{23}\in \mathbb{F}_{q}\}.
 \end{align*}
 \item [(2)] If $Y_6\nsubseteq \ker{\psi}$,
then
$\mathrm{Irr}(T)_2:= \{\psi\in \mathrm{Irr}(T)  \mid  Y_6\nsubseteq \ker\psi\}
 = \{\mathrm{Ind}_N^T\lambda^{A_{17}^*, 0, 0}
       \mid  A_{17}\in \mathbb{F}_q^* \}$.
\end{itemize}
Thus, $\mathrm{Irr}(T)=\mathrm{Irr}(T)_1 \dot{\cup} \mathrm{Irr}(T)_2$, i.e. $T$ has $q^4$ linear characters
and $(q-1)$ irreducible characters of degree $q^2$.
 If $y:=y(0,t_2,t_3,t_4,t_5,t_6)\in T=Y_2Y_3Y_4Y_5Y_6$ is a representative of one conjugacy class of $T$,
 then the character table of $T$ is the one in Table \ref{table:character table-T-G2}.
\begin{table}[!htp]
\caption{Character table of $T=Y_2Y_3Y_4Y_5Y_6$}
\label{table:character table-T-G2}
{\small
\begin{center}
\begin{tabular}{c|ccccc}
$|^Ty|$
& $1$
& $q$
& $q$
& $q$
& $q$\\
\begin{tabular}{c}
$y$
\end{tabular}
&  $y_6(t_6)$
&  $y_5(t_5^*)$
&  $y_4(t_4^*) y_5(t_5)$
&  $y_3(t_3^*)y_4(t_4) y_5(t_5)$
&  \begin{tabular}{l}
    $y_2(t_2^*)y_3(t_3)$\\
    $\cdot y_4(t_4) y_5(t_5)$
   \end{tabular}
\\\hline
${\psi}^{A_{16}, A_{15},A_{13}, A_{23}}$
& $1$
& $\vartheta(A_{16}t_5^*)$
& \begin{tabular}{l}
     $\vartheta(A_{16}t_5)$\\
     $\cdot \vartheta(2A_{15}t_4^*)$\\
   \end{tabular}
& \begin{tabular}{l}
     $\vartheta(A_{16}t_5)$\\
     $\cdot \vartheta(2A_{15}t_4)$\\
     $\cdot \vartheta(-A_{13}t_3^*)$\\
   \end{tabular}
& \begin{tabular}{l}
     $\vartheta(A_{16}t_5)$\\
     $\cdot \vartheta(2A_{15}t_4)$\\
     $\cdot \vartheta(-A_{13}t_3)$\\
     $\cdot \vartheta(A_{23}t_2^*)$
   \end{tabular}
\\
$\psi^{A_{17}^*}$
& $q^2\cdot \vartheta(A_{17}^*t_6)$
& $0$
& $0$
& $0$
& $0$\\
\end{tabular}
\end{center}
}%
\end{table}
\end{Lemma}

Now we give the constructions of the irreducible characters of $G_2^{syl}(q)$.
\begin{Proposition}\label{construction of irr. char. of G2}
Let $U=G_2^{syl}(q)$, $\mathrm{char}{\,\mathbb{F}_q}=p>3$,
and
$A_{ij}\in \mathbb{F}_q$,
$A_{ij}^*\in \mathbb{F}_q^*$ $(1\leq i, j \leq 8)$.
 \begin{itemize}
  \item [(1)]
  Let
$\bar{U}:={Y_3Y_4Y_5Y_6}\backslash U=\bar{Y}_2\bar{Y}_1$,
$\bar{\chi}_{lin}^{A_{12},A_{23}}\in \mathrm{Irr}(\bar{U})$,
$\bar{\chi}_{lin}^{A_{12},A_{23}}(\bar{y}_2(t_2)\bar{y}_1(t_1))
         :=\vartheta(A_{12}t_1)\cdot \vartheta(A_{23}t_2)$,
and  $\chi_{lin}^{A_{12},A_{23}}$
  be the lift of $\bar{\chi}_{lin}^{A_{12},A_{23}}$ to $U$.
Then
\begin{align*}
\mathfrak{F}_{lin}
    := \{\chi\in \mathrm{Irr}(U) \mid  Y_3Y_4Y_5Y_6\subseteq \ker{\chi}\}
    = \{\chi_{lin}^{A_{12},A_{23}}  \mid A_{12}, A_{23}\in \mathbb{F}_{q}\}.
\end{align*}
  \item [(2)]
  Let
  $\bar{U}:={Y_4Y_5Y_6}\backslash U=\bar{Y}_2\bar{Y}_1\bar{Y}_3$,
  $\bar{H}:=\bar{Y}_1\bar{Y}_3$,
  $\bar{\chi}_{3,q}^{A_{13},{A}_{12}}\in \mathrm{Irr}(\bar{H})$,
  $\bar{\chi}_{3,q}^{A_{13},{A}_{12}}(\bar{y}_1(t_1)\bar{y}_3(t_3))
         :=\vartheta({A}_{12}t_1-A_{13}t_3)$,
and
  $\chi_{3,q}^{A_{13}^*}$
  be the lift of $\mathrm{Ind}_{\bar{H}}^{\bar{U}} \bar{\chi}_{3,q}^{A_{13}^*,0}$ to $U$.
Then
  \begin{align*}
   \mathfrak{F}_{3}
    :=& \{\chi\in \mathrm{Irr}(U) \mid  Y_4Y_5Y_6\subseteq \ker{\chi},{\ }Y_3 \nsubseteq \ker{\chi}\}
    = \{\chi_{3,q}^{A_{13}^*}
     \mid A_{13}^*\in \mathbb{F}_{q}^*\}.
  \end{align*}
  \item [(3)]
  Let
  $\bar{U}:={Y_5Y_6}\backslash U=\bar{Y}_2\bar{Y}_1\bar{Y}_3\bar{Y}_4$,
  $\bar{H}:=\bar{Y}_2\bar{Y}_3\bar{Y}_4$,
  $\bar{\chi}_{4,q}^{A_{15},A_{23},A_{13}}\in \mathrm{Irr}(\bar{H})$,
  \begin{align*}
  \bar{\chi}_{4,q}^{A_{15},A_{23},A_{13}}(\bar{y}_2(t_2)\bar{y}_3(t_3)\bar{y}_4(t_4))
         :=\vartheta(A_{23}t_2)\cdot \vartheta(-A_{13}t_3)
            \cdot \vartheta(2A_{15}t_4),
  \end{align*}
and $\chi_{4,q}^{A_{15}^*,A_{23}}$
be the lift of $\mathrm{Ind}_{\bar{H}}^{\bar{U}} \bar{\chi}_{4,q}^{A_{15}^*,A_{23},0}$ to $U$.
Then
  \begin{align*}
   \mathfrak{F}_{4}
    :=& \{\chi\in \mathrm{Irr}(U) \mid  Y_5Y_6\subseteq \ker{\chi},{\ }Y_4 \nsubseteq \ker{\chi}\}
    = \{\chi_{4,q}^{A_{15}^*,A_{23}}
     \mid A_{15}^*\in \mathbb{F}_{q}^*, A_{23}\in \mathbb{F}_{q}\}.
  \end{align*}
  \item [(4)]
  Let
$\bar{U}:={Y_6}\backslash U=\bar{Y}_2\bar{Y}_1\bar{Y}_3\bar{Y}_4\bar{Y}_5$,
$\bar{H}:=\bar{Y}_2\bar{Y}_3\bar{Y}_4\bar{Y}_5$,
$\bar{\chi}_{5,q}^{A_{16},A_{23},A_{13},A_{15}}\in \mathrm{Irr}(\bar{H})$,
\begin{align*}
  &\bar{\chi}_{5,q}^{A_{16},A_{23},A_{13},A_{15}}
           (\bar{y}_2(t_2)\bar{y}_3(t_3)\bar{y}_4(t_4)\bar{y}_5(t_5))
         := \vartheta(A_{23}t_2)\cdot \vartheta(-A_{13}t_3)
           \cdot \vartheta(2A_{15}t_4)
           \cdot \vartheta(A_{16}t_5),
\end{align*}
and $\chi_{5,q}^{A_{16}^*,A_{23},A_{13}}$
  be the lift of $\mathrm{Ind}_{\bar{H}}^{\bar{U}} \bar{\chi}_{5,q}^{A_{16}^*,A_{23},A_{13},0}$ to $U$.
Then
  \begin{align*}
   \mathfrak{F}_{5}
    :=& \{\chi\in \mathrm{Irr}(U) \mid  Y_6\subseteq \ker{\chi},{\ }Y_5 \nsubseteq \ker{\chi}\}
    = \{\chi_{5,q}^{A_{16}^*,A_{23},A_{13}}
     \mid A_{16}^*\in \mathbb{F}_{q}^*, A_{23}, A_{13}\in \mathbb{F}_{q}\}.
  \end{align*}
  \item [(5)]
  Let
  $ H:=Y_1Y_4Y_5Y_6$,
  $ \bar{H}:={Y_4Y_5}\backslash H \cong \bar{Y}_1\bar{Y}_6$,
  $ \bar{\chi}_{6,q^2}^{A_{17},A_{12}}\in \mathrm{Irr}(\bar{H})$,
  and
  \begin{align*}
   \bar{\chi}_{6,q^2}^{A_{17},A_{12}}(\bar{y}_1(t_1)\bar{y}_6(t_6))
    :=\vartheta(A_{12}t_1)\cdot \vartheta(A_{17}t_6).
  \end{align*}
Let $\tilde{\chi}_{6,q^2}^{A_{17},A_{12}}$
   denote the lift of $\bar{\chi}_{6,q^2}^{A_{17},A_{12}}$
    from $\bar{H}$ to $H$,
and $\chi_{6,q^2}^{A_{17}^*,A_{12}}:=\mathrm{Ind}_{H}^{U} \tilde{\chi}_{6,q^2}^{A_{17}^*,A_{12}}$.
Then
  \begin{align*}
   \mathfrak{F}_{6}
    :=& \{\chi\in \mathrm{Irr}(U) \mid  Y_6 \nsubseteq \ker{\chi}\}
    = \{ \chi_{6,q^2}^{A_{17}^*,A_{12}}
     \mid A_{17}^*\in \mathbb{F}_{q}^*, A_{12}\in \mathbb{F}_{q}\}.
  \end{align*}
 \end{itemize}
Hence $\mathrm{Irr}(U)=\mathfrak{F}_{lin}\dot{\cup} \mathfrak{F}_{3}
                        \dot{\cup} \mathfrak{F}_{4} \dot{\cup} \mathfrak{F}_{5}\dot{\cup} \mathfrak{F}_{6}$.
\end{Proposition}

 \begin{proof}
 Let $\chi \in \mathrm{Irr}(U)$.
 We prove the hard case:
  Family $\mathfrak{F}_{6}$, where $Y_6\nsubseteq \ker\chi$.
 Let $T=Y_2Y_3Y_4Y_5Y_6$,  $N=Y_4Y_5Y_6$,
 and $\chi\in \mathrm{Irr}(U)$ such that $Y_6\nsubseteq \ker(\chi)$.
 Then $Z(T)=Z(U)=Y_6$.
 If $\psi\in \mathrm{Irr}(T)$ and
 $\langle \psi, \mathrm{Res}^U_T \chi \rangle_T >0$,
  then $Y_6\nsubseteq \ker \psi$.
 Let $\lambda^{A_{17}, A_{16}, A_{15}}\in \mathrm{Irr}(N)$ and
 $\psi^{A_{17}^*}:=\mathrm{Ind}_N^T{\lambda^{A_{17}^*,0,0}}$.
 Then
 by \ref{irr. char. X23456-G2},
 we have
 $     \{\psi\in \mathrm{Irr}(T) \mid  Y_6 \nsubseteq \ker{\psi}\}
     = \{\mathrm{Ind}_N^T{\lambda^{A_{17}^*,0,0}} \mid A_{17}^*\in \mathbb{F}_{q}^*\}
     = \{\psi^{A_{17}^*} \mid A_{17}^*\in \mathbb{F}_{q}^*\}$.
 By (5) of \ref{some inertia groups N-G2}, we have $I_U(\psi^{A_{17}^*})=U$,
 so $\mathrm{Res}^{U}_T{\chi}=z^*\psi^{A_{17}^*}$ for some $z^*\in \mathbb{N}^*$.
 Thus
 \begin{align*}
  \mathfrak{F}_6
 =& \{\chi\in \mathrm{Irr}(U) \mid Y_6 \nsubseteq \ker\chi\}
 = \bigcup _{\substack{\psi\in \mathrm{Irr}(T)\\ Y_6 \nsubseteq \ker\psi}}
    \{\chi\in \mathrm{Irr}(U) \mid   \langle \chi, \mathrm{Ind}_T^U\psi \rangle_U>0 \}\\
 {=}
  & \bigcup _{A_{17}^*\in \mathbb{F}_q^*}\{\chi\in \mathrm{Irr}(U)
      \mid   \langle \chi ,
                \mathrm{Ind}_T^U{\psi^{A_{17}^*}} \rangle_U>0\}
 {=}
   \bigcup _{A_{17}^*\in \mathbb{F}_q^*}\{\chi\in \mathrm{Irr}(U)
      \mid   \langle \chi , \mathrm{Ind}_N^U{\lambda^{A_{17}^*,0,0}}\rangle_U>0\}.
 \end{align*}
If $H={Y}_1{Y}_4{Y}_5{Y}_6$,
 then $H'=Y_5$ and $Z(H)=Y_4Y_5\trianglelefteq H$.
 Let $\tilde{\chi}^{A_{17},A_{15},A_{12}}\in \mathrm{Irr}(H)$
 as in (1) of \ref{irr. char. X1456-G2}.
 For all $y_4(t_4)y_5(t_5)y_6(t_6)\in N$,
 \begin{align*}
    & \left(\mathrm{Res}^H_N \tilde{\chi}^{A_{17}^*,0,A_{12}}\right)(y_4(t_4)y_5(t_5)y_6(t_6))
  =  \tilde{\chi}^{A_{17}^*,0,A_{12}}(y_4(t_4)y_5(t_5)y_6(t_6))\\
  = & \bar{\chi}(\bar{y}_4(t_4)\bar{y}_6(t_6))
  =  \vartheta(A_{17}^*t_6)
  =  \lambda^{A_{17}^*,0,0}(y_4(t_4)y_5(t_5)y_6(t_6)).
 \end{align*}
 Thus $\mathrm{Res}^H_N {\tilde{\chi}^{A_{17}^*,0,A_{12}}}=\lambda^{A_{17}^*,0,0}$
    { for all } $A_{12}\in \mathbb{F}_{q}$.
 By (4) of \ref{some inertia groups N-G2}, we have
        $I_H(\lambda^{A_{17}^*,0,0})=H$.
Thus $\mathrm{Ind}_N^H{\lambda^{A_{17}^*,0,0}}
          =\sum_{A_{12}\in \mathbb{F}_q}{\tilde{\chi}^{A_{17}^*,0,A_{12}}} $.
 By (2) of \ref{some inertia groups N-G2},
 we get
        $I_U(\lambda^{A_{17}^*,0,0})=H$.
By Clifford's Theorem, we obtain that $\mathrm{Ind}_H^U \tilde{\chi}^{A_{17}^*,0,A_{12}}\in \mathrm{Irr}(U)$
 for all $A_{17}^*\in \mathbb{F}_q^*$.
 Thus
 \begin{align*}
  \mathfrak{F}_6
 {=}& \bigcup _{A_{17}^*\in \mathbb{F}_q^*}\{\chi\in \mathrm{Irr}(U)
      \mid   \langle \chi ,
                      \mathrm{Ind}_H^U\mathrm{Ind}_N^H{\lambda^{A_{17}^*,0,0}}\rangle_U>0\}\\
 {=}& \bigcup _{\substack{A_{17}^*\in \mathbb{F}_q^*\\A_{12}\in \mathbb{F}_{q}}}
       \{\chi\in \mathrm{Irr}(U)
      \mid   \langle \chi , \mathrm{Ind}_H^U{\tilde{\chi}^{A_{17}^*,0,A_{12}}}\rangle_U>0 \}
 {=} \{\mathrm{Ind}_H^U{\tilde{\chi}^{A_{17}^*,0,A_{12}}}
      \mid   A_{17}^*\in \mathbb{F}_q^*,A_{12}\in \mathbb{F}_{q}\}.
 \end{align*}

 For $A_{17}^*\in \mathbb{F}_q^*$ and $A_{12}\in \mathbb{F}_{q}$,
 $Y_4Y_5\subseteq \ker(\tilde{\chi}^{A_{17}^*,0,A_{12}})$ and $Y_4Y_5 \trianglelefteq H$.
 Thus $\tilde{\chi}^{A_{17}^*,0,A_{12}}$ is the lift to $H$ of some irreducible character of
  $\bar{H}:={{Y_4Y_5}\backslash H}\cong \bar{Y}_1\bar{Y}_6$.
 Let
 $ \bar{\chi}_{6,q^2}^{A_{17}^*,A_{12}}\in \mathrm{Irr}(\bar{H})$,
 $\bar{\chi}_{6,q^2}^{A_{17}^*,A_{12}}(\bar{y}_1(t_1)\bar{y}_6(t_6))
     :=\vartheta(A_{12}t_1)\cdot \vartheta(A_{17}^*t_6)$,
 and
   $\tilde{\chi}_{6,q^2}^{A_{17}^*,A_{12}}$
     denote the lift of $\bar{\chi}_{6,q^2}^{A_{17}^*,A_{12}}$
     from $\bar{H}$ to $H$.
 Then $\tilde{\chi}_{6,q^2}^{A_{17}^*,A_{12}}=\tilde{\chi}^{A_{17}^*,0,A_{12}}$.
 If ${\chi}_{6,q^2}^{A_{17}^*,A_{12}}:=\mathrm{Ind}_H^U{\tilde{\chi}_{6,q^2}^{A_{17}^*,A_{12}}}$,
 then
 $\mathfrak{F}_6
 {=} \{\mathrm{Ind}_H^U{\tilde{\chi}^{A_{17}^*,0,A_{12}}}
      \mid   A_{17}^*\in \mathbb{F}_q^*,A_{12}\in \mathbb{F}_{q}\}
 {=} \{{\chi}_{6,q^2}^{A_{17}^*,A_{12}}
         \mid A_{17}^*\in \mathbb{F}_q^*, A_{12}\in \mathbb{F}_{q}\}$.
 \end{proof}



\begin{Proposition}
\label{prop:character table-G2,p not 3}
The character table of $G_2^{syl}(q)$ $(q>3)$
is shown in Table \ref{table:character table-G2,p not 3}.
\end{Proposition}
\begin{sidewaystable}
\caption{Character table of $G_2^{syl}(q)$ for $p>3$}
\label{table:character table-G2,p not 3}
{
\footnotesize
\begin{align*}
\renewcommand\arraystretch{1.5}
\begin{array}{l|ccccccccc}
×
& I_8
& \begin{array}{c}
  y_1(t_1^*)\\
  \cdot y_6(t_6)
  \end{array}
&
\begin{array}{c}
 y_2(t_2^*)\\
 \cdot y_4(t_4)\\
 \cdot y_5(t_5)
\end{array}
&
\begin{array}{c}
 y_2(t_2^*)\\
 \cdot y_1(t_1^*)
\end{array}
&
\begin{array}{c}
y_3(t_3^*)\\
\cdot y_5(t_5)
\end{array}
& y_4(t_4^*)
& y_5(t_5^*)
& y_6(t_6^*)\\
\hline
\chi_{lin}^{0,0}
& 1
& 1
& 1 & 1 & 1 & 1 & 1 & 1\\
\chi_{lin}^{A_{12}^*,0}
& 1
& \vartheta (A_{12}^*t_1^*)
& 1
& \vartheta (A_{12}^*t_1^*)
& 1 & 1 & 1 & 1 \\
\chi_{lin}^{0,A_{23}^*}
& 1
& 1
& \vartheta (A_{23}^*t_2^*)
& \vartheta (A_{23}^*t_2^*)
& 1 & 1 & 1 & 1 \\
\chi_{lin}^{A_{12}^*,A_{23}^*}
& 1
& \vartheta (A_{12}^*t_1^*)
& \vartheta (A_{23}^*t_2^*)
& \begin{array}{l}
\vartheta (A_{12}^*t_1^*)\\
\cdot \vartheta (A_{23}^*t_2^*)
\end{array}
& 1 & 1 & 1 & 1
\\
\chi_{3,q}^{A_{13}^*}
& q
& 0
& 0
& 0
& \begin{array}{l}
  \vartheta(-A_{13}^*t_{3}^*)\\
  \cdot q
  \end{array}
& q & q & q
\\
\chi_{4,q}^{A_{15}^*,A_{23}}
& q
& 0
& \begin{smallmatrix}
    \sum_{r_1\in \mathbb{F}_q}
    \vartheta(-2A_{15}^*t_2^*r_1^2)\\
    \cdot \vartheta(2A_{15}^*t_4)
    \cdot \vartheta(A_{23}t_2^*)
  \end{smallmatrix}
& 0
& 0
& \begin{array}{l}
\vartheta(2{A_{15}^*}t_4^*)\\
\cdot  q
  \end{array}
& q & q
\\
\chi_{5,q}^{A_{16}^*,A_{23},A_{13}}
& q
& 0
& \begin{smallmatrix}
    \sum_{r_1\in \mathbb{F}_q}
    \vartheta(A_{13}t_2^*r_1\\
    -A_{16}^*t_2^*r_1^3+3A_{16}^*t_4r_1)\\
    \cdot \vartheta(A_{16}^*t_5)
    \cdot \vartheta(A_{23}t_2^*)
  \end{smallmatrix}
& 0
& \begin{smallmatrix}
    \sum_{r_1\in \mathbb{F}_q}
    \vartheta(3A_{16}^*t_3^*r_1^2)\\
    \cdot \vartheta(A_{16}^*t_5)
    \cdot \vartheta(-A_{13}t_3^*)
  \end{smallmatrix}
& 0
& \begin{array}{l}
  \vartheta(A_{16}^*t_{5}^*)\\
  \cdot q
  \end{array}
& q
\\
\chi_{6,q^2}^{A_{17}^*,A_{12}}
& q^{2}
& \begin{smallmatrix}
    \vartheta(A_{17}^*t_6)
    \cdot \vartheta(A_{12}t_1^*)\\
    \cdot \sum_{r_3\in \mathbb{F}_q}
    \vartheta(-3A_{17}t_1^*r_3^2)
  \end{smallmatrix}
& 0 & 0
& 0 & 0 & 0
& \begin{array}{l}
  \vartheta(A_{17}^*t_{6}^*)\\
  \cdot q^{2}
  \end{array}
\end{array}
\end{align*}
}
where
 the elements of the 1st column (i.e. the row headers)
 are the complete pairwise orthogonal
  irreducible characters of $G_2^{syl}(q)$
  (see Proposition \ref{construction of irr. char. of G2}).
 The entries of the 1st row (i.e. the column headers)
 are all of the representatives of conjugacy classes of $G_2^{syl}(q)$
 (see Proposition \ref{prop:conjugacy classes-G2,p not 3}).
\end{sidewaystable}

 \begin{proof}
 Let $y:=y(t_1,t_2,t_3,t_4,t_5,t_6)\in U=G_2^{syl}(q)$.
  We shall determine the values of $\chi_{6,q^2}^{A_{17}^*,A_{12}}$
  for all $A_{17}^*\in \mathbb{F}_q^*$ and $A_{12}\in \mathbb{F}_{q}$.
 We use the notations of (5) of Proposition \ref{construction of irr. char. of G2},
 then
 \begin{align*}
 \chi_{6,q^2}^{A_{17}^*,A_{12}}\left( y \right)
 =& \left(\mathrm{Ind}_H^U \tilde{\chi}_{6,q^2}^{A_{17}^*,A_{12}}\right)(y)
 = \frac{1}{|H|}
     \sum_{\substack{g\in {U}\\g\cdot y\cdot g^{-1}\in H} }
     \tilde{\chi}_{6,q^2}^{A_{17}^*,A_{12}}(g\cdot y\cdot g^{-1})\\
 =& \frac{1}{|H|}
     \sum_{\substack{g\in {U}\\g\cdot y\cdot g^{-1}\in H} }
     \bar{\chi}_{6,q^2}^{A_{17}^*,A_{12}}({Y_4Y_5}\cdot{(g y g^{-1})}).
 \end{align*}
 Thus,
 \begin{align*}
 \chi_{6,q^2}^{A_{17}^*,A_{12}}\left(y_2(t_2^*)y_4(t_4)y_5(t_5)\right)
 = \chi_{6,q^2}^{A_{17}^*,A_{12}}\left(y_2(t_2^*)y_1(t_1^*)\right)
 =\chi_{6,q^2}^{A_{17}^*,A_{12}}\left(y_3(t_3^*)y_5(t_5)\right)
 \stackrel{gyg^{-1}\notin H}{=}0,
 \end{align*}
 and
 \begin{align*}
  & \chi_{6,q^2}^{A_{17}^*,A_{12}}\left( y_4(t_4)y_5(t_5)y_6(t_6) \right)\\
 =& \frac{1}{|H|}
     \sum_{\substack{g:=y(r_1,r_2,r_3,r_4,r_5,r_6)\in {U}
             \\g\cdot  y_4(t_4)y_5(t_5)y_6(t_6)\cdot g^{-1}\in H} }
     \tilde{\chi}_{6,q^2}^{A_{17}^*,A_{12}}(g\cdot y_4(t_4)y_5(t_5)y_6(t_6)\cdot g^{-1})\\
 {=}
  & \frac{1}{|H|}
     \sum_{r_1,r_2,r_3,r_4,r_5,r_6\in \mathbb{F}_{q}}
     \tilde{\chi}_{6,q^2}^{A_{17}^*,A_{12}}
     (y_4(t_4)y_5(t_5+3r_1t_4)
      y_6(t_6+r_2t_5+3r_1r_2t_4+3r_3t_4))\\
 = & \frac{1}{|H|}
     \sum_{r_1,r_2,r_3,r_4,r_5,r_6\in \mathbb{F}_{q}}
     \bar{\chi}_{6,q^2}^{A_{17}^*,A_{12}}
     (\bar{y}_6\left(t_6+r_2t_5+3r_1r_2t_4+3r_3t_4\right)) \\
 = & \frac{1}{q}
     \sum_{r_1,r_2,r_3\in \mathbb{F}_q }
     \bar{\chi}_{6,q^2}^{A_{17}^*,A_{12}}
     (\bar{y}_6\left(t_6+r_2t_5+3r_1r_2t_4+3r_3t_4\right)).
 \end{align*}
Hence
 $\chi_{6,q^2}^{A_{17}^*,A_{12}}(I_8)=q^2$,
 $ \chi_{6,q^2}^{A_{17}^*,A_{12}}(y_4(t_4^*))
  =\chi_{6,q^2}^{A_{17}^*,A_{12}}(y_5(t_5^*))
  =0$,
 $ \chi_{6,q^2}^{A_{17}^*,A_{12}}(y_6(t_6^*))
   =q^2\cdot \vartheta(A_{17}^*t_6^*)$,
 and
 \begin{align*}
  &\chi_{6,q^2}^{A_{17}^*,A_{12}}\left( y_1(t_1^*)y_6(t_6) \right)
 = \frac{1}{|H|}
     \sum_{\substack{g:=y(r_1,r_2,r_3,r_4,r_5,r_6)\in {U}\\g\cdot y_1(t_1^*)y_6(t_6)\cdot g^{-1}\in H} }
     \tilde{\chi}_{6,q^2}^{A_{17}^*,A_{12}}(g\cdot y_1(t_1^*)y_6(t_6)\cdot g^{-1})\\
 {=}
  & \frac{1}{|H|}
     \sum_{\substack{r_2=0\\r_1,r_3,r_4,r_5,r_6\in \mathbb{F}_q} }
     \bar{\chi}_{6,q^2}^{A_{17}^*,A_{12}}
     (\bar{y}_1(t_1^*)
      \bar{y}_6(t_6-3t_1^*r_3^2))
 =  \sum_{r_3\in \mathbb{F}_{q}}
     \bar{\chi}_{6,q^2}^{A_{17}^*,A_{12}}
     (\bar{y}_1(t_1^*)
      \bar{y}_6(t_6-3t_1^*r_3^2))\\
 = & \vartheta(A_{12}t_1^*+A_{17}^*t_6)
    \cdot \sum_{r_3\in \mathbb{F}_{q}}{\vartheta
    (-3A_{17}^*t_1^*r_3^2)}.
 \end{align*}
 Thus we get all of the values of $\chi_{6,q^2}^{A_{17}^*,A_{12}}$.
 All the other values are determined by similar calculation.
 \end{proof}

\begin{Proposition}[Supercharacters and irreducible characters]\label{relation: superchar. and irr.-G2}
The following relations between supercharacters and irreducible characters of $G_2^{syl}(q)$
are obtained.
\begin{align*}
 &\Psi_{M{(A_{17}^*e_{17})}}= q \sum_{A_{12} \in \mathbb{F}_{q}}{\chi_{6,q^2}^{A_{17}^*,A_{12}}},
\qquad
\Psi_{M{(A_{16}^*e_{16})}}= q \sum_{A_{13},A_{23}\in \mathbb{F}_{q}}
                             {\chi_{5,q}^{A_{16}^*,A_{23},A_{13}}},\\
&\Psi_{M{({A_{15}^*}(e_{14}+e_{15}))}}= q \sum_{A_{23}\in \mathbb{F}_{q}}
                                               {\chi_{4,q}^{A_{15}^*,A_{23}}},
\quad
\Psi_{M{(A_{13}^*e_{13})}}= \chi_{3,q}^{A_{13}^*},
\quad
\Psi_{M{(A_{12}e_{12}+A_{23}e_{23})}}= \chi_{lin}^{A_{12},A_{23}}.
\end{align*}
\end{Proposition}
By Propositions
\ref{prop:conjugacy classes-G2,p not 3},
\ref{construction of irr. char. of G2}
and \ref{prop:character table-G2,p not 3},
we obtain the number of the conjugacy classes of $G_2^{syl}(q)$
and determine the numbers of the complex irreducible characters of degree $q^c$ with $c\in \mathbb{N}$
(also see \cite[Table 1]{GR2009} and \cite[Table 3]{GMR2016}).
Let
$\#\mathrm{Irr}_c$ be the number of irreducible characters of $G_2^{syl}(q)$
of dimension $q^c$ with $c\in \mathbb{N}$.
Then
$\#\mathrm{Irr}_2  = q^2-q = (q-1)^2+(q-1),{\ }
\#\mathrm{Irr}_1  = q^3-1 = (q-1)^3+3(q-1)^2+3(q-1),{\ }
\#\mathrm{Irr}_0  = q^2   = (q-1)^2+2(q-1)+1$
and
$ \# \{\text{Irreducible Characters of } G_2^{syl}(q) \}
 = \# \{\text{Conjugacy Classes of } G_2^{syl}(q)\}
 = q^3+2q^2-q-1
 = (q-1)^3+5(q-1)^2+6(q-1)+1.$
Hence,
if we consider the analogue of Higman's conjecture,
Lehrer's conjecture and Isaacs' conjecture of $A_n(q)$
for $G_2^{syl}(q)$,
then the conjectures are true for $G_2^{syl}(q)$.

\begin{Comparison}[Irreducible characters]
\label{com:character table-G2}
For $G_2^{syl}(q)$,
Goodwin, Mosch and R{\"o}hrle
\cite{GMR2016}
obtained an algorithm for the adjoint orbits
and determined the numbers of the complex irreducible characters of the fixed degrees.
Except the trivial character $\chi_{lin}^{0,0}$
and the linear characters $\{\chi_{lin}^{A_{12}^*,A_{23}^*}\mid A_{12}^*,A_{23}^*\in \mathbb{F}_q^*\}$,
Himstedt, Le and Magaard \cite[\S8.3]{HLM2016}
determined all the other irreducible characters of $G_2^{syl}(q)$
by parameterizing {\it midafis}.
We construct all of the irreducible characters for $G_2^{syl}(q)$ by Clifford theory
and calculate the values of the irreducible characters on conjugacy classes (see Table \ref{table:character table-G2,p not 3}),
which is an adaption of that for ${{^3D}_4^{syl}}(q^3)$ (see \cite[\S4]{sun1}).
\end{Comparison}



\section*{Acknowledgements}
This paper is a part of my PhD thesis \cite{sunphd} at the University of Stuttgart, Germany,
so I am deeply grateful to my supervisor Richard Dipper.
I also would like to thank  Jun Hu, Markus Jedlitschky
and Mathias Werth
for the helpful discussions and valuable suggestions.


\bibliographystyle{abbrv}
\bibliography{bibliographyG2supercharacterandirr}

\end{document}